\documentclass[final]{amsart}
\usepackage[margin=1.5in]{geometry}
\usepackage{xcolor}
\usepackage[notcite,notref]{showkeys}
\usepackage{xcolor}
\newtheorem{theorem}{Theorem}[section]
\newtheorem{lemma}[theorem]{Lemma}
\newtheorem{proposition}{Proposition}[section]

\theoremstyle{definition}

\theoremstyle{remark}

\usepackage{cite}
\usepackage{color}

\numberwithin{equation}{section}

\begin{document}
\title[The relativistic quantum Boltzmann equation near equilibrium]{The relativistic quantum Boltzmann equation near equilibrium}

\author{Gi-Chan Bae}
\address{Department of mathematics, Sungkyunkwan University, Suwon 16419, Republic of Korea }
\email{gcbae02@skku.edu}

\author{Jin Woo Jang}
\address{University of Bonn, Institute for Applied Mathematics, Endenicher Allee 60, 53115 Bonn, Germany}
\email{jangjinw@iam.uni-bonn.de}

\author{Seok-Bae Yun}
\address{Department of mathematics, Sungkyunkwan University, Suwon 16419, Republic of Korea }
\email{sbyun01@skku.edu}

%    General info
%\subjclass[2010]{76P05,46N55,82C40,35F25   }

%\date{January 1, 2001 and, in revised form, June 22, 2001.}

\keywords{Relativistic quantum Boltzmann equation, Relativistic Uehling-Uhlenbeck equation, Bose-Einstein distribution, Fermi-Dirac distribution, Nonlinear energy methods}
\allowdisplaybreaks
\begin{abstract}
The relativistic quantum Boltzmann equation (or the relativistic Uehling-Uhlenbeck equation) describes the dynamics of single-species fast-moving quantum particles.
With the recent development of the relativistic quantum mechanics, the relativistic quantum Boltzmann equation has been widely used in physics and engineering such as in the quantum collision experiments and the simulations of electrons in graphene.
In spite of such importance, there has been no mathematical theory on the existence of solutions for the relativistic quantum Boltzmann equation to the best of authors' knowledge. 
In this paper, we prove the global existence of a unique classical solution to the relativistic Boltzmann equation for both bosons and fermions when the initial distribution is nearby a global equilibrium. 

\end{abstract}

\maketitle
%\tableofcontents
\section{Introduction}
%\subsection{The relativistic quantum Boltzmann equation}
The dynamics of single-species fast-moving quantum particles is described by the relativistic quantum Boltzmann equation:
%The relativistic quantum Boltzmann equation reads
\begin{align}\label{RQBE0}
\begin{split}
	p^{\mu}\partial_{\mu}F = p^0\partial_t F + p\cdot\nabla_xF = C(F,F,F,F),
\end{split}
\end{align}
where $F(x,p,t)$ is a momentum distribution function on the phase point $(x,p)\in\mathbb{T}^3\times\mathbb{R}^3$ at time $t\in[0,\infty)$.
The relativistic quantum collision operator is given by 
\begin{multline}\label{C}
C(F_1,F_2,F_3,F_4)=\int_{\mathbb{R}^3}\frac{dq}{q^0}\int_{\mathbb{R}^3}\frac{dp'}{p'^0}\int_{\mathbb{R}^3}\frac{dq'}{q'^0}W(p,q|p',q')\cr
\times\big[F_1(p')F_2(q')(1+\tau F_3(p))(1+\tau F_4(q)) 
-(1+\tau F_1(p'))(1+\tau F_2(q'))F_3(p)F_4(q)\big],
\end{multline}
where $\tau=+1$ for bosons, and $\tau=-1$ for fermions.
When the functions $F_1$, $F_2$, $F_3$ and $F_4$ are identical, we denote $C(F,F,F,F)$ as $C(F)$. The transition rate $W(p,q|p',q')$ is defined as
\begin{align}\label{W}
	W(p,q|p',q') = \frac{c}{2}s\sigma(g,\theta)\delta^{(4)}(p^{\mu}+q^{\mu}-p'^{\mu}-q'^{\mu}),
\end{align}
where $\sigma(g,\theta)$ denotes the differential cross-section which describes the collisions between particles. The $4$-dimensional Dirac-delta function implies the conservation laws of momentum and energy. The precise definitions of $s$ and $\sigma(g,\theta)$ are given in Section \ref{sec:notations}.

\subsection{Motivations and a brief history}
By the recent advances in the relativistic quantum theory, there has been increasing interest in the relativistic quantum Boltzmann equation in various fields of physics and engineering.
Li et al. \cite{li2008recent} announced several experimental and theoretical results about the interaction between heavy-ion and neutron using the relativistic quantum Boltzmann equation.
In \cite{buss2012transport}, Buss et al. provided several experiments of quantum interaction such as pion-nucleus interactions, heavy-ion reactions, electron-nucleus collisions and neutrino-nucleus interactions using an equation called Giessen Boltzmann-Uehling-Uhlenbeck (GiBUU) transport model.
Since the speed of the electron in graphene is close to that of light, Lapitski developed the lattice Boltzmann method to simulate the stream of the electrons in graphene in \cite{MR3389279}. 
%Kim et al used the relativistic quantum Boltzmann equation to simulate the rare isotope collision experiments having nuclear force in \cite{kim2016introduction}.
We also introduce some other experiments on the relativistic quantum Boltzmann model  \cite{kim2016introduction,li1997equation,succi2002lattice}.

Despite such modern advances in physics and engineering, there have been quite few mathematical studies on the relativistic quantum Boltzmann equation. Akama  \cite{akama1970relativistic} considered several types of relativistic quantum kinetic models such as the Boltzmann, the Fokker-Planck, the Landau equations, and this was the first appearance of the relativistic quantum Boltzmann equation.
The unique determination of the relativistic quantum equilibrium satisfying $H$-theorem and conservation laws has been considered by Escobedo et al. in \cite{MR1958975,MR2145021}.
To the best of authors' knowledge, the mathematical theory on the existence of solutions to the equation has not been studied yet.

Regarding the Newtonian non-quantum Boltzmann equation, a brief list of the classic literature includes \cite{MR1857879,MR2997586,MR1313028,MR1307620,MR0258399,MR2525118,MR1379589,MR1908664,MR2000470,MR2095473,MR1942465,MR1014927,MR3040372,MR2534787,MR2227952,MR3157048,MR0135535,MR33674,harris2004introduction,huang1987statistical,MR0479206,MR2683475,spohn2012large}.

The mathematical studies on the relativistic non-quantum Boltzmann equation are also not rich.  In 1940, Lichnerowicz and Marrot then suggested the first relativistic Boltzmann equation in \cite{MR4796}.
 In \cite{MR933458}, Dudy\'{n}ski and Ekiel-Je\.{z}ewska studied the linearized relativistic Boltzmann equation and obtained the estimates for the linear term in the hard and soft potential cases. The authors also constructed the global existence of mild solution in \cite{MR1151987}.
Glassey and Strauss obtained some estimates on the derivatives of the collision map between pre-post collisional momenta including the upperbound on the average of $\partial p' /\partial p$ and $\partial q' /\partial p$ in \cite{MR1105532}. Very recently, Chapman et al. \cite{2006.02540} provided an analytic and numerical evidence that the lower bound for the Jacobian determinant for the collision map $p\mapsto p'$ or $q'$ in the \textit{center-of-momentum} frame is zero.
In the \textit{nearby-equilibrium} regime, Glassey and Strauss \cite{MR1211782} established a unique global solution in $\mathbb{T}^3$ for the hard potential case, and it was extended to the whole space case via the fourteen-moment compensating function in \cite{MR1321370}. 
For the soft potential case, the global existence and the asymptotic behavior were constructed in torus in \cite{MR2728733} and extended to the whole space in \cite{MR2911100}. 
In \cite{MR2891870}, Guo and Strain obtained the stability of the relativistic Vlasov-Maxwell-Boltzmann system via the use of two different coordinate represenatations for the post-collisional momenta in order to resolve the issue of the momentum singularity. 
 In \cite{duan2017relativistic}, Duan and Yu provided the global existence of solutions in the weighted $L^\infty$ framework.  Recently, Wang \cite{wang2018global} showed the global wellposedness of the relativistic Boltzmann equation with large amplitude initial data. In the case of Coulombic interaction, Strain and Guo \cite{MR2100057} constructed a unique global-in-time classical solution for the relativistic Landau-Maxwell system. The reduction of the relativistic collision operator using the \textit{center-of-momentum} frame can be found in \cite{MR1958975,MR2765751}.
Andr\'easson et al. \cite{MR2102321} proved the finite-time blowup of the gain term when the initial data starts with the characteristic ball. The uniform $L^1$ stability of mild solution is established in \cite{MR2982812} when the initial data is sufficiently small and it decays exponentially fast. In \cite{MR3880739}, Jang and Yun proved the regularizing estimate of the gain term. For the spatially homogeneous case, Strain and Yun obtained useful inequalities about the relativistic pre-post collisional velocity and proved the existence of solution in \cite{MR3166961}. The uniform $L^{\infty}$ bounds of the solution are established in \cite{jang2019propagation} and the $L^p$ bounds are established in \cite{Jang-Yun-Lp}.
For the Newtonian limit of the relativistic particles, we refer to \cite{MR2679588}.

The mathematical studies on the non-relativistic quantum Boltzmann equation are also not rich. We would like to mention that Nordheim in $1928$ and Uehling - Uhlenbeck in $1933$ discussed the non-relativistic quantum Boltzmann equation in \cite{kikuchi1930kinetische} and \cite{uehling1933transport}, respectively. 
 Benedetto et al. \cite{MR2301288,MR2357423} showed the rigorous validity of the quantum Boltzmann equation from the $N$-body Schr\"{o}dinger equation in the weak coupling regime. In spatially homogeneous case for the Fermi-Dirac particles, Lu classified quantum equilibrium in \cite{MR1861208}. In \cite{MR2029003}, Lu and Wennberg established the strong stability in $L^1$ and proved the convergence of the solution to the equilibrium. The weak solution to the equation is constructed in the soft potential case \cite{MR2264618,MR2433484}.
For the spatially homogeneous case, several studies on bosons can be found in \cite{MR3493188,MR3215584,MR3906275,MR1751703,MR2096049,MR2157856,MR3038680,MR3217534,MR3451497,MR2811470}. For general mathematical and physical reviews, we refer to  \cite{MR2997586,MR1313028,MR1307620,MR1898707,MR0258399,MR2525118,MR1379589,MR1908664,MR2000470,MR2095473,MR2366140,MR1942465}.

\subsection{Notations}\label{sec:notations} Before introducing our main results, we define several notations on the relativistic quantities. First of all, we remark that we normalize all physical constants to $1$ throughout the paper including the speed of light $c$. We denote the energy-momentum $4$-vector as $p^{\mu}$ and usually write it as $p^{\mu}=(p^0,p^1,p^2,p^3)$. The energy-momentum $4$-vector with the lower index is written as a product in the Minkowski metric $p_{\mu}=\eta_{\mu \nu}p^{\nu}$, where the Minkowski metric is given by $\eta_{\mu \nu} =diag(-1,1,1,1) $. 
The inner product of energy-momentum $4$-vectors $p^{\mu}$ and $q_{\mu}$ is defined via the Minkowski metric:
\begin{align*}
p^{\mu}q_{\mu} = p^{\mu}\eta_{\mu \mu}q^{\mu} = -p^0q^0 +\sum_{i=1}^3p^iq^i.
\end{align*}
The energy of a relativistic particle is given by $p^0=\sqrt{1+|p|^2}$. Then we can see that the inner product of an energy-momentum $4$-vector with itself is $p^{\mu}p_{\mu}=-1$. We note that the inner product of energy-momentum $4$-vectors is Lorentz invariant $p^{\mu}q_{\mu}=\Lambda p^{\mu}\Lambda q_{\mu}$, where $\Lambda$ is a Lorentz transform which will be defined below. 

We define the relative energy $s$ between two energy-momentum 4-vectors $p^\mu$ and $q^\mu$ as 
\begin{align}\label{s}
	s(p^{\mu},q^{\mu}) = -(p^{\mu}+q^{\mu})(p_{\mu}+q_{\mu}) = -2p^{\mu}q_{\mu}+2,
\end{align}
and the relative momentum $g$ between them as
\begin{align}\label{g}
	g(p^{\mu},q^{\mu}) = \sqrt{(p^{\mu}-q^{\mu})(p_{\mu}-q_{\mu})} = \sqrt{2(-p^{\mu}q_{\mu}-1)}.
\end{align}
Note that $s=g^2+4$. In the same sense, we define the relative momentum between $p^{\mu}$ and another momentum $p'^\mu$ and $q'^\mu$ as $\bar{g}= g(p^{\mu},p'^{\mu})$ and $\tilde{g}= g(p^{\mu},q'^{\mu})$, respectively. Also, we define $\bar{s}= s(p^{\mu},p'^{\mu})$ and $\tilde{s}= s(p^{\mu},q'^{\mu})$, so that 
$\bar{s}=\bar{g}^2+4$ and $\tilde{s}=\tilde{g}^2+4$.

\subsection{Conservation laws and the Boltzmann $H$-theorem}
	
	The conservation law of the pre-post collisional energy-momentum $4$-vectors is given by
	\begin{align}\label{conserv}
		p^{\mu}+q^{\mu}=p'^{\mu}+q'^{\mu}.
	\end{align}
	The inner product with itself gives 
	\begin{align*}
		(p^{\mu}+q^{\mu})(p_{\mu}+q_{\mu}) = (p'^{\mu}+q'^{\mu})(p'_{\mu}+q'_{\mu}).
	\end{align*}
	Using $p^{\mu}p_{\mu}=p'^{\mu}p'_{\mu}=q^{\mu}q_{\mu}=q'^{\mu}q'_{\mu}=-1$, we have $p^{\mu}q_{\mu}=p'^{\mu}q'_{\mu}$. Similarly, we can have $p^{\mu}p'_{\mu}=q^{\mu}q'_{\mu}$ and $p^{\mu}q'_{\mu}=p'^{\mu}q_{\mu}$, which implies 
	\begin{align}\label{gsymme}
		\begin{split}
			g&=g(p^{\mu},q^{\mu})=g(p'^{\mu},q'^{\mu}), \cr
			\bar{g}&= g(p^{\mu},p'^{\mu}) =g(q^{\mu},q'^{\mu}), \cr
			\tilde{g}&= g(p^{\mu},q'^{\mu}) =g(p'^{\mu},q^{\mu}).
		\end{split}
	\end{align}
	Then it turns out \cite[Proposition 2.7]{Jang-Yun-Lp} that \eqref{conserv} can further imply the \textit{Pythagorean} theorem as
	\begin{align}\label{g triangle}
		g^2=\bar{g}^2+\tilde{g}^2.
	\end{align}  
	The angle of the scattering kernel $\sigma(g,\theta)$ in \eqref{W} is given by
	\begin{align}\label{cos}
		\cos\theta = \frac{(p^{\mu}-q^{\mu})(p'_{\mu}-q'_{\mu})}{g^2}.
	\end{align}
	By \eqref{gsymme} and \eqref{g triangle}, $\cos \theta$ can also be written in the following form \cite[page 277, (A.11)]{MR635279}:
	\begin{align}\label{cosa}
		\cos\theta= 1-2\frac{\bar{g}^2}{g^2}.
	\end{align}
	
	Now we introduce the \textit{center-of-momentum} framework. The relativistic pre-post collisional momenta $p$, $q$ and $p'$, $q'$ satisfying the conservation law \eqref{conserv} can be expressed as
	\begin{equation}\label{com}\begin{split}
	p'&=\frac{p+q}{2}+\frac{g}{2}\left(w-(\gamma-1)(p+q)\frac{(p+q)\cdot w}{|p+q|^2}\right), \cr	
	q'&=\frac{p+q}{2}-\frac{g}{2}\left(w-(\gamma-1)(p+q)\frac{(p+q)\cdot w}{|p+q|^2}\right),
	\end{split}\end{equation}
	where $\gamma=(p^0+q^0)/\sqrt{s}$, and $w$ denotes $\mathbb{S}^2$ component of the unit sphere:
	\begin{align}\label{w}
		w = (\sin\theta\cos\phi,\sin\theta\sin\phi, \cos\theta),
	\end{align}
	on $\theta \in[0,\pi]$ and $\phi \in[0,2\pi]$. The pre-post collisional energy $p'^0$ and $q'^0$ is written by
	\begin{align*}
		p'^0&=\frac{p^0+q^0}{2}+\frac{g}{2\sqrt{s}}(p+q)\cdot w, \cr	
		q'^0&=\frac{p^0+q^0}{2}-\frac{g}{2\sqrt{s}}(p+q)\cdot w. 
	\end{align*}

Once we divide each side of \eqref{RQBE0} by $p^0$, then we can rewrite the relativistic quantum Boltzmann equation as follows:
\begin{align}\label{RQBE}
	\begin{split}
		\partial_tF+\hat{p}\cdot\nabla_x F&=Q(F,F,F,F), \cr
		F(x,p,0)&=F_0(x,p).
	\end{split}
\end{align}
Note that $Q(F,F,F,F)=\frac{1}{p^0}C(F)$. We generally denote the normalized momentum as $\hat{p}$ where
\begin{align*}
	\hat{p}=\frac{p}{p^0}=\frac{p}{\sqrt{1+|p|^2}}.
\end{align*}
In this framework, the relativistic quantum collision operator is reduced to the following form \cite{MR2765751}:
\begin{multline}\label{Q}
	Q(F_1,F_2,F_3,F_4)=\int_{\mathbb{R}^3}dq\int_{\mathbb{S}^2}dw ~  v_{\o}(p^{\mu},q^{\mu}) ~ \sigma(g,\theta)\\\times \big[F_1(p')F_2(q')(1+\tau F_3(p))(1+\tau F_4(q)) 
	-(1+\tau F_1(p'))(1+\tau F_2(q'))F_3(p)F_4(q)\big],
\end{multline}
where the M$\o$ller velocity $v_{\o}$ is given by 
\begin{align*}
	v_{\o}(p^{\mu},q^{\mu}) =\sqrt{\bigg|\frac{p}{p^0}-\frac{q}{q^0}\bigg|^2-\bigg|\frac{p}{p^0}\times \frac{q}{q^0}\bigg|^2}= \frac{g\sqrt{s}}{p^0q^0}.
\end{align*}
One of the properties that the relativistic quantum collision operator satisfies is the following identity
\begin{align*}
	\int_{\mathbb{R}^3}dp ~Q(F,F,F,F)\left( \begin{array}{c} 1 \cr p^{\mu} \end{array}\right) =0 ,
\end{align*}
which implies the conservation laws of the total mass, momentum and energy:
\begin{align}\label{NPE0}
	\frac{d}{dt}\int_{\mathbb{T}^3}dx\int_{\mathbb{R}^3}dp~F\left( \begin{array}{c} 1 \cr p^{\mu} \end{array}\right)  =  0.
\end{align}
%Throughout this paper, we use following notation for any function with different momentum:
%\begin{align*}
%	F=F(x,p,t), \quad F_*=F(x,q,t), \quad F'=F'(x,p',t), \quad F_*'=F'(x,q',t).
%\end{align*}
The $H$-theorem for the relativistic quantum Boltzmann equation is established in \cite{MR1958975} as
\begin{align*}
\frac{d}{dt}\int_{\mathbb{T}^3}dx\int_{\mathbb{R}^3}dp~ F\ln F -\tau^{-1} (1+\tau F)\ln(1+\tau F) \leq 0 .
\end{align*}

\subsection{Global equilibria}\label{sec:globaleq} It has been shown in \cite{MR2145021} that the global equilibria for the relativistic quantum Boltzmann equation \eqref{RQBE0} has the form of 
$$\mathcal{F}(p)=\frac{1}{e^{\nu(p)}-\tau },\text{ for } \nu(p)=ap^0+b\cdot p+c,$$ for some constant $a$, $b,$ and $c.$ In this paper, we rescale the problem and consider the situation that the macroscopic mean velocity $b=0$ is zero. Then we can define the global equilibrium $m(p)$ in the form of 
\begin{align} \label{m}
	m(p)= \frac{1}{e^{ap^0+c}-\tau },
\end{align}
where $a>0$ and $c\geq -a$ in the case of bosons and $a>0$ and $c\in\mathbb{R}$ in the case of fermions. In this paper, we only consider $a>0$ and $c>-a$ in the case of bosons in order to exclude a possible blow-up of the equilibrium. We choose and fix the constant $a$ and $c$ and consider the initial distribution $F_0$ nearby the equilibrium $m(p)$. Then we will prove that the particle distribution $F$ whose initial distribution $F_0$ is sufficiently close to the relativistic quantum global equilibrium $m(p)$ will converge to $m(p)$ in $H^N_x L^2_v$ sense for $N\ge 3$.

We also denote the non-quantum relativistic equilibrium $J(p)$ as
\begin{align} \label{J}
	J(p^0)= e^{-ap^0}.
\end{align}
% Then we remark that $m(p)\leq CJ(p^0)$ by definition.
\subsection{Spaces}
Now we define some notations on norms and inner products which are frequently used throughout this paper. The constant $C$ is generically used, whose value can be changed from line to line. Especially, when we want to indicate the dependency of $a$, we indicate it as $C_{a}$.
We define the standard $L^2$ norm as
\begin{align*}
	\|f\|_{L^2_p}=\left(\int_{\mathbb{R}^3}dp~|f(p)|^2 \right)^{\frac{1}{2}}, \quad \|f\|_{L^2_{x,p}}=\left(\int_{\mathbb{T}^3}dx\int_{\mathbb{R}^3}dp~|f(x,p)|^2 \right)^{\frac{1}{2}},
\end{align*}
and we define the weighted $L^2$ norm as\begin{align*}
	\|f\|_{\nu}=\left(\int_{\mathbb{R}^3}dp~\nu(p)|f(p)|^2 \right)^{\frac{1}{2}}, \quad \|f\|_{x,\nu}=\left(\int_{\mathbb{T}^3}dx\int_{\mathbb{R}^3}dp~ \nu(p)|f(x,p)|^2 \right)^{\frac{1}{2}}.
\end{align*}
The standard $L^2$ inner product is given by 
\begin{align*}
	\langle f,g \rangle_{L^2_p}&= \int_{\mathbb{R}^3}dp~f(p)g(p) , \quad
	\langle f,g \rangle_{L^2_{x,p}}=\int_{\mathbb{T}^3}dx\int_{\mathbb{R}^3}dp~f(x,p)g(x,p).
\end{align*}
We use the multi-index notation
\begin{align*}
	\alpha=(\alpha_0,\alpha_1,\alpha_2,\alpha_3),
\end{align*}
to simplify the differential operator:
\begin{align*}
	\partial^{\alpha}=\partial_{t}^{\alpha_0}\partial_{x_1}^{\alpha_1}\partial_{x_2}^{\alpha_2}\partial_{x_3}^{\alpha_3}.
\end{align*}
For a brevity of discussion on the frequently used function $m+\tau m^2$, we often write the variable only at the end of the function as
\begin{align*}
(m+\tau m^2)(p)=m(p)+\tau (m(p))^2.
\end{align*}
We define the higher-order energy norm as follows:
\begin{align*}
	\mathcal{E}(f(t))=\frac{1}{2}\sum_{|\alpha|\leq N}\|\partial^{\alpha}f(t)\|^2_{L^2_{x,p}} + \int_0^t\sum_{|\alpha|\leq N}\|\partial^{\alpha}f(s)\|^2_{x,\nu} ds .
\end{align*}

\subsection{Hypothesis on the collision kernel}
We assume that the differential cross section $\sigma(g,\theta)$ satisfies the hard potential assumption with an angular cut-off as in \cite{MR933458,dudynski2007relativistic}:
\begin{align}\label{sigma}
	\sigma(g,\theta) = g \sin\theta.
\end{align}This assumption is analogous to the standard hard-sphere assumption for the Newtonian Boltzmann equation; if $|p+q|$ or $|p-q|$ is sufficiently small or if $|p-q|$ is much larger than $|p^0-q^0|$, then it behaves as the Newtonian hard sphere kernel.

\subsection{Main results}
In this paper, we prove the existence of a unique global-in-time classical solution of the relativistic quantum Boltzmann equation nearby the global equilibrium. Before we state our main theorem, we would like to introduce the reformulation of the relativistic quantum Boltzmann equation via a special linearization that is relevant to the relativistic and quantum case. 

In the non-quantum case, the standard decomposition of the perturbed solution $f$ near the global equilibrium $\mu$ is $F=\mu+\sqrt{\mu}f$ for $\mu=e^{-|p|^2/2}$ or $\mu=e^{-p^0}$ in the Newtonian and the relativistic cases, respectively. 
But in the quantum case, the previous decomposition $F=\mu+\sqrt{\mu}f$ does not guarantee the non-negativity of $\langle Lf, f\rangle_{L^2_{v}}$ as in Lemma \ref{null space}. Thus inspired by the previous work related to the quantum kinetic models in \cite{MR4096124,van1982generalized,MR1773932,MR2902121}, we choose the following decomposition of $F$:  
\[F(x,p,t)= m(p)+\sqrt{m(p)+\tau m^2(p)}f(x,p,t),\]
where the global equilibrium $m(p)$ is defined as in \eqref{m}:
\begin{align*}
	m(p)= \frac{1}{e^{ap^0+c}-\tau }.
\end{align*}
%such that satisfying the assumption \eqref{assumption} with $\int dxdp ~pF_0 = 0$. 
%The well-definedness of such a relativistic quantum equilibrium is obtained in \cite{MR2145021} under suitable choices of the constants $a$ and $c$, as we introduced in Section \ref{sec:globaleq}. 
We plug the decomposition into \eqref{RQBE} and divide each side of the equation by $\sqrt{m+\tau m^2}$ to have
\begin{align}\label{pert1}
\begin{split}
	\partial_tf+\hat{p}\cdot\nabla_xf +Lf&= \Gamma(f)+T(f), \cr
	f(x,p,0) &= f_0(x,p). 
\end{split}
\end{align}
The linear term $Lf$ is given by $$Lf=\nu(p)f+K_1f-K_2f,$$ where $\nu(p)$ is the collision frequency of a relativistic quantum particle, and $K_1$ and $K_2$ are compact operators. The right-hand side of \eqref{pert1} consists of nonlinear terms $\Gamma(f)$ and $T(f)$ where $\Gamma(f)$ consists of all the second-order nonlinear terms and  $T(f)$ consists of all the third-order nonlinear terms. We can easily check that the fourth-order nonlinear terms disappear by cancellation. In the linearization process, we observe that the collision operator does not satisfy the quad-linearity $Q(k+h,f,f,f)= Q(k,f,f,f)+Q(h,f,f,f)$. For more of the detailed linearization process, see Section \ref{sec:linearization}. Now we are ready to state our main theorem.
\begin{theorem}\label{Main Theorem}
Let $N\geq 3$. Suppose that the initial data $F_0$ satisfies
\begin{align*}
\left\{\begin{array}{ll}  0 \leq F_0(x,p) \leq 1 \quad \mbox{for fermions,} \\ 0  \leq F_0(x,p)  \qquad \hspace{3mm} \mbox{for bosons,} \end{array} \right.
\end{align*}
and the global equilibrium $m(p)$ shares the same total mass, momentum and energy with the initial data:
\begin{align}\label{assumption}
	\int_{\mathbb{T}^3\times \mathbb{R}^3} dxdp~ F_0(x,p)\left( \begin{array}{c} 1 \cr p^{\mu} \end{array}\right)  = \int_{\mathbb{T}^3\times \mathbb{R}^3} dxdp~m(p)\left( \begin{array}{c} 1 \cr p^{\mu} \end{array}\right).
\end{align}
Then there exist $\delta>0$ and $C>0$ such that if $\mathcal{E}(f_0)\leq \delta$ then there exists a unique global-in-time solution of \eqref{pert1} such that
\begin{enumerate}
\item The distribution function $F(x,p,t)$  has the following bounds:
\begin{align*}
	\left\{\begin{array}{ll}  0 \leq F(x,p,t) \leq 1 \quad \mbox{for fermions,} \\ 0 \leq F(x,p,t) \qquad \hspace{3mm} \mbox{for bosons.} \end{array} \right.
\end{align*}
\item The energy norm is bounded globally in time:
\[\sup_{t\in\mathbb{R}^+}\mathcal{E}(f(t))\leq C\mathcal{E}(f_0).\]
\item There exists a uniform constant $\epsilon>0$ such that the perturbation decays exponentially:
\[\sum_{|\alpha|\leq N}\|\partial^{\alpha}f(t)\|^2_{L^2_{x,p}}\leq Ce^{-\epsilon t}.\]
\item Let $f$ and $\bar{f}$ be the solutions with the initial data $f_0$ and $\bar{f}_0$, respectively. Then there exists a positive constant $\delta>0$ such that 
\[\|f-\bar{f} \|_{L^2_{x,p}} \leq e^{-\delta t}\|f_0-\bar{f}_0 \|_{L^2_{x,p}}. \]
\end{enumerate}
\end{theorem}
To prove the main theorem, we need a coercivity estimate of the linear operator $L$. The linear operator $L$ satisfies the following dissipation property:
\begin{align*}
	\langle Lf, f\rangle_{L^2_{v}} &\geq \delta \|(I-P)f\|_{\nu}^2,
\end{align*}
for some positive $\delta>0$. The macroscopic projection $Pf$ denotes the orthonormal projection onto $L^2_p$ space with respect to the following $5$-basis:
\begin{align}\label{basis1}
	\left\{\sqrt{m+\tau m^2},p_1\sqrt{m+\tau m^2},p_2\sqrt{m+\tau m^2},p_3\sqrt{m+\tau m^2},p^0\sqrt{m+\tau m^2}\right\}.
\end{align}
The $5$-dimensional basis above constitutes the kernel of $L$. 
Then it is crucial to obtain some upper-bound estimates of the nonlinear terms $\Gamma$ and $T$ in order to construct the solution. Some parts of the nonlinear terms can be estimated by a simple change of variables and the H\"{o}lder inequality. However, regarding some second-order nonlinear terms whose integrands consist of the product of $f(p')$ (or $f(q')$) and $f(p)$, we need to take a change of variables of either $p\mapsto p'$ (or $p\mapsto q'$) and there occur some non-trivial difficulties. Different from the non-relativistic case, a uniform positive lower bound for the Jacobian for the change of variables $|\partial p' /\partial p|$ (or $|\partial q' /\partial p|$) does not exist as shown in \cite{2006.02540}, and hence we need another way to deal with this difficulty. One way is to proceed the estimates of the second-order nonlinear terms by lifting the $dq$ integral to energy-momentum four vector integral $dq^{\mu}$ imposing additional Dirac-delta function. This technique has been introduced in \cite{MR635279, Jang2016, jang2019propagation}.
 Then we reduce the integral by computing the Dirac-delta function. In this process, there appears a singularity of $1/\bar{g}$ and the exponential growth with respect to the $p$ variable from $\exp(p^0-p'^0)$. We will explain this more in detail in Section \ref{sec:singbarg}.

  Once we obtain the estimates of the linear terms and the nonlinear terms, we define an iteration scheme to construct the local-in-time solution as
\begin{align}\label{itera}
	(\partial_t +\hat{p}\cdot \nabla_x)F^{n+1}=Q(F^n,F^n,F^{n+1},F^n).
\end{align}
Regarding the case of fermions, the solution $F$ has to be bounded by $1$ from above. Therefore, we also need to prove that the function $F^{n+1}$ is bounded in the closed interval $[0,1]$ in each iteration scheme based on the induction hypothesis that $F^n$ is in $[0,1]$. Then we obtain that the collision operator $Q(F^n,F^n,F^{n+1},F^n)$ is well-defined.
On the right-hand side of \eqref{itera}, note that we place $F^{n+1}$ in the $p$ variable position of the collision operator $Q$ (i.e., at the third input of $Q(\cdot,\cdot,\cdot,\cdot)$). This is by the specific structure of the nonlinear operator $Q$ from \eqref{Q}. Once we place $F^{n+1}$ as a third input of $Q$ then we can rewrite \eqref{itera} in a more clear structure
\begin{align}\label{itera2}
\{\partial_t +\hat{p}\cdot \nabla_x-\tau G(F^n)+R(F^n)\}F^{n+1}=G(F^n), 
\end{align}where $G$ and $R$ is now defined as
\begin{align*}
G(F_1,F_2,F_4)&=\int_{\mathbb{R}^3}dq\int_{\mathbb{S}^2}dw~v_{\o}\sigma(g,\theta)F_1(p')F_2(q')(1+\tau F_4(q)), \cr
R(F_1,F_2,F_4)&=\int_{\mathbb{R}^3}dq\int_{\mathbb{S}^2}dw~v_{\o}\sigma(g,\theta)(1+\tau F_1(p'))(1+\tau F_2(q'))F_4(q).
\end{align*}
Also, we can show that the boundedness $0 \leq F^n \leq 1$ implies that $0\leq G(F^n)$ and $0 \leq R(F^n)$, which guarantees the boundedness of $F^{n+1}$. 
Then the additional linearization of $F^{n+1}=m+\sqrt{m+\tau m^2}f^{n+1}$ results in creating the dissipation term $\nu(p)f^{n+1}$ in the left-hand side of \eqref{itera2}, which would then result in the exponential decay of the perturbation $f^{n+1}$ in the $L^2$ sense. By the induction argument, we obtain the uniform boundedness of the energy locally in time.

Then the standard way of extending the local existence to the global one is to eliminate the dissipation of the linear part $L$. Similarly to the Newtonian case, we substitute $f=(I-P)f+Pf$ on each side of \eqref{pert1} where $Pf$ is defined by the orthonormal projection of \eqref{basis1}.
Then the expansion of the linear term $(\partial_t+\hat{p}\cdot\nabla_x)Pf$ yields a linear combination with respect to the $14$-basis. Thus we can achieve the following coercivity estimate:
\begin{align*}
	\sum_{|\alpha|\leq N}\langle L\partial^{\alpha}f, \partial^{\alpha}f\rangle_{L^2_{x,p}} &\geq \delta \sum_{|\alpha|\leq N}\|\partial^{\alpha}f\|_{x,\nu}^2,
\end{align*}
for some positive constant $\delta>0$. This full coercivity estimate enables the extension of the local-in-time solution to a global-in-time solution.

\subsection{Main difficulties and our strategy}\label{sec:novelty}
In this subsection, we present the difficulties that arise from the estimates of the nonlinear terms. In the relativistic quantum case, there appear new types of nonlinear terms involving both pre- and post-collisional momenta at the same time as $f(p)f(p')$ and $f(p)f(q')$. In general, this kind of terms has been expected to appear only in the non-cutoff Boltzmann theory in the non-quantum case. In the non-cutoff Boltzmann theory, it has been considered very crucial to understand the change of variables $q\mapsto p'$ or $q'$ and the cancellation lemma to understand the fractional diffusive behavior \cite{MR2784329,Jang2016}. 
In the non-relativistic case, these new nonlinear terms can be handled because $|\partial p'/\partial p|$ and $|\partial q'/\partial q|$ are bounded from below, as in \cite{MR1857879,MR2525118}. However, in the relativistic case under the \textit{center-of-momentum} frame, such a positive uniform lower-bound of the Jacobian of the collision map does not exist \cite{2006.02540}.

The remedy of the issue has been introduced in  \cite{MR635279,MR1321370,MR2728733} in the \textit{center-of-momentum} frame \eqref{com} where the linear term $K_2$ of the difficulty has been calculated. Namely, the authors lift the $dq$ integral to the energy-momentum four vector integral $dq^{\mu}$ by imposing an additional Dirac-delta function, and they take a change of variables via a suitable Lorentz transform. In this paper, we also follow a similar technique to deal with the nonlinearity that occurs in dealing with the terms $\Gamma(f)$ and $T(f)$. In this direction, we however still encounter several other additional difficulties on the nonlinear term involving $f(p)f(p')$ as below. Let us first denote the nonlinear term $\Gamma_{2,1}$ involving $f(p)f(p')$ as 
\begin{multline}\label{gamma21o}
\big| \langle \Gamma_{2,1}(f,h) ,\eta \rangle_{L^2_p} \big|
\leq \int_{\mathbb{R}^3}\frac{dp}{p^0}\int_{\mathbb{R}^3}\frac{dq}{q^0}\int_{\mathbb{R}^3}\frac{dp'}{p'^0}\int_{\mathbb{R}^3}\frac{dq'}{q'^0} s\sigma(g,\theta)  \cr
\times  \delta^{(4)}(p^{\mu}+q^{\mu}-p'^{\mu}-q'^{\mu})J(q^0)J(p'^0/2)|f(p)| |h(p')| |\eta(p)|,
\end{multline}
and the integral part with respect to the measure $dqdq'$ as 
\begin{align}\label{rela B}
B=\int_{\mathbb{R}^3}\frac{dq}{q^0}\int_{\mathbb{R}^3}\frac{dq'}{q'^0} s\sigma(g,\theta) \delta^{(4)}(p^{\mu}+q^{\mu}-p'^{\mu}-q'^{\mu})J(q^0).
\end{align}
Then we present the three main difficulties that arise regarding the nonlinear terms.
\subsubsection{The Lorentz-invariant measure $\frac{dp}{p^0}$ and the nonlinear structure of $s\sigma(g,\theta)$}
The first difficulty arises from the Lorentz-invaraint measure $\frac{dp}{p^0}$ including the fraction of the energy and the nonlinear structure of the differential cross section $s\sigma(g,\theta)$ with respect to the collision variables $p$ and $q$. A good way to remove the fraction of the energy is to lift $dq$ and $dq'$ integrals to the energy momentum $4$-vector $dq^{\mu}$ and $dq'^{\mu}$ integrals by considering extra Dirac-delta and unit step functions as
%\begin{align*}
%B&=\int_{\mathbb{R}^3}\frac{dq}{q^0}\int_{\mathbb{R}^3}\frac{dq'}{q'^0} s\sigma(g,\theta) \delta^{(4)}(p^{\mu}+q^{\mu}-p'^{\mu}-q'^{\mu})J(q^0).
%\end{align*}
\begin{multline*}
B=\int_{\mathbb{R}^4}dq^{\mu}\int_{\mathbb{R}^4}dq'^{\mu} s\sigma(g,\theta) \delta^{(4)}(p^{\mu}+q^{\mu}-p'^{\mu}-q'^{\mu}) J(q^0) u(q^0)u(q'^0)\delta(q^{\mu}q_{\mu}+1)\delta(q'^{\mu}q'_{\mu}+1).
\end{multline*}
If we simply eliminate the $dq'^{\mu}$ integral by computing the Dirac-delta function of $\delta^{(4)}(p^{\mu}+q^{\mu}-p'^{\mu}-q'^{\mu})$, the rest of the integral becomes highly complicated to deal with. Thus, motivated by de Groot et al. in \cite{MR635279}, we instead apply a symmetric change of variables $\bar{q}^{\mu}=q^{\mu}+q'^{\mu}$ and $\bar{q}'^{\mu}=q^{\mu}-q'^{\mu}$. 
Despite the nonlinear structure of the $s$, $g$ and $\cos\theta$ in \eqref{s}, \eqref{g} and \eqref{cos}, respectively, we can also represent them as the terms that depend only on the variables $p^{\mu}$, $p'^{\mu}$, $\bar{q}^{\mu}$ and $\bar{q}'^{\mu}$ (See Lemma \ref{g comp}) as follows:
\begin{align*}
	g_c^2=\bar{g}^2-\frac{1}{2}(p^{\mu}+p'^{\mu})(\bar{q}_{\mu}-p_{\mu}-p'_{\mu}), \quad s_c=g_c^2+4, \quad 	\cos\theta_c= 1-2\frac{\bar{g}^2}{g_c^2}.
\end{align*}
Then we can have $B$ in the following form 
\begin{align*}
	B&=\frac{1}{4}\int_{\mathbb{R}^4\times \mathbb{R}^4}d\Theta(\bar{q}^{\mu},\bar{q}'^{\mu}) s_c\sigma(g_c,\theta_c) \delta^{(4)}(p^{\mu}-p'^{\mu}+\bar{q}'^{\mu}) J\left(\frac{\bar{q}^0+\bar{q}'^0}{2}\right),
\end{align*}
where
\begin{align*}
	d\Theta(\bar{q}^{\mu},\bar{q}'^{\mu})&=  d\bar{q}^{\mu}d\bar{q}'^{\mu}u(\bar{q}^0)u(\bar{s}-4)\delta((\bar{q}^{\mu}\bar{q}_{\mu}+\bar{q}'^{\mu}\bar{q}'_{\mu})+4)\delta(\bar{q}^{\mu}\bar{q}'_{\mu}).
\end{align*}
By substituting $\bar{q}'^{\mu}=p'^{\mu}-p^{\mu}$, we now reduce the $4$-dimensional Dirac-delta function with the energy-momentum $4$-vector $\delta^{(4)}(p^{\mu}+q^{\mu}-p'^{\mu}-q'^{\mu}) $ and obtain a more intuitive form of $B$ as 
\begin{multline}\label{B2}
B=\frac{1}{4}\int_{\mathbb{R}^4}d\bar{q}^{\mu}u(\bar{q}^0)u(\bar{s}-4)\delta(\bar{q}^{\mu}\bar{q}_{\mu}+(p'^{\mu}-p^{\mu})(p'_{\mu}-p_{\mu})+4)\cr
\times \delta(\bar{q}^{\mu}(p'_{\mu}-p_{\mu})) s_c\sigma(g_c,\theta_c) J\left(\frac{\bar{q}^0+p'^0-p^0}{2}\right).
\end{multline}
\subsubsection{Singularity with respect to $\bar{g}$}\label{sec:singbarg}
Then another difficulty arises from the singularity in the relative momentum $\bar{g}$, which occurs from the reduction of the second Dirac-delta function. We first remark that the second Dirac-delta function of \eqref{B2} consists of an inner product between energy momentum $4$-vectors. Motivated by the explicit form of the Lorentz transform by Strain \cite{MR2728733}, we apply the Lorentz trasnform which converts $p'_{\mu}-p_{\mu}$ to $(0,0,0,\bar{g})$ (i.e. $\Lambda(p'_{\mu}-p_{\mu})=(0,0,0,\bar{g})$). Then we obtain
\begin{align*}
\delta(\bar{q}^{\mu}(p'_{\mu}-p_{\mu})) = \delta(\Lambda \bar{q}^{\mu}\Lambda(p'_{\mu}-p_{\mu})) = \delta(\bar{q}^3\bar{g}) = \frac{1}{\bar{g}}\delta(\bar{q}^3),
\end{align*}
where we used that the Lorentz transform is invariant under the inner-product and that $\delta(ax)=\frac{1}{a}\delta(x)$. We can now see that one $\bar{g}$ in the Dirac-delta function comes out as $1/\bar{g}$ and creates an additional singularity. However, motivated by \eqref{cosa}, we use the half angle formula \cite[page 277, (A.11)]{MR635279} of the $\cos\theta$ and observe that
\[\sin^2(\theta/2) = \frac{1-\cos\theta}{2}=\frac{\bar{g}^2}{g^2}.\]
Combining with the assumption of the differential cross section, we obtain
\begin{align*}
\sigma(g,\theta)=g\sin\theta = 2g\sin\frac{\theta}{2}\cos\frac{\theta}{2} = 2g\sqrt{\frac{1-\cos\theta}{2}}\cos\frac{\theta}{2} = 2\bar{g} \cos\frac{\theta}{2}.
\end{align*}
This allows us to eliminate the singularity of $\bar{g}$.

\subsubsection{Exponential growth}The last difficulty is regarding the last multiplier on the right-hand side of the equation \eqref{B2}. Since the function $J(\bar{q}^0/2)$ is contained in $d\bar{q}^0$ integral, we have an exponential decay in $p'^0$ and an exponential growth in $p^0$ from $J\left((p'^0-p^0)/2\right)$ at the same time. The decay for $p'^0$ is beneficial for the upper-bound estimate, but there is a problematic term of the exponential growth of $p^0$ even in \eqref{gamma21o}.
However, we prove that the remaining part of the right-hand side of \eqref{B2} includes the following exponential decaying factor
\begin{align*}
\exp\left(-\sqrt{\frac{(p^0+p'^0)^2}{4}-\frac{|p\times p'|^2}{\bar{g}^2}}\right).
\end{align*}
Then, by the estimates in \cite[Lemma 3.1, (iii) and (iv)]{MR1211782} of
\begin{align*}
\frac{(p^0+p'^0)^2}{4}-\frac{|p\times p'|^2}{\bar{g}^2} = |p-p'|^2\frac{\bar{g}^2+4}{4\bar{g}^2} 
\geq \max\left\{\frac{\bar{g}^2}{4}+1,\frac{1}{4}|p-p'|^2\right\},
\end{align*}
we can have an exponential decay of $\exp\left(-|p-p'|/2\right)$. Since the difference of the energy $|p'^0-p^0|$ can further be bounded by the difference of the momentum $|p'-p|$, the exponential growth can be absorbed by the exponential decay as 
\begin{align*}
J\left(\frac{p'^0-p^0}{2}\right)\exp\left(-\sqrt{\frac{(p^0+p'^0)^2}{4}-\frac{|p\times p'|^2}{\bar{g}^2}}\right) \leq e^{-\frac{1}{2}(p'^0-p^0)}e^{-\frac{1}{2}|p-p'|} \leq 1.
\end{align*}
Then we can have properly weighted $L^2$ bounds for the nonlinear terms.

\subsection{Outline of the paper}
This paper is organized as follows. 
In Section \ref{sec:linearization}, we linearize the collision operator of the relativistic quantum Boltzmann equation nearby a global equilibrium. In Section \ref{sec:nonlinear}, we establish several estimates on the linear and the nonlinear terms. Section \ref{sec:localintime} is devoted to constructing the unique local-in-time classical solution. In the last section, we prove the coercivity estimate and establish the global-in-time classical solution.

\section{Linearization of the relativistic quantum Boltzmann equation}\label{sec:linearization}
In this section, we introduce the reformulation of the equation \eqref{RQBE0} via the linearization of the relativistic quantum Boltzmann equation nearby the global equilibrium:
\begin{align*}
	m(p)= \frac{1}{e^{ap^0+c}-\tau }.
\end{align*}
\begin{proposition}\label{linearization}
If we substitute $F=m+\sqrt{m+\tau m^2}f$ in \eqref{RQBE}, then we have
\begin{align*}
	\partial_t f +  \hat{p}\cdot \nabla_x f +Lf &=\Gamma(f) +T(f)
\end{align*}
where the linear term $Lf$ is decomposed as following form:
\begin{align*}
	Lf&= \nu(p)f +K_1f - K_2f,
\end{align*}
where the collision frequency $\nu(p)$ is given by
\begin{align}\label{nu}
	\nu(p) &=\frac{1}{1+\tau m(p)}\int_{\mathbb{R}^3}dq\int_{\mathbb{S}^2}dw~ v_{\o}\sigma(g,\theta)m(q)(1+\tau m(p'))(1+\tau m(q')),
\end{align}
and the compact operator $K_1$ and $K_2$ are defined by
\begin{align}\label{Kf}
\begin{split}
	K_1f(p)&=\int_{\mathbb{R}^3}dq\int_{\mathbb{S}^2}dw~ v_{\o}\sigma(g,\theta) \sqrt{m+\tau m^2(p')}\sqrt{m+\tau m^2(q')}f(q), \cr
	K_2f(p)&=2 \int_{\mathbb{R}^3}dq\int_{\mathbb{S}^2}dw~ v_{\o}\sigma(g,\theta)\sqrt{m+\tau m^2(q)}\sqrt{m+\tau m^2(q')} f(p').
\end{split}
\end{align}
The nonlinear term $\Gamma(f)$ is represented as follows:
\begin{align*}
	\Gamma(f)&= \sum_{i=1}^6\Gamma_i(f,f), \quad T(f) = \sum_{i=1}^4T_i(f,f,f).
\end{align*}
We denote the precise definition of nonlinear terms at the end of this proof.
\end{proposition}
\begin{proof}
We substitute $F=m+\sqrt{m+\tau m^2}f$ in \eqref{RQBE} to have
\begin{align*}
	\sqrt{m+\tau m^2}\partial_t f + \sqrt{m+\tau m^2} \hat{p}\cdot \nabla_x f = Q(m+\sqrt{m+\tau m^2}f).
\end{align*}
Dividing $\sqrt{m+\tau m^2}$ on each side gives an equation for the perturbation $f$:
\begin{align}\label{pertf2}
	\partial_t f +  \hat{p}\cdot \nabla_x f  &= \frac{1}{\sqrt{m+\tau m^2}}Q(m+\sqrt{m+\tau m^2}f).
\end{align}
To decompose the right-hand side into linear and nonlinear terms, we first define the zeroth-order-in-$f$ term $Q_0$:
\begin{align*}
Q_0=Q(m,m,m,m).
\end{align*}
As we can see in \eqref{Q}, the quantum collision operator includes $(1+\tau F)$ terms. Because of these terms, we cannot have the quad-linear property of $Q$. In other words, we have $Q(k+h,f,f,f)\neq Q(k,f,f,f)+Q(h,f,f,f)$. So we define the following four first-order-in-$f$ terms $Q_1$, $Q_2$, $Q_3$ and $Q_4$:
\begin{align*}
	Q_1=Q(m+\sqrt{m+\tau m^2}f,m,m,m)-Q_0, \cr Q_2=Q(m,m+\sqrt{m+\tau m^2}f,m,m)-Q_0, \cr
	Q_3=Q(m,m,m+\sqrt{m+\tau m^2}f,m)-Q_0, \cr Q_4=Q(m,m,m,m+\sqrt{m+\tau m^2}f)-Q_0.
\end{align*}
In view of this notation, the collection of the zeroth and the first-order terms in $f$ can be written as
\begin{align*}
\frac{1}{\sqrt{m+\tau m^2}} \left( \sum_{i=1}^4 Q_i  +Q_0 \right).
\end{align*}
Thus we divide the right-hand side of \eqref{pertf2} into the collection of zeroth, first-order terms and other terms as
\begin{align*}
\frac{1}{\sqrt{m+\tau m^2}}Q(m+\sqrt{m+\tau m^2}f) = -Lf + \Gamma(f),
\end{align*}
where
\begin{align*}
Lf=-\frac{1}{\sqrt{m+\tau m^2}}\sum_{i=1}^4 Q_i-\frac{1}{\sqrt{m+\tau m^2}}Q_0,
\end{align*}
and
\begin{align*}
	\Gamma(f)=\frac{1}{\sqrt{m+\tau m^2}}Q(m+\sqrt{m+\tau m^2}f)+Lf.
\end{align*}
We first calculate the linear part $Lf$. By definition of $Q_0$, we have
\begin{multline*}
Q_0=\int_{\mathbb{R}^3}dq\int_{\mathbb{S}^2}dw~v_{\o}\sigma(g,\theta) \bigg[m(p')m(q')(1+\tau m(p))(1+\tau m(q))\cr
 -(1+\tau m(p'))(1+\tau m(q'))m(p)m(q) \bigg] .
\end{multline*}
We observe from the conservation of energy
\[p^0+q^0=p'^0+q'^0,\]
that
%\begin{align*}
%\frac{1}{e^{ap^0+c}+1}\frac{1}{e^{aq'^0+c}+1}\frac{e^{ap^0+c}}{e^{ap^0+c}+1}\frac{e^{aq^0+c}}{e^{aq^0+c}+1}	\cr
%=\frac{e^{ap^0+c}}{e^{ap^0+c}+1}\frac{e^{aq'^0+c}}{e^{aq'^0+c}+1}\frac{1}{e^{ap^0+c}+1}\frac{1}{e^{aq^0+c}+1},
%\end{align*}
\begin{align}\label{equal}
	m(p')m(q')(1+\tau m(p))(1+\tau m(q)) =(1+\tau m(p'))(1+\tau m(q')) m(p)m(q),
\end{align}
which gives
\[Q_0=0.\]
For the first-order linear term, an explicit computation gives
\begin{multline*}
	Q_1=\int_{\mathbb{R}^3}dq\int_{\mathbb{S}^2}dw~v_{\o}\sigma(g,\theta) \bigg[\sqrt{m+\tau m^2(p')}m(q')(1+\tau m(p))(1+\tau m(q)) \cr
-\tau  \sqrt{m+\tau m^2(p')}(1+\tau m(q'))m(p)m(q)\bigg]f(p'),
\end{multline*}
%We can notate
%\begin{align*}
%Q_1&=\int_{\mathbb{R}^3}dq\int_{\mathbb{S}^2}dw~v_{\o}\sigma(g,\theta) \bigg[\frac{\sqrt{m+\tau m^2(p')}}{m(p')}m(p')m(q')(1-m(p))(1-m(q)) \cr
%&\quad + \frac{\sqrt{m+\tau m^2(p')}}{1-m(p')}(1+\tau m(p'))(1+\tau m(q'))m(p)m(q)\bigg]f(p').
%\end{align*}
By \eqref{equal} on the first-order term $Q_1$, we have
\begin{multline*}
Q_1%&=\int_{\mathbb{R}^3}dq\int_{\mathbb{S}^2}dw~v_{\o}\sigma(g,\theta) \bigg[\frac{\sqrt{m+\tau m^2(p')}}{m(p')}m(p)m(q)(1+\tau m(p'))(1+\tau m(q')) \cr
%&\quad + \frac{\sqrt{m+\tau m^2(p')}}{1-m(p')}(1+\tau m(p'))(1+\tau m(q'))m(p)m(q)\bigg]f(p') \cr
=\int_{\mathbb{R}^3}dq\int_{\mathbb{S}^2}dw~v_{\o}\sigma(g,\theta) m(p)m(q)(1+\tau m(p'))(1+\tau m(q')) \cr
\quad \times \left(\frac{\sqrt{m+\tau m^2(p')}}{m(p')}-\tau \frac{\sqrt{m+\tau m^2(p')}}{1+\tau m(p')}\right)f(p'),
\end{multline*}
which is equal to 
\begin{align*}
Q_1&=\int_{\mathbb{R}^3}dq\int_{\mathbb{S}^2}dw~v_{\o}\sigma(g,\theta) m(p)m(q)(1+\tau m(p'))(1+\tau m(q'))  \frac{f(p')}{\sqrt{m+\tau m^2(p')}}.
\end{align*}
With similar computations, combining with $Q_2$,$Q_3$ and $Q_4$ yields
\begin{multline}\label{Lf}
\begin{split}
Lf= \frac{1}{\sqrt{m+\tau m^2(p)}} \int_{\mathbb{R}^3}dq\int_{\mathbb{S}^2}dw~v_{\o}\sigma(g,\theta)  m(p)m(q)(1+\tau m(p'))(1+\tau m(q')) 	\cr
\times \bigg(\frac{f(p)}{\sqrt{m+\tau m^2(p)}}+	\frac{f(q)}{\sqrt{m+\tau m^2(q)}}-\frac{f(p')}{\sqrt{m+\tau m^2(p')}}
-\frac{f(q')}{\sqrt{m+\tau m^2(q')}}\bigg).
\end{split}
\end{multline}
We can easily see that the first term of $Lf$ is equal to $\nu f$. We define the second term of $Lf$ as $K_1 f$: 
\begin{multline}\label{K1}
		K_1f=\frac{m(p)}{\sqrt{m+\tau m^2(p)}}\int_{\mathbb{R}^3}dq\int_{\mathbb{S}^2}dw~ v_{\o}\sigma(g,\theta) \frac{m(q)}{\sqrt{m+\tau m^2(q)}}\\\times(1+\tau m(p'))(1+\tau m(q'))f(q).
\end{multline}
We observe
\begin{align} \label{m/m-m^2}
	\frac{m(p)}{\sqrt{m+\tau m^2(p)}}= e^{-\frac{1}{2}(ap^0+c)}.
\end{align}
Combining with the energy conservation law gives 
\begin{align}\label{equal2}
	\frac{m(p)}{\sqrt{m+\tau m^2(p)}}\frac{m(q)}{\sqrt{m+\tau m^2(q)}}=\frac{m(p')}{\sqrt{m+\tau m^2(p')}}\frac{m(q')}{\sqrt{m+\tau m^2(q')}}.
\end{align}
Applying it to \eqref{K1}, we have 
\begin{align*}
	K_1f&=\int_{\mathbb{R}^3}dq\int_{\mathbb{S}^2}dw~ v_{\o}\sigma(g,\theta) \sqrt{m+\tau m^2(p')}\sqrt{m+\tau m^2(q')}f(q).
\end{align*}
We define the collection of the third and the fourth term of $Lf$ in \eqref{Lf} as $ - K_2f$:
\begin{multline}\label{K2feq}
K_2f=\frac{ m(p)}{\sqrt{m+\tau m^2(p)}}\int_{\mathbb{R}^3}dq\int_{\mathbb{S}^2}dw~ v_{\o}\sigma(g,\theta)m(q)\frac{1+\tau m(p')}{\sqrt{m+\tau m^2(p')}}(1+\tau m(q'))f(p') \cr
\quad+\frac{m(p)}{\sqrt{m+\tau m^2(p)}}\int_{\mathbb{R}^3}dq\int_{\mathbb{S}^2}dw~ v_{\o}\sigma(g,\theta)m(q)(1+\tau m(p'))\frac{1+\tau m(q')}{\sqrt{m+\tau m^2(q')}}f(q').
\end{multline}
We would call the first line of the $K_2f$ in \eqref{K2feq} as $K_{2,1}f$ and the second line of \eqref{K2feq} as $K_{2,2}f$.
We then write the $dw$ integral of $K_{2,2}$ in the spherical coordinate $w\mapsto(\phi,\theta)$ as in \eqref{w} as follows:
\begin{multline*}
K_{2,2}f=\frac{m(p)}{\sqrt{m+\tau m^2(p)}} \int_{\mathbb{R}^3}dq \int_{0}^{2\pi}d\phi\int_{0}^{\pi}\sin\theta ~d\theta~ v_{\o}\sigma(g,\theta)\\\times m(q)(1+\tau m(p'))\frac{1+\tau m(q')}{\sqrt{m+\tau m^2(q')}}f(q').
\end{multline*}
Then we apply the change of variables $\theta \rightarrow \pi-\theta$ and $\phi\rightarrow \pi+\phi$. The change of variables would result in the exchanged roles of $p'$ and $q'$, since $w$ in \eqref{w} changes into $-w$. Thus we have
\begin{multline*}
K_{2,2}f=\frac{m(p)}{\sqrt{m+\tau m^2(p)}} \int_{\mathbb{R}^3}dq \int_{\pi}^{3\pi}d\phi\int_{\pi}^{0}\sin\theta ~(-d\theta) \cr
\times v_{\o}\sigma(g,\pi-\theta)m(q)(1+\tau m(q'))\frac{1+\tau m(p')}{\sqrt{m+\tau m^2(p')}}f(p').
\end{multline*}
By the assumption of the differential cross section $\sigma(g,\theta)$ in \eqref{sigma}, we have $\sigma(g,\theta)=g\sin\theta =\sigma(g,\pi-\theta)$. This further gives
\begin{align*}
K_{2,2}f&=\frac{m(p)}{\sqrt{m+\tau m^2(p)}}\int_{\mathbb{R}^3}dq\int_{\mathbb{S}^2}dw ~ v_{\o}\sigma(g,\theta)m(q)(1+\tau m(q'))\frac{1+\tau m(p')}{\sqrt{m+\tau m^2(p')}}f(p').
\end{align*}
This shows that $K_{2,2}=K_{2,1}$. Thus we have
\begin{align}\label{K2}
\begin{split}
K_2f&=2 K_{2,1}f\\
&=2\frac{ m(p)}{\sqrt{m+\tau m^2(p)}}\int_{\mathbb{R}^3}dq\int_{\mathbb{S}^2}dw~ v_{\o}\sigma(g,\theta)m(q)\frac{1+\tau m(p')}{\sqrt{m+\tau m^2(p')}}(1+\tau m(q'))f(p').
\end{split}
\end{align}
Similarly, we observe that 
\begin{align} \label{1-m/m-m^2}
	\frac{1+\tau m(p)}{\sqrt{m+\tau m^2(p)}}= e^{\frac{1}{2}(ap^0+c)}.
\end{align}
Combining with \eqref{m/m-m^2}, we have
\begin{align}\label{equal3}
	\frac{m(p)}{\sqrt{m+\tau m^2(p)}}\frac{1+\tau m(p')}{\sqrt{m+\tau m^2(p')}}=\frac{1+\tau m(q)}{\sqrt{m+\tau m^2(q)}}\frac{m(q')}{\sqrt{m+\tau m^2(q')}}.
\end{align}
Substituting it in \eqref{K2}, we have 
\begin{align*}
K_2f&=2\int_{\mathbb{R}^3}dq\int_{\mathbb{S}^2}dw~ v_{\o}\sigma(g,\theta)\sqrt{m+\tau m^2(q)}\sqrt{m+\tau m^2(q')}  f(p').
\end{align*}
This completes the derivation of the linear term $\nu f$, $K_1f$ and $K_2f$.

Now we consider the nonlinear terms. Since the quantum collision operator is not quad-linear (Recall that $Q(k+h,f,f,f)\neq Q(k,f,f,f)+Q(h,f,f,f)$), the second-order nonlinear term is highly complicated to be represented. Thus, we first observe one of the second-order nonlinear terms involving $f(
p')$ and $f(q')$:
\begin{multline*}
\int_{\mathbb{R}^3}dq\int_{\mathbb{S}^2}dw ~  \frac{v_{\o}\sigma(g,\theta)}{\sqrt{m+\tau m^2(p)}}\\\times  \big[\sqrt{m+\tau m^2(p')}f(p')\sqrt{m+\tau m^2(q')}f(q')(1+\tau m(p))(1+\tau m(q)) \cr
-\sqrt{m+\tau m^2(p')}f(p')\sqrt{m+\tau m^2(q')}f(q')m(p)m(q)\big].
\end{multline*}There are 6 second-order nonlinear terms of the same kind similar to the term above where the number 6 is coming from the number of choices for choosing 2 variables among the four variables $p$, $q$, $p'$, and $q'$. We represent all of the second-order nonlinear terms ($6$ nonlinear terms) as follows: 
\begin{align}\label{Gamma}
\begin{split}
\Gamma(f,h) &= \int_{\mathbb{R}^3}dq\int_{\mathbb{S}^2}dw \frac{v_{\o}\sigma(g,\theta)}{\sqrt{m+\tau m^2(p)}} \cr
&\times \bigg[(m(p')m(q')-(1+\tau m(p'))(1+\tau m(q')))\sqrt{m+\tau m^2(p)}f(p)\sqrt{m+\tau m^2(q)}h(q)	\cr
&\quad +\tau (m(q')(1+\tau m(q))-m(q)(1+\tau m(q')))\sqrt{m+\tau m^2(p)}f(p)\sqrt{m+\tau m^2(p')}h(p')	\cr
&\quad +\tau (m(q')(1+\tau m(p))-m(p)(1+\tau m(q')))\sqrt{m+\tau m^2(q)}f(q)\sqrt{m+\tau m^2(p')}h(p')	\cr
&\quad +\tau (m(p')(1+\tau m(q))-m(q)(1+\tau m(p')))\sqrt{m+\tau m^2(p)}f(p)\sqrt{m+\tau m^2(q')}h(q')	\cr
&\quad +\tau (m(p')(1+\tau m(p))-m(p)(1+\tau m(p')))\sqrt{m+\tau m^2(q)}f(q)\sqrt{m+\tau m^2(q')}h(q')	\cr
&\quad +(1+\tau m(p))((1+\tau m(q))-m(p)m(q))\sqrt{m+\tau m^2(p')}f(p')\sqrt{m+\tau m^2(q')}h(q')\bigg]\cr
&=\Gamma_1(f,h)+\Gamma_2(f,h)+\Gamma_3(f,h)+\Gamma_4(f,h)+\Gamma_5(f,h)+\Gamma_6(f,h).
\end{split}
\end{align}
We also denote as
\begin{align*}
\Gamma(f) &= \sum_{1 \leq i \leq 6}\Gamma_i(f,f),
\end{align*}when $f=h$.
Lastly, we consider the following third-order nonlinear terms involving $f(p')$, $f(q')$ and $f(p)$:
\begin{multline*}
\int_{\mathbb{R}^3}dq\int_{\mathbb{S}^2}dw ~  v_{\o} \sigma(g,\theta)\big[\tau \sqrt{m+\tau m^2(p')}f(p')\sqrt{m+\tau m^2(q')}f(q')f(p)(1+\tau m(q)) \cr
-\sqrt{m+\tau m^2(p')}f(p')\sqrt{m+\tau m^2(q')}f(q')f(p)m(q)\big].
\end{multline*}
Similarly, we represent all of the third-order nonlinear terms ($4$ nonlinear terms) as
\iffalse
\begin{align*}
T(f,h,\eta)&= \int_{\mathbb{R}^3}dq\int_{\mathbb{S}^2}dw \frac{v_{\o}\sigma(g,\theta)}{\sqrt{m+\tau m^2(p)}} \cr
&\times \bigg[\sqrt{m+\tau m^2(p)}f(p)\sqrt{m+\tau m^2(q)}h(q)\sqrt{m+\tau m^2(p')}\eta(p') \cr
&\quad +\sqrt{m+\tau m^2(p)}f(p)\sqrt{m+\tau m^2(q)}h(q)\sqrt{m+\tau m^2(q')}\eta(q')	\cr
&\quad-\sqrt{m+\tau m^2(p)}f(p)\sqrt{m+\tau m^2(p')}h(p')\sqrt{m+\tau m^2(q')}\eta(q')	\cr
&\quad-\sqrt{m+\tau m^2(q)}f(q)\sqrt{m+\tau m^2(p')}h(p')\sqrt{m+\tau m^2(q')}\eta(q')	\bigg]	\cr
&=T_1(f,h,\eta)+T_2(f,h,\eta)+T_3(f,h,\eta)+T_4(f,h,\eta).
\end{align*}
\fi
\begin{align}\label{T}
\begin{split}
T_1(f,h,\eta)&= -\tau \int_{\mathbb{R}^3}dq\int_{\mathbb{S}^2}dw~v_{\o}\sigma(g,\theta) f(p)\sqrt{m+\tau m^2(q)}h(q)\sqrt{m+\tau m^2(p')}\eta(p'), \cr 
T_2(f,h,\eta)&= -\tau \int_{\mathbb{R}^3}dq\int_{\mathbb{S}^2}dw~v_{\o}\sigma(g,\theta) f(p)\sqrt{m+\tau m^2(q)}h(q)\sqrt{m+\tau m^2(q')}\eta(q'), \cr
T_3(f,h,\eta)&= \tau  \int_{\mathbb{R}^3}dq\int_{\mathbb{S}^2}dw~v_{\o}\sigma(g,\theta) f(p)\sqrt{m+\tau m^2(p')}h(p')\sqrt{m+\tau m^2(q')}\eta(q'), \cr
T_4(f,h,\eta)&= \tau \frac{1}{\sqrt{m+\tau m^2(p)}}\int_{\mathbb{R}^3}dq\int_{\mathbb{S}^2}dw~v_{\o}\sigma(g,\theta) \sqrt{m+\tau m^2(q)}f(q) \cr
&\quad \times \sqrt{m+\tau m^2(p')}h(p')\sqrt{m+\tau m^2(q')}\eta(q').
\end{split}
\end{align}
Similarly, we define 
\begin{align*}
T(f)&= \sum_{1 \leq i \leq 4}T_i(f,f,f).
\end{align*}
We can easily check that the fourth-order nonlinear term is cancelled. 
\end{proof}
By the linearization proposition above and substituting $F=m+\sqrt{m+\tau m^2}f$ in \eqref{RQBE}, we obtain the linearized equation for the relativistic quantum Boltzmann model \eqref{RQBE} as follows: 
\begin{align}\label{pertf}
\begin{split}
\partial_tf+\hat{p}\cdot\nabla_xf +Lf&= \Gamma(f)+T(f), \cr
f(x,p,0) &= f_0(x,p). 
\end{split}
\end{align}
where $f_0(x,p) = (F_0(x,p)-m)/\sqrt{m+\tau m^2}$. Then the conservation laws \eqref{NPE0} can be written as following form: 
\begin{align}\label{consf}
\begin{split}
\int_{\mathbb{T}^3}dx\int_{\mathbb{R}^3}dp~f(x,p,t)\sqrt{m+\tau m^2} &= \int_{\mathbb{T}^3}dx\int_{\mathbb{R}^3}dp~f_0(x,p)\sqrt{m+\tau m^2}, \cr
\int_{\mathbb{T}^3}dx\int_{\mathbb{R}^3}dp~f(x,p,t)p\sqrt{m+\tau m^2} &= \int_{\mathbb{T}^3}dx\int_{\mathbb{R}^3}dp~f_0(x,p)p\sqrt{m+\tau m^2}, \cr
\int_{\mathbb{T}^3}dx\int_{\mathbb{R}^3}dp~f(x,p,t)p^0\sqrt{m+\tau m^2} &= \int_{\mathbb{T}^3}dx\int_{\mathbb{R}^3}dp~f_0(x,p)p^0\sqrt{m+\tau m^2}.
\end{split}
\end{align}
Now we state several useful properties for the linear term $Lf$.
\begin{lemma}\label{symmetric Lf} For any smooth function $\phi$, we have
\begin{align*}
	&\int_{\mathbb{R}^3}dp\int_{\mathbb{R}^3}dq\int_{\mathbb{S}^2}dw~v_{\o}\sigma(g,\theta)m(p)m(q)(1+\tau m(p'))(1+\tau m(q')) \cr
	&\times	\bigg(\frac{f(p)}{\sqrt{m+\tau m^2(p)}}+\frac{f(q)}{\sqrt{m+\tau m^2(q)}}	-\frac{f(p')}{\sqrt{m+\tau m^2(p')}}
	-\frac{f(q')}{\sqrt{m+\tau m^2(q')}}\bigg)\phi(p) 	\cr
	&=\int_{\mathbb{R}^3}dp\int_{\mathbb{R}^3}dq\int_{\mathbb{S}^2}dw~v_{\o}\sigma(g,\theta)m(p)m(q)(1+\tau m(p'))(1+\tau m(q')) \cr
	&\times	\bigg(\frac{f(p)}{\sqrt{m+\tau m^2(p)}}+\frac{f(q)}{\sqrt{m+\tau m^2(q)}}	-\frac{f(p')}{\sqrt{m+\tau m^2(p')}}
	-\frac{f(q')}{\sqrt{m+\tau m^2(q')}}\bigg)\phi(q)	\cr
	&=\int_{\mathbb{R}^3}dp\int_{\mathbb{R}^3}dq\int_{\mathbb{S}^2}dw~v_{\o}\sigma(g,\theta)m(p)m(q)(1+\tau m(p'))(1+\tau m(q')) \cr
	&\times	\bigg(\frac{f(p)}{\sqrt{m+\tau m^2(p)}}+\frac{f(q)}{\sqrt{m+\tau m^2(q)}}	-\frac{f(p')}{\sqrt{m+\tau m^2(p')}}
	-\frac{f(q')}{\sqrt{m+\tau m^2(q')}}\bigg)(-\phi(p')) 	\cr
	&=\int_{\mathbb{R}^3}dp\int_{\mathbb{R}^3}dq\int_{\mathbb{S}^2}dw~v_{\o}\sigma(g,\theta)m(p)m(q)(1+\tau m(p'))(1+\tau m(q')) \cr
	&\times	\bigg(\frac{f(p)}{\sqrt{m+\tau m^2(p)}}+\frac{f(q)}{\sqrt{m+\tau m^2(q)}}	-\frac{f(p')}{\sqrt{m+\tau m^2(p')}}
	-\frac{f(q')}{\sqrt{m+\tau m^2(q')}}\bigg)(-\phi(q')).
\end{align*}
\end{lemma}
\begin{proof}
Note that $m(p)m(q)(1+\tau m(p'))(1+\tau m(q'))$ is invariant under the change of variables $(p,q) \leftrightarrow (p',q')$ from \eqref{equal}. Thus, by the change of variables $p \leftrightarrow q$ and $(p,q) \leftrightarrow (p',q')$ introduced in \cite{MR1379589}, we have the desired results.
\end{proof}

\begin{lemma}\label{null space} We have the following properties for the linear operator $L$: 
\begin{enumerate}
\item $L$ is a symmetric operator: $\langle Lf, g \rangle_{L^2_p}=\langle f, Lg \rangle_{L^2_p}$. 
\item $Lf=0$ if and only if $f=Pf$.
\end{enumerate}
where $Pf$ is defined as the orthonormal projection to the $L^2_p$ space which is spanned by following $5$-dimensional basis:
\begin{align*}
	\left\{\sqrt{m+\tau m^2},p_1\sqrt{m+\tau m^2},p_2\sqrt{m+\tau m^2},p_3\sqrt{m+\tau m^2},p^0\sqrt{m+\tau m^2}\right\}.
\end{align*}
\end{lemma}
\begin{proof}
(1) We take an inner product of $g$ with $Lf$ of \eqref{Lf} to have  
\begin{align*}
\int_{\mathbb{R}^3}dp~ g Lf  &= \int_{\mathbb{R}^3}dp\frac{g(p)}{\sqrt{m+\tau m^2(p)}} \int_{\mathbb{R}^3}dq\int_{\mathbb{S}^2}dw~ v_{\o}\sigma(g,\theta)  m(p)m(q)(1+\tau m(p'))(1+\tau m(q')) 	\cr
& \times \bigg(\frac{f(p)}{\sqrt{m+\tau m^2(p)}}+\frac{f(q)}{\sqrt{m+\tau m^2(q)}}	-\frac{f(p')}{\sqrt{m+\tau m^2(p')}}
-\frac{f(q')}{\sqrt{m+\tau m^2(q')}}\bigg).
\end{align*}
By Lemma \ref{symmetric Lf}, we substitute
\[\phi(p)=\frac{1}{\sqrt{m+\tau m^2(p)}}g(p) \]
and obtain
\begin{align*}
\int_{\mathbb{R}^3}dp~ gLf  &= \frac{1}{4}\int_{\mathbb{R}^3}dp\int_{\mathbb{R}^3}dq\int_{\mathbb{S}^2}dw~ v_{\o}\sigma(g,\theta) m(p)m(q)(1+\tau m(p'))(1+\tau m(q')) 	\cr
&\times \left(\frac{f(p)}{\sqrt{m+\tau m^2(p)}}+\frac{f(q)}{\sqrt{m+\tau m^2(q)}}	-\frac{f(p')}{\sqrt{m+\tau m^2(p')}}-\frac{f(q')}{\sqrt{m+\tau m^2(q')}}\right) \cr
&\times \left(\frac{g(p)}{\sqrt{m+\tau m^2(p)}}+\frac{g(q)}{\sqrt{m+\tau m^2(q)}}	-\frac{g(p')}{\sqrt{m+\tau m^2(p')}}-\frac{g(q')}{\sqrt{m+\tau m^2(q')}}\right).
\end{align*}
This proves the symmetricity of $L$. \newline
(2) Once we substitute $g=f$ in the equation above, then we have 
\begin{align*}
	\int_{\mathbb{R}^3}dp~ fLf  &= \frac{1}{4}\int_{\mathbb{R}^3}dp\int_{\mathbb{R}^3}dq\int_{\mathbb{S}^2}dw~ v_{\o}\sigma(g,\theta) m(p)m(q)(1+\tau m(p'))(1+\tau m(q')) 	\cr
	&\times \bigg(\frac{f(p)}{\sqrt{m+\tau m^2(p)}}+\frac{f(q)}{\sqrt{m+\tau m^2(q)}}	-\frac{f(p')}{\sqrt{m+\tau m^2(p')}}
	-\frac{f(q')}{\sqrt{m+\tau m^2(q')}}\bigg)^2 \cr
	& \geq 0 .
\end{align*}
Therefore $Lf =0$ implies 
\begin{align*}
	\frac{f(p)}{\sqrt{m+\tau m^2(p)}}+\frac{f(q)}{\sqrt{m+\tau m^2(q)}}=	\frac{f(p')}{\sqrt{m+\tau m^2(p')}}
	+\frac{f(q')}{\sqrt{m+\tau m^2(q')}},
\end{align*}
which is satisfied if and only if $f=Pf$. %$f(p)=\sqrt{m+\tau m^2}(1,p_i,p^0)$ for $i=1,2,3$.
%By an explicit computation, substituting $Pf$ in \eqref{Lf} gives 
Moreover, $L(Pf)=0$ gives the desired results.
\end{proof}

\section{Estimates of the linear and the nonlinear terms}\label{sec:nonlinear}
This section is devoted to proving several estimates of the linear and the nonlinear terms. Especially, we emphasize that the estimates of the nonlinear term $\Gamma_2$ involve the main difficulties as we mentioned in Section \ref{sec:novelty}. Before we move onto it, we present some useful properties that are commonly used throughout this paper.
\begin{lemma}\label{mJ} There exist positive constants $C_1>0$ and $C_2>0$ such that
\begin{align*}
	C_1J(p^0)\leq m(p) \leq C_2 J(p^0).
\end{align*}
\end{lemma}
\begin{proof} 
In the fermionic case, $C_1= 1/(1+e^c)$ and $C_2=e^{-c}$ give the desired estimates:
\begin{align*}
\frac{1}{e^c+1}\frac{1}{e^{ap^0}} \leq \frac{1}{e^{ap^0+c}+1} \leq \frac{1}{e^{ap^0+c}}.
\end{align*}
In the case of bosons, we choose $C_1=e^{-c}$ and $C_2=1/(e^c-e^{-a})$. Since $c>-a$, we have $C_2<\infty$, and we get$$
\frac{1}{e^{ap^0+c}} \leq \frac{1}{e^{ap^0+c}-1} \leq \frac{1}{e^c-e^{-a}}\frac{1}{e^{ap^0}} .
$$
\end{proof}
By this Lemma, the relativistic quantum equilibrium $m(p)$ can be treated as the non-quantum relativistic equilibrium $J(p^0)$. Now we present the estimates for collision frequency $\nu(p)$ and operator $K_1$ and $K_2$. 
\begin{lemma} There exists a positive constant $C>0$ such that 
\begin{align*}
\frac{1}{C}(p^0)^{\frac{1}{2}} \leq \nu(p) \leq C(p^0)^{\frac{1}{2}}.
\end{align*}
The integral operator $K_i(f)$ is a compact operator in $L^2_p$ for $i=1,2$.
\end{lemma}
\begin{proof}
By Lemma \ref{mJ},
\begin{align}\label{1+m}
1+\tau m(p) = \frac{e^{ap^0+c}}{e^{ap^0+c}-\tau} = \frac{e^cm(p)}{J(p^0)} \leq C.
\end{align} 
Then by \eqref{1+m}, we have the upper bound of $\nu$ as 
\begin{align*}
\nu(p) &\leq C \int_{\mathbb{R}^3}dq\int_{\mathbb{S}^2}dw~ v_{\o}\sigma(g,\theta)m(q).
\end{align*}
Similarly we have
\begin{align}\label{m-m^2}
\sqrt{m+\tau m^2(p)}= \frac{e^{\frac{1}{2}(ap^0+c)}}{e^{ap^0+c}-\tau } \leq C e^{\frac{1}{2}ap^0} = CJ(p^0/2),
\end{align}
which yields 
\begin{align*}
K_1f(p)&\leq C\int_{\mathbb{R}^3}dq\int_{\mathbb{S}^2}dw~ v_{\o}\sigma(g,\theta) J(p'^0/2)J(q'^0/2)f(q), \text{and}\\
K_2f(p)&\leq C\int_{\mathbb{R}^3}dq\int_{\mathbb{S}^2}dw~ v_{\o}\sigma(g,\theta)J(q^0/2)J(q'^0/2) f(p').
\end{align*}
We note that the upper bounds of $\nu$ and $K_i$ are identical to those of non-quantum relativistic linear terms in \cite{MR933458}. Therefore we obtain the desired results.
\end{proof}
\begin{lemma}\label{coercivity} The linear operator $L$ satisfies following dissipation property for some positive constant $\delta$:
\[	\langle Lf, f\rangle_{L^2_{v}} \geq \delta \|(I-P)f\|_{\nu}^2.\]
\end{lemma}
\begin{proof} Since $K_1$ and $K_2$ are compact and $\nu$ satisfies Lemma \ref{mJ}, we have the desired results by following the same proof of Lemma 3.4 of \cite{MR2728733}.
\end{proof}

\subsection{Estimates of the second-order nonlinear terms}
In this subsection, we establish the estimates on the second-order nonlinear terms.
\begin{lemma}\label{nonlin ff} %Let $f(x,p)$, $h(x,p)$ and $\eta(x,p) \in C^{\infty}(\mathbb{T}^3 \times \mathbb{R}^3)$. Then we have
We have
\begin{align*}
\big|\langle \Gamma(f,h) ,\eta \rangle_{L^2_p} \big| \leq  C \left(\|f\|_{L^2_p}\|h\|_{\nu}+\|f\|_{\nu}\|h\|_{L^2_p}\right)\|\eta\|_{\nu}.
\end{align*}
\end{lemma}
\begin{proof}
As we can see in \eqref{Gamma}, there are six second-order nonlinear terms. By the change of variables $p' \leftrightarrow q'$, we can see that $\Gamma_2(f,h)=\Gamma_4(f,h)$ and $\Gamma_3(f,h)=\Gamma_5(f,h)$. Thus we write the proof in three parts. Firstly, we prove the estimates of $\Gamma_1(f,h)$ and $\Gamma_6(f,h)$, since they are similar to the non-quantum relativistic case. Secondly, we present the estimate of $\Gamma_3(f,h)$. Lastly, we prove the most difficult part $\Gamma_2(f,f)$. For this we need several techniques to deal with the difficulties that arise from the relativistic integrals.\newline
{\bf (1) Estimates of $\Gamma_1$ and $\Gamma_6$:} We first consider $\Gamma_1(f,h)$ term. By \eqref{1+m}, \eqref{m-m^2}, and Lemma \ref{mJ}, we have
\begin{align}\label{gamma1}
\big| \langle \Gamma_1(f,h) ,\eta \rangle_{L^2_p} \big|&\leq  C \int_{\mathbb{R}^3}dp\int_{\mathbb{R}^3}dq\int_{\mathbb{S}^2}dw ~ v_{\o}\sigma(g,\theta) J(q^0/2)|f(p)| |h(q)| |\eta(p)|.
\end{align}
By the H\"{o}lder inequality, we have  
\begin{align*}
\big| \langle \Gamma_1(f,h) ,\eta \rangle_{L^2_p} \big|& \leq C \int_{\mathbb{R}^3}dp\int_{\mathbb{S}^2}dw \left(\int_{\mathbb{R}^3}dq~v_{\o}^2\sigma^2(g,\theta)J(q^0)\right)^{\frac{1}{2}}\left(\int_{\mathbb{R}^3}dq|h(q)|^2\right)^{\frac{1}{2}} |f(p)| |\eta(p)|.
\end{align*}
For the first $dq$ integral, we recall the definition of $v_{\o}$ and $\sigma(g,\theta)$:
\begin{align*}
v_{\o}= \frac{g\sqrt{s}}{p^0q^0}, \qquad \sigma(g,\theta)=g\sin\theta.
\end{align*}
By an explicit computation, we have 
\begin{align*}
s%&=-(p^{\mu}+q^{\mu})(p_{\mu}+q_{\mu}) \cr
&=2+2p^0q^0-2p\cdot q 
~\leq~ 2p^0q^0+2(1-\cos\theta)|p||q| 
~\leq~ 4p^0q^0,
\end{align*}
and
\begin{align*}
g=\sqrt{s-4} \leq \sqrt{s} \leq 2\sqrt{p^0q^0}.
\end{align*}
Applying above inequalities, we have 
\begin{align*}
\int_{\mathbb{R}^3}dq~v_{\o}^2\sigma^2(g,\theta)J(q^0) \leq  Cp^0.
\end{align*}
We apply the H\"{o}lder inequality again to obtain the following estimate: 
\begin{align*}
\big| \langle \Gamma_1(f,h) ,\eta \rangle_{L^2_p} \big|& \leq C \|h\|_{L^2_p}\int_{\mathbb{R}^3}dp \int_{\mathbb{S}^2}dw~ (p^0)^{\frac{1}{2}} |f(p)| |\eta(p)|  \leq C\|f\|_{\nu} \|h\|_{L^2_p} \|\eta\|_{\nu}.
\end{align*}
Now we consider $\Gamma_6(f,h)$ term. We first use the energy conservation law $p^0+q^0=p'^0+q'^0$ to have
\begin{align*}
\sqrt{m+\tau m^2(p')}\sqrt{m+\tau m^2(q')}\leq  CJ(p'^0/2)J(q'^0/2) = CJ(p^0/2)J(q^0/2),
\end{align*} 
which gives 
\begin{align}\label{gamma6}
\big| \langle \Gamma_6(f,h) ,\eta \rangle_{L^2_p} \big|&\leq C \int_{\mathbb{R}^3}dp\int_{\mathbb{R}^3}dq\int_{\mathbb{S}^2}dw~ v_{\o}\sigma(g,\theta) J(q^0/2)|f(p')| |h(q')| |\eta(p)|.
\end{align}
Applying the H\"{o}lder inequality yields 
\begin{align*}
\big| \langle \Gamma_6(f,h) ,\eta \rangle_{L^2_p} \big|&\leq  C \int_{\mathbb{R}^3}dp\int_{\mathbb{S}^2}dw\left(\int_{\mathbb{R}^3}dq ~ v_{\o}\sigma^2(g,\theta) J(q^0)\right)^{\frac{1}{2}}\cr
&\quad \times \left(\int_{\mathbb{R}^3}dq ~ v_{\o}|f(p')|^2|h(q')|^2\right)^{\frac{1}{2}}|\eta(p)|.
\end{align*}
Similarly we use $s \leq 4p^0q^0 $ and $g \leq  2(p^0)^{\frac{1}{2}}(q^0)^{\frac{1}{2}}$ to have 
\begin{align*}
\big| \langle \Gamma_6(f,h) ,\eta \rangle_{L^2_p} \big|&\leq C \int_{\mathbb{R}^3}dp\int_{\mathbb{S}^2}dw~(p^0)^{\frac{1}{2}}\left(\int_{\mathbb{R}^3}dq~  v_{\o}|f(p')|^2|h(q')|^2\right)^{\frac{1}{2}}|\eta(p)|.
\end{align*}
By the H\"{o}lder inequality for the $dpdw$ integral, we have
\begin{align*}
\big| \langle \Gamma_6(f,h) ,\eta \rangle_{L^2_p} \big|	&\leq C\|\eta\|_{\nu} \left(\int_{\mathbb{R}^3}dp\int_{\mathbb{S}^2}dw \ (p^0)^{\frac{1}{2}}\int_{\mathbb{R}^3}dq v_{\o}|f(p')|^2|h(q')|^2\right)^{\frac{1}{2}}.
\end{align*}
The energy conservation law implies
\begin{align*}
(p^0)^{1/2}\leq (p'^0+q'^0)^{1/2} \leq (p'^0)^{1/2}+(q'^0)^{1/2},
\end{align*}
which yields 
\begin{align*}
\big| \langle \Gamma_6(f,h) ,\eta \rangle_{L^2_p} \big|	&\leq \|\eta\|_{\nu} \int_{\mathbb{S}^2}dw\int_{\mathbb{R}^3}dp \int_{\mathbb{R}^3}dq~ v_{\o}\left((p'^0)^{\frac{1}{2}}+(q'^0)^{\frac{1}{2}}\right) |f(p')|^2|h(q')|^2 .
\end{align*}
Then we consider the pre-post change of variables $(p,q) \leftrightarrow (p',q')$ with the Jacobian in \cite{MR1105532}:
\begin{align}\label{pq,p'q'}
\frac{dpdq}{p^0q^0} =\frac{dp'dq'}{p'^0q'^0},
\end{align}
and recall from \eqref{gsymme} that we already have
\begin{align*}
s(p^{\mu},q^{\mu}) = s(p'^{\mu},q'^{\mu}), \qquad g(p^{\mu},q^{\mu}) = g(p'^{\mu},q'^{\mu}).
\end{align*}
Thus we have 
\begin{multline*}
\int_{\mathbb{S}^2}dw\int_{\mathbb{R}^3}dp \int_{\mathbb{R}^3}dq~ v_{\o}\left((p'^0)^{\frac{1}{2}}+(q'^0)^{\frac{1}{2}}\right) |f(p')|^2|h(q')|^2 . \cr
= \int_{\mathbb{S}^2}dw\int_{\mathbb{R}^3}dp'\int_{\mathbb{R}^3}dq'~\frac{g(p'^{\mu},q'^{\mu})\sqrt{s(p'^{\mu},q'^{\mu})}}{p'^0q'^0}\left( (p'^0)^{\frac{1}{2}}+(q'^0)^{\frac{1}{2}}\right) |f(p')|^2|h(q')|^2.
\end{multline*}
Finally by $v_{\o}(p'^{\mu},q'^{\mu}) \leq C$ and the H\"{o}lder inequality, we obtain the desired results:
\begin{align*}
\big| \langle \Gamma_6(f,h) ,\eta \rangle_{L^2_p} \big|&\leq C \left(\|f\|_{L^2_p}\|h\|_{\nu}+\|f\|_{\nu}\|h\|_{L^2_p}\right)\|\eta\|_{\nu}.
\end{align*}
\newline
{\bf (2) Estimates of $\Gamma_3$:} %Applying the change of variables $\phi \rightarrow \pi-\phi$, and $\theta\rightarrow \pi+\theta$, which implies $p' \leftrightarrow q'$, we can see that $\Gamma_3(f,h)=\Gamma_5(f,h)$. So we only consider $\Gamma_3$ term. 
We split the $\Gamma_3$ term into two parts: 
\begin{align*}
\Gamma_3(f,h) &= \frac{\tau }{\sqrt{m+\tau m^2(p)}}\int_{\mathbb{R}^3}dq\int_{\mathbb{S}^2}dw~ v_{\o}\sigma(g,\theta) \{m(q')(1+\tau m(p))-m(p)(1+\tau m(q'))\} \cr
& \quad \times \sqrt{m+\tau m^2(q)}f(q)\sqrt{m+\tau m^2(p')}h(p') \cr
&= \Gamma_{3,1}(f,h)+ \Gamma_{3,2}(f,h).
\end{align*}
By $1+\tau m(p)\leq C$ and $1+\tau m(q')\le C$, we obtain
\begin{align*}
|\Gamma_{3,1}(f,h)| &\leq C\int_{\mathbb{R}^3}dq\int_{\mathbb{S}^2}dw~ v_{\o}\sigma(g,\theta) \frac{J(q'^0)J(q^0/2)J(p'^0/2)}{J(p^0/2)}|f(q)| |h(p')| \cr
&\leq C\int_{\mathbb{R}^3}dq\int_{\mathbb{S}^2}dw~ v_{\o}\sigma(g,\theta) J(q^0)J(q'^0/2) |f(q)| |h(p')|,
\end{align*}
and
\begin{align*}
|\Gamma_{3,2}(f,h)| &\leq C\int_{\mathbb{R}^3}dq\int_{\mathbb{S}^2}dw~ v_{\o}\sigma(g,\theta) J(p^0/2)J(q^0/2)J(p'^0/2)|f(q)| |h(p')|,
\end{align*}
where we used $J(p'^0/2)J(q'^0/2)=J(p^0/2)J(q^0/2)$ for the exponential term. In the case of $\Gamma_{3,2}$, there are decays of two precollisional momentum variables $J(p^0/2)J(q^0/2)$. Thus we can have the exponential decay with respect to all the pre-post momentum variables by the energy conservation: 
\begin{align*} 
J(p^0/2)J(q^0/2)=J(p^0/4)J(q^0/4)J(p'^0/4)J(q'^0/4).
\end{align*}
Thus, the estimate of $\Gamma_{3,2}$ can be absorbed in that of $\Gamma_{3,1}$. Thus, without loss of generality, we only consider the estimates of $\Gamma_{3,1}$: 
\begin{align*}
	\big| \langle \Gamma_{3,1}(f,h) ,\eta \rangle_{L^2_p} \big| 
	\leq  C\int_{\mathbb{R}^3}dp\int_{\mathbb{R}^3}dq\int_{\mathbb{S}^2}dw~ v_{\o}\sigma(g,\theta) J(q^0)J(q'^0/2)|f(q)| |h(p')| |\eta(p)| .
\end{align*}
To separate pre-collisional variable and post-collisional variable, we apply the H\"{o}lder inequality as follows:
\begin{align*}
	\big| \langle \Gamma_{3,1}(f,h) ,\eta \rangle_{L^2_p} \big|  &\leq C \left(\int_{\mathbb{R}^3}dp\int_{\mathbb{R}^3}dq\int_{\mathbb{S}^2}dw~ v_{\o}\sigma(g,\theta) J(q^0)J(q'^0/2)|f(q)|^2|\eta(p)|^2 \right)^{\frac{1}{2}} \cr
	&\times \left(\int_{\mathbb{R}^3}dp\int_{\mathbb{R}^3}dq\int_{\mathbb{S}^2}dw~ v_{\o}\sigma(g,\theta) J(q^0)J(q'^0/2)|h(p')|^2 \right)^{\frac{1}{2}} \cr
	&= II_1 \times II_2.
\end{align*}
Using $s \leq 4p^0q^0 $ and $g \leq  2(p^0)^{\frac{1}{2}}(q^0)^{\frac{1}{2}}$, we have
\begin{align*}
	(II_1)^2 &\leq C \int_{\mathbb{R}^3}dp\int_{\mathbb{R}^3}dq\int_{\mathbb{S}^2}dw~ (p^0)^{\frac{1}{2}}(q^0)^{\frac{1}{2}} J(q^0)J(q'^0/2)|f(q)|^2|\eta(p)|^2 .
\end{align*}
Since $(q^0)^{\frac{1}{2}}J(q^0/2)$ and $J(q'^0/2)$ are bounded by constants, we obtain 
\begin{align*}
	(II_1)^2&\leq C\int_{\mathbb{R}^3}dq~ |f(q)|^2\int_{\mathbb{R}^3} dp~ (p^0)^{\frac{1}{2}}|\eta(p)|^2 \leq C\|f\|_{L^2_p}^2\|\eta\|_{\nu}^2.
\end{align*}
For $II_2$, we apply the pre-post momentum change of variables $(p,q) \leftrightarrow (p',q')$ similarly to the estimate of $\Gamma_6(f,h)$. Note that $v_{\o} = \frac{g\sqrt{s}}{p^0q^0}$ is invariant under the change of variables, and we obtain
\begin{align*}
	(II_2)^2&=\int_{\mathbb{R}^3}dp'\int_{\mathbb{R}^3}dq'\int_{\mathbb{S}^2}dw~ \frac{g(p'^{\mu},q'^{\mu})\sqrt{s(p'^{\mu},q'^{\mu})}}{p'^0q'^0}\sigma(g(p'^{\mu},q'^{\mu}),\theta) J(q^0)J(q'^0/2)|h(p')|^2.
\end{align*}
Similarly, we use $g(p'^{\mu},q'^{\mu}) \leq 2(p'^0)^{1/2}(q'^0)^{1/2}  $ and $s(p'^{\mu},q'^{\mu}) \leq 4p'^0q'^0$ to have
\begin{align*}
(II_2)^2&\leq\int_{\mathbb{R}^3}dp'\int_{\mathbb{R}^3}dq'\int_{\mathbb{S}^2}dw~ (p'^0)^{\frac{1}{2}}(q'^0)^{\frac{1}{2}} J(q^0)J(q'^0/2)|h(p')|^2 \cr
&\leq\int_{\mathbb{S}^2}dw \int_{\mathbb{R}^3}dq'~(q'^0)^{\frac{1}{2}} J(q'^0/2) \int_{\mathbb{R}^3}dp'~(p'^0)^{\frac{1}{2}}|h(p')|^2 \cr
&\leq C\|h\|_{\nu}^2.
\end{align*}
Combining the estimates of $II_1$ and $II_2$ yields 
\begin{align*}
	\big| \langle \Gamma_{3,1}(f,h) ,\eta \rangle_{L^2_p} \big| &\leq C\|f\|_{L^2_p}\|h\|_{\nu}\|\eta\|_{\nu}.
\end{align*}
\newline
{\bf (3) Estimates of $\Gamma_2$:}
%Applying the change of variables $\phi \rightarrow \pi-\phi$, and $\theta\rightarrow \pi+\theta$ on $\Gamma_4(f,h)$ term, we also have $\Gamma_2(f,h)=\Gamma_4(f,h)$. So we only consider $\Gamma_2$ term. 
We first separate $\Gamma_2$ by two terms: 
\begin{align*}
\Gamma_2(f,h) &= -\tau \int_{\mathbb{R}^3}dq\int_{\mathbb{S}^2}dw~ v_{\o}\sigma(g,\theta) \big[m(q)(1+\tau m(q'))-m(q')(1+\tau m(q))\big] \cr
& \quad  \times \sqrt{m+\tau m^2(p')} f(p)h(p') \cr
& =\Gamma_{2,1}(f,h)-\Gamma_{2,2}(f,h).
\end{align*}
Using $1+\tau m(p) \leq C$, we estimate each term as follows:  
\begin{align*}
\big|\Gamma_{2,1}(f,h)\big| &\leq C \int_{\mathbb{R}^3}dq\int_{\mathbb{S}^2}dw~ v_{\o}\sigma(g,\theta) J(q^0)J(p'^0/2) |f(p)||h(p')|,	
\end{align*}
and
\begin{align*}
\big|\Gamma_{2,2}(f,h)\big| &\leq C \int_{\mathbb{R}^3}dq\int_{\mathbb{S}^2}dw~ v_{\o}\sigma(g,\theta) J(q'^0)J(p'^0/2) |f(p)||h(p')|.
\end{align*}
Since $\Gamma_{2,2}$ has the exponential decays with respect to two post-collisional variables $p'$ and $q'$, we can have the decays for all pre- and post-collisional momentum variables by using the energy conservation:
\begin{align*} 
	J(p^0/2)J(q^0/2)=J(p^0/4)J(q^0/4)J(p'^0/4)J(q'^0/4).
\end{align*}
Thus the estimate of $\Gamma_{2,2}$ can be absorbed into that of $\Gamma_{2,1}$. Thus, we only consider $\Gamma_{2,1}$ term. However, different from the previous cases, it is difficult to separate $f,h$ and $\eta$ using the H\"{o}lder inequality. Thus, we proceed the estimate via lifting the $dw$ integral to $dp'dq'$ integral imposing a $4$-dimensional Dirac-delta function: 
%\begin{align*}
%\big| \langle \Gamma_{2,1}(f,h) ,\eta \rangle_{L^2_p} \big|&\leq  C \int_{\mathbb{R}^3}dp\int_{\mathbb{R}^3}dq\int_{\mathbb{S}^2}dw v_{\o}\sigma(g,\theta) J(q^0)J(p'^0/2)|f(p)||h(p')||\eta(p)|.
%\end{align*}
\begin{multline}\label{gamma21}
\big| \langle \Gamma_{2,1}(f,h) ,\eta \rangle_{L^2_p} \big|
\leq C \int_{\mathbb{R}^3}\frac{dp}{p^0}\int_{\mathbb{R}^3}\frac{dq}{q^0}\int_{\mathbb{R}^3}\frac{dp'}{p'^0}\int_{\mathbb{R}^3}\frac{dq'}{q'^0} s\sigma(g,\theta)  \cr
\times  \delta^{(4)}(p^{\mu}+q^{\mu}-p'^{\mu}-q'^{\mu})J(q^0)J(p'^0/2)|f(p)| |h(p')| |\eta(p)|.
\end{multline}
Then we focus on the following estimate:
\begin{align}\label{B}
B&=\int_{\mathbb{R}^3}\frac{dq}{q^0}\int_{\mathbb{R}^3}\frac{dq'}{q'^0} s\sigma(g,\theta) \delta^{(4)}(p^{\mu}+q^{\mu}-p'^{\mu}-q'^{\mu})J(q^0).
\end{align}
We also lift the $dq$ and $dq'$ integrals to the relativistic energy-momentum $4$-vector integral $dq^{\mu}$ and $dq'^{\mu}$ imposing Dirac-delta function and unit step function:
\begin{multline}\label{B1}
B=\int_{\mathbb{R}^4}dq^{\mu}\int_{\mathbb{R}^4}dq'^{\mu} s\sigma(g,\theta) \delta^{(4)}(p^{\mu}+q^{\mu}-p'^{\mu}-q'^{\mu}) J(q^0) u(q^0)u(q'^0)\delta(q^{\mu}q_{\mu}+1)\delta(q'^{\mu}q'_{\mu}+1),
\end{multline}
where the unit step function $u(x)$ is defined by $1$ when $x\geq 0$ and $0$ when $x<0$. Note that $B$ has the integration over $dq^{\mu}$ and $dq'^{\mu}$, but $s$ and $g$ are functions depending on $p^{\mu}$ and $q^{\mu}$. 
Thus it is convenient to split the $q^{\mu}$ and the $q'^{\mu}$ parts from $g$ as in the lemma below.

Before we proceed furtherly, we first provide a preliminary lemma on the properties of $g,$ $\bar{g}$ and $\tilde{g}$.
\begin{lemma}[\cite{Jang-Yun-Lp, Jang2016, MR2728733}] \label{g comp} Define $g,$ $\bar{g}$ and $\tilde{g}$ as \eqref{gsymme}. Then we have 
\begin{align*}
&(1) ~	g^2=\bar{g}^2+\tilde{g}^2, \cr
&(2) ~	\bar{g}^2=-\frac{1}{2}(p^{\mu}+q'^{\mu})(q_{\mu}+p'_{\mu}-p_{\mu}-q'_{\mu}) , \cr
&(3) ~ \tilde{g}^2=-\frac{1}{2}(p^{\mu}+p'^{\mu})(q_{\mu}+q'_{\mu}-p_{\mu}-p'_{\mu}).
\end{align*}
\end{lemma}
\begin{proof} The proofs can be found in \cite{Jang-Yun-Lp, MR2728733}, but for the readers' convenience, we present the proof. \newline
(1) By definition of $g$ in \eqref{gsymme}, we have 
\begin{align*}
	-g^2+\bar{g}^2+\tilde{g}^2 %&= 2+2p^{\mu}q_{\mu} -2-2p^{\mu}p'_{\mu}-2-2p^{\mu}q'_{\mu} \cr
	&= -2+2p^{\mu}q_{\mu} -2p^{\mu}p'_{\mu}-2p^{\mu}q'_{\mu}.
\end{align*}
Since $p^{\mu}$ is energy momentum $4$-vector, we have $p^{\mu}p_{\mu}=-1$, which gives 
\begin{align*}
-g^2+\bar{g}^2+\tilde{g}^2&= 2p^{\mu}p_{\mu}+2p^{\mu}q_{\mu} -2p^{\mu}p'_{\mu}-2p^{\mu}q'_{\mu} \cr
	&=2p^{\mu}(p_{\mu}+q_{\mu} -p'_{\mu}-q'_{\mu}).
\end{align*}
Then the conservation law of energy momentum $4$-vector lead to the desired results. \newline
(2) 
We denote the right-hand side as $R$ and expand as follows: 
\begin{align*}
R=
-\frac{1}{2}(p^{\mu}q_{\mu}+p^{\mu}p'_{\mu}+q'^{\mu}q_{\mu}+q'^{\mu}p'_{\mu})+\frac{1}{2}(p^{\mu}p_{\mu}+2p^{\mu}q'_{\mu}+q'^{\mu}q'_{\mu}).
\end{align*}
We recall from \eqref{gsymme} that $p^{\mu}q_{\mu}=p'^{\mu}q'_{\mu}$ and  $p^{\mu}p'_{\mu}=q^{\mu}q'_{\mu}$ to have  
\begin{align*}
R=-\frac{1}{2}(2p^{\mu}q_{\mu}+2p^{\mu}p'_{\mu})+\frac{1}{2}(-2+2p^{\mu}q'_{\mu})
=-p^{\mu}q_{\mu}+p^{\mu}q'_{\mu}-p^{\mu}p'_{\mu}-1.
\end{align*}
Using $p^{\mu}p_{\mu}=-1$, we have
\begin{align*}
R%=-p^{\mu}q_{\mu}+p^{\mu}q'_{\mu}-p^{\mu}p'_{\mu}+p^{\mu}p_{\mu}
=p^{\mu}(-q_{\mu}+q'_{\mu}-p'_{\mu}+p_{\mu}).
\end{align*}
Then the conservation laws of energy momentum $4$-vector yields 
\begin{align*}
R=2p^{\mu}(-p'_{\mu}+p_{\mu}) = -2p^{\mu}p'_{\mu}-2 .
\end{align*}
By definition of $\bar{g}$, we derived desired results. \newline
(3) The proof can be obtained by the same way as that of (2). So we omit it.
\end{proof}

By Lemma \ref{g comp} (1) and (3) above, we use the following representation of $g$: 
\begin{align*}
	g^2=\bar{g}^2-\frac{1}{2}(p^{\mu}+p'^{\mu})(q_{\mu}+q'_{\mu}-p_{\mu}-p'_{\mu}).
\end{align*}
Now we go back to the estimates of $B$ and apply the following change of variables 
\begin{align}\label{changofv}
	\bar{q}^{\mu}=q^{\mu}+q'^{\mu}, \qquad \bar{q}'^{\mu}=q^{\mu}-q'^{\mu}.
\end{align}
Then the reverse relation can be written by 
\begin{align*}
	q^{\mu}=\frac{1}{2}(\bar{q}^{\mu}+\bar{q}'^{\mu}),\qquad  q'^{\mu}=\frac{1}{2}(\bar{q}^{\mu}-\bar{q}'^{\mu}),
\end{align*}
and the Jacobian is given by 
\begin{align*}
	\frac{\partial (q^{\mu},q'^{\mu})}{\partial(\bar{q}^{\mu},\bar{q}'^{\mu})} = \frac{1}{16}.
\end{align*}
Then the $g_c$, $s_c$ and $\theta_c$ are now expressed as
\begin{align}\label{gs}
g_c^2=\bar{g}^2-\frac{1}{2}(p^{\mu}+p'^{\mu})(\bar{q}_{\mu}-p_{\mu}-p'_{\mu}), \quad s_c=g_c^2+4, \quad 	\cos\theta_c= 1-2\frac{\bar{g}^2}{g_c^2}.
\end{align}
To calculate the Dirac-delta function and the unit step function in \eqref{B1}, we use followings.
First we use $\delta(x)\delta(y)=2\delta(x+y)\delta(x-y)$ to have 
\begin{align*}
	\delta(q^{\mu}q_{\mu}+1)\delta(q'^{\mu}q'_{\mu}+1) &= 2\delta(q^{\mu}q_{\mu}+q'^{\mu}q'_{\mu}+2)\delta(q^{\mu}q_{\mu}-q'^{\mu}q'_{\mu}) \cr
	%&=2\delta\left(\frac{1}{4}(\bar{q}^{\mu}+\bar{q}'^{\mu})(\bar{q}_{\mu}+\bar{q}'_{\mu})+\frac{1}{4}(\bar{q}^{\mu}-\bar{q}'^{\mu})(\bar{q}_{\mu}-\bar{q}'_{\mu})+2\right)\cr
	%&\times \delta\left(\frac{1}{4}(\bar{q}^{\mu}+\bar{q}'^{\mu})(\bar{q}_{\mu}+\bar{q}'_{\mu})-\frac{1}{4}(\bar{q}^{\mu}-\bar{q}'^{\mu})(\bar{q}_{\mu}-\bar{q}'_{\mu})\right) \cr
	&=4\delta((\bar{q}^{\mu}\bar{q}_{\mu}+\bar{q}'^{\mu}\bar{q}'_{\mu})+4)\delta(\bar{q}^{\mu}\bar{q}'_{\mu}).
\end{align*}
Note that $q^0\geq 0$ and $q'^0\geq 0$ are equivalent to $q^0+q'^0\geq 0$ and $q^0q'^0 \geq 0 $. Also, $q^0q'^0 \geq 0 $ is equivalent to $\bar{g}^2=2q^0q'^0-2q\cdot q'-2 \geq 0$ under the assumption that $q^{\mu}q_{\mu}+1 = 0 $ and $q'^{\mu}q'_{\mu}+1=0$. Thus we have
\begin{align*}
	u(q^0)u(q'^0)\delta(q^{\mu}q_{\mu}+1)\delta(q'^{\mu}q'_{\mu}+1) &= 
	u(\bar{q}^0)u(\bar{s}-4)4\delta((\bar{q}^{\mu}\bar{q}_{\mu}+\bar{q}'^{\mu}\bar{q}'_{\mu})+4)\delta(\bar{q}^{\mu}\bar{q}'_{\mu}).
\end{align*}
Thus we have
\begin{align*}
B&=\frac{1}{4}\int_{\mathbb{R}^4\times \mathbb{R}^4}d\Theta(\bar{q}^{\mu},\bar{q}'^{\mu}) s_c\sigma(g_c,\theta_c) \delta^{(4)}(p^{\mu}-p'^{\mu}+\bar{q}'^{\mu}) J\left(\frac{\bar{q}^0+\bar{q}'^0}{2}\right),
\end{align*}
where
\begin{align*}
d\Theta(\bar{q}^{\mu},\bar{q}'^{\mu})&=  d\bar{q}^{\mu}d\bar{q}'^{\mu}u(\bar{q}^0)u(\bar{s}-4)\delta((\bar{q}^{\mu}\bar{q}_{\mu}+\bar{q}'^{\mu}\bar{q}'_{\mu})+4)\delta(\bar{q}^{\mu}\bar{q}'_{\mu}).
\end{align*}
Now we substitute $\bar{q}'^{\mu}=p'^{\mu}-p^{\mu}$ to reduce the $4$-dimensional integration $d\bar{q}'^{\mu}$ by reducing the $4$-dimensional Dirac-delta function as follows:
\begin{align*}
B&=\frac{1}{4}\int_{\mathbb{R}^4}d\Theta(\bar{q}^{\mu}) s_c\sigma(g_c,\theta_c) J\left(\frac{\bar{q}^0+p'^0-p^0}{2}\right),
\end{align*}
where
\begin{align*}
d\Theta(\bar{q}^{\mu})&=  d\bar{q}^{\mu}u(\bar{q}^0)u(\bar{s}-4)\delta(\bar{q}^{\mu}\bar{q}_{\mu}+(p'^{\mu}-p^{\mu})(p'_{\mu}-p_{\mu})+4)\delta(\bar{q}^{\mu}(p'_{\mu}-p_{\mu})).
\end{align*}
To remove one more Dirac-delta function, we follow that
\begin{align*}
u(\bar{q}^0)\delta(\bar{q}^{\mu}\bar{q}_{\mu}+(p'^{\mu}-p^{\mu})(p'_{\mu}-p_{\mu})+4) %&= u(\bar{q}^0)\delta(\bar{q}^{\mu}\bar{q}_{\mu}+(p'^{\mu}p'_{\mu}-2p^{\mu}p'_{\mu}+p^{\mu}p_{\mu})+4) \cr
&= u(\bar{q}^0)\delta(\bar{q}^{\mu}\bar{q}_{\mu}-2p^{\mu}p'_{\mu}+2) \cr
&= u(\bar{q}^0)\delta(-(\bar{q}^0)^2+|\bar{q}|^2+\bar{s}) \cr
&= \frac{\delta(\bar{q}^0-\sqrt{|\bar{q}|^2+\bar{s}})}{2\sqrt{|\bar{q}|^2+\bar{s}}}.
\end{align*}
Then the $d\bar{q}^0$ integral with the delta function above is reduced as follows:
\begin{align*}
B&=\frac{1}{4}J\left(\frac{p'^0-p^0}{2}\right)\int_{\mathbb{R}^3}\frac{d\bar{q}}{\bar{q}^0} s_c\sigma(g_c,\theta_c) J\left(\frac{\bar{q}^0}{2}\right)u(\bar{s}-4)\delta(\bar{q}^{\mu}(p'_{\mu}-p_{\mu})),
\end{align*}
where $\bar{q}^0$ is defined by $\sqrt{|\bar{q}|^2+\bar{s}}$. Since $\bar{s}-4=\bar{g}^2\geq0$, the last unit step function is always equal to $1$. In order to make the final Dirac-delta function even simpler, we consider the Lorentz transform satisfying
\begin{align*}
\Lambda(p_{\mu}+p'_{\mu}) &= (\sqrt{\bar{s}},0,0,0), \qquad \Lambda(p_{\mu}-p'_{\mu}) = (0,0,0,-\bar{g}).
\end{align*}
We can see that in \cite{MR2728733} that the Lorentz transform satisfying the relation above is uniquely determined by following form: 
\begin{align}\label{Lorentz}
\Lambda &= \left[ {\begin{array}{cccc}
\frac{p^0+p'^0}{\sqrt{\bar{s}}} & -\frac{p_1+p'_1}{\sqrt{\bar{s}}} & -\frac{p_2+p'_2}{\sqrt{\bar{s}}} & -\frac{p_3+p'_3}{\sqrt{\bar{s}}}  \\
\Lambda^{01} & \Lambda^{11} & \Lambda^{21} & \Lambda^{31}  \\
0 & \frac{(p \times p')_1}{|p\times p'|} & \frac{(p \times p')_2}{|p\times p'|} & \frac{(p \times p')_3}{|p\times p'|} \\
\frac{p^0-p'^0}{\bar{g}} & -\frac{p_1-p'_1}{\bar{g}} & -\frac{p_2-p'_2}{\bar{g}} &  -\frac{p_3-p'_3}{\bar{g}} 
\end{array} } \right],
\end{align}
where
\begin{align*}
	\Lambda^{i1} = \frac{2(p_i\{p^0+p'^0p^{\mu}p'_{\mu}\}+p'_i\{p'^0+p^0p^{\mu}p'_{\mu}\})}{\bar{g}\sqrt{\bar{s}}|p\times p'|}.
\end{align*}
Then the exponential part can also be expressed carrying out the Lorentz transform as follows: 
\begin{align*}
J\left(\bar{q}^0/2\right)  = \exp\left(-\bar{q}_0/2\right) = \exp\left(-\bar{q}^{\mu}U_{\mu}/2\right)= \exp\left(-\Lambda\bar{q}^{\mu}\Lambda U_{\mu}/2\right),
\end{align*}
where we used following simple four-vectors:
\begin{align*}
	U^{\mu} = (1,0,0,0), \qquad U_{\mu} = (-1,0,0,0).
\end{align*}
Applying this Lorentz transform yields 
\begin{align*}
B&= \frac{1}{4}J\left(\frac{p'^0-p^0}{2}\right)\int_{\mathbb{R}^3}\frac{d\bar{q}}{\bar{q}^0} s_{\Lambda}\sigma(g_{\Lambda},\theta_{\Lambda}) \exp\left(-\Lambda\bar{q}^{\mu}\Lambda U_{\mu}/2\right)\delta(\Lambda \bar{q}^{\mu}\Lambda(p'_{\mu}-p_{\mu})),
\end{align*}
where $g_{\Lambda}$ is written via the Lorentz transform:
\begin{align*}
g_{\Lambda}^2&=\bar{g}^2-\frac{1}{2}\Lambda(p^{\mu}+p'^{\mu})\Lambda(\bar{q}_{\mu}-p_{\mu}-p'_{\mu}) \cr
&=\bar{g}^2+\frac{1}{2}(\sqrt{\bar{s}},0,0,0)(\Lambda\bar{q}_{\mu}-(\sqrt{\bar{s}},0,0,0)).
\end{align*}
We can also represent $s_\Lambda$ and $\cos\theta_\Lambda$ in terms of $\bar{g}$ and $g_\Lambda$ as
\begin{align*}
s_{\Lambda}=g_{\Lambda}^2+4 \quad \text{and}\quad 	\cos\theta_{\Lambda}= 1-2\frac{\bar{g}^2}{g_{\Lambda}^2}.
\end{align*}
We now apply the change of variables $\Lambda\bar{q}^{\mu}= \bar{q}^{\mu}$. Since $d\bar{q}/\bar{q}^0$ is Lorentz invariant, we have
\begin{align*}
B&= \frac{1}{4}J\left(\frac{p'^0-p^0}{2}\right)\int_{\mathbb{R}^3}\frac{d\bar{q}}{\bar{q}^0} s_{\lambda}\sigma(g_{\lambda},\theta_{\lambda}) \exp\left(-\bar{q}^{\mu}\Lambda U_{\mu}/2\right)\delta(\bar{q}^3\bar{g}),
\end{align*}
where
\begin{align}\label{gs2}
g_{\lambda}^2&=\bar{g}^2+\frac{1}{2}\sqrt{\bar{s}}(\bar{q}^0-\sqrt{\bar{s}}), \quad s_{\lambda}=g_{\lambda}^2+4, \quad \cos\theta_{\lambda}= 1-2\frac{\bar{g}^2}{g_{\lambda}^2}.
\end{align}
Now we use the spherical coordinates for the variable $\bar{q}$:
\begin{align*}
\bar{q} = |\bar{q}|(\sin\psi\cos\phi,\sin\psi\sin\phi, \cos\psi).
\end{align*}
Then $B$ is equal to 
\begin{align*}
\frac{1}{4}J\left(\frac{p'^0-p^0}{2}\right)\int_{0}^{2\pi}d\phi \int_{0}^{\pi}d\psi \sin\psi  \int_{0}^{\infty}\frac{|\bar{q}|^2d|\bar{q}|}{\bar{q}^0} s_{\lambda}\sigma(g_{\lambda},\theta_{\lambda}) \exp\left(-\bar{q}^{\mu}\Lambda U_{\mu}/2\right)\delta(|\bar{q}|\cos\psi\bar{g}).
\end{align*}
Using $\delta(ax)=(1/a)\delta(x)$ and substituting $\psi= \pi/2$ reduce the $d\psi$ integral with $\delta(\cos\psi)$ as follows: 
\begin{align}\label{spher}
B&= \frac{1}{4}J\left(\frac{p'^0-p^0}{2}\right)\int_{0}^{2\pi}d\phi  \int_{0}^{\infty}\frac{|\bar{q}|d|\bar{q}|}{\bar{g}\bar{q}^0} s_{\lambda}\sigma(g_{\lambda},\theta_{\lambda}) \exp\left(-\bar{q}^{\mu}\Lambda U_{\mu}/2\right)|_{\psi=\pi/2}.
\end{align}
We first consider the scattering kernel $s_{\lambda}\sigma(g_{\lambda},\theta_{\lambda})$. The half-angle formula from \eqref{gs2} gives
\begin{align*}
\cos\theta_{\lambda}=1-2\sin^2\frac{\theta_{\lambda}}{2}= 1-2\frac{\bar{g}^2}{g_{\lambda}^2},
\end{align*}
which implies
\begin{align}\label{half angle}
	\sin\frac{\theta_{\lambda}}{2} =\frac{\bar{g}}{g_{\lambda}},
\end{align}
for  $0 \leq \theta_{\lambda} \leq 2\pi$. Thus we have 
\begin{align*}
\sigma(g_{\lambda},\theta_{\lambda}) = g_{\lambda}\sin\theta_{\lambda} = 2g_{\lambda}\sin\frac{\theta_{\lambda}}{2}\cos\frac{\theta_{\lambda}}{2} = 2\bar{g}\cos\frac{\theta_{\lambda}}{2} \leq 2\bar{g}.
\end{align*}
From \eqref{gs2} and the definition of $\bar{q}^0$, we have 
\begin{align*}
s_{\lambda}=g_{\lambda}^2+4=\bar{g}^2+4+\frac{1}{2}\sqrt{\bar{s}}(\bar{q}^0-\sqrt{\bar{s}}) = \bar{s}+\frac{1}{2}\sqrt{\bar{s}}(\sqrt{|\bar{q}|^2+\bar{s}}-\sqrt{\bar{s}}).
\end{align*}
Then we consider the exponential part of \eqref{spher}. Using the Lorentz transform \eqref{Lorentz}, we have
\begin{align*}
	\Lambda U_{\mu}=\left( \frac{p^0+p'^0}{\sqrt{\bar{s}}}, \frac{2 |p \times p'|}{\bar{g}\sqrt{\bar{s}}} , 0 , \frac{p^0-p'^0}{\bar{g}} \right).
\end{align*}
Combining with the spherically expression of $\bar{q}$, we calculate
%\begin{align*}
%\bar{q}^{\mu}\Lambda U_{\mu} = (\sqrt{1+r^2},r\sin\phi \cos\vartheta,r\sin\phi \sin\vartheta,r\cos\phi )\left( \frac{p^0+p'^0}{\sqrt{\bar{s}}}, \frac{2 |p \times p'|}{\bar{g}\sqrt{\bar{s}}} , 0 , \frac{p^0-p'^0}{\bar{g}} \right)
%\end{align*}
\begin{align*}
\bar{q}^{\mu}\Lambda U_{\mu}|_{\psi=\pi/2} &= (\sqrt{|\bar{q}|^2+\bar{s}},|\bar{q}| \cos\phi,|\bar{q}| \sin\phi,0 )\left( \frac{p^0+p'^0}{\sqrt{\bar{s}}}, \frac{2 |p \times p'|}{\bar{g}\sqrt{\bar{s}}} , 0 , \frac{p^0-p'^0}{\bar{g}} \right) \cr
&=-\sqrt{|\bar{q}|^2+\bar{s}}\frac{p^0+p'^0}{\sqrt{\bar{s}}} + |\bar{q}| \cos\phi\frac{2 |p \times p'|}{\bar{g}\sqrt{\bar{s}}}.
\end{align*} 
Thus, \eqref{spher} is bounded by 
\begin{align*}
B&\leq \frac{1}{2}J\left(\frac{p'^0-p^0}{2}\right)\int_{0}^{2\pi}d\phi  \int_{0}^{\infty}\frac{|\bar{q}|d|\bar{q}|}{\sqrt{|\bar{q}|^2+\bar{s}}} s_{\lambda} \exp\left(-\sqrt{|\bar{q}|^2+\bar{s}}\frac{p^0+p'^0}{2\sqrt{\bar{s}}} + |\bar{q}| \cos\phi\frac{ |p \times p'|}{\bar{g}\sqrt{\bar{s}}}\right).
\end{align*}
We apply the change of variables $|\bar{q}|=\sqrt{\bar{s}}y$ to have 
\begin{align*}
B&\leq \frac{1}{2}J\left(\frac{p'^0-p^0}{2}\right)\int_{0}^{2\pi}d\phi \int_{0}^{\infty}\frac{\sqrt{\bar{s}}ydy}{\sqrt{y^2+1}} s_{\lambda} \exp\left(-\sqrt{y^2+1}\frac{p^0+p'^0}{2} + y \cos\phi\frac{ |p \times p'|}{\bar{g}}\right),
\end{align*}
where
\begin{align*}
s_{\lambda}=\bar{s}+\frac{1}{2}\sqrt{\bar{s}}(\sqrt{|\bar{q}|^2+\bar{s}}-\sqrt{\bar{s}}) = \frac{\bar{s}}{2}(1+\sqrt{y^2+1}).
\end{align*}
So we have
\begin{multline*}
B%&\leq \frac{\bar{s}^{3/2}}{2}J\left(\frac{p'^0-p^0}{2}\right)\int_{0}^{2\pi}d\phi  \int_{0}^{\infty}\frac{ydy}{\sqrt{y^2+1}} \frac{1+\sqrt{y^2+1}}{2}  \exp\left(-\sqrt{y^2+1}\frac{p^0+p'^0}{2} + y \cos\phi\frac{ |p \times p'|}{\bar{g}}\right) \cr
\leq \frac{\bar{s}^{3/2}}{2}J\left(\frac{p'^0-p^0}{2}\right) \int_{0}^{\infty}\frac{ydy}{\sqrt{y^2+1}} \frac{1+\sqrt{y^2+1}}{2} \cr
\times  \exp\left(-\sqrt{y^2+1}\frac{p^0+p'^0}{2}\right)\int_{0}^{2\pi}d\phi  \exp\left(y \cos\phi\frac{ |p \times p'|}{\bar{g}}\right),
\end{multline*}
For the notational simplicity, we denote
\begin{align}\label{Rrdef}
	R= \frac{p^0+p'^0}{2}, \qquad r= \frac{ |p \times p'|}{\bar{g}},
\end{align}
then we can write
\begin{align*}
B\leq \frac{\bar{s}^{3/2}\pi}{2}J\left(\frac{p'^0-p^0}{2}\right) \int_{0}^{\infty}dy \left(\frac{y}{\sqrt{y^2+1}}+y\right)\exp\left(-R\sqrt{y^2+1}\right)I_0(ry),
\end{align*}
%\begin{align*}
%\int_{0}^{2\pi}d\phi  \exp\left(y \cos\phi\frac{ |p \times p'|}{\bar{g}}\right) = 2\pi I_0\left(y \frac{ |p \times p'|}{\bar{g}}\right)
%\end{align*}
where $I_0$ denotes the modified Bessel function of the first kind:
\begin{align*}
I_0(y) = \frac{1}{\pi}\int_0^{\pi} e^{y\cos\phi} d\phi.
\end{align*}
The formula including above Bessel function can be found in   \cite{MR3307944} and \cite{MR1211782} when $R> r\geq0 $:
\begin{align}\label{Bessel comp}
\begin{split}
I&= \int_{0}^{\infty} \frac{y}{\sqrt{y^2+1}}e^{-R\sqrt{y^2+1}}  I_0\left(ry\right) dy = \frac{e^{-\sqrt{R^2-r^2}}}{\sqrt{R^2-r^2}}, \cr
II&= \int_{0}^{\infty} y e^{-R\sqrt{y^2+1}} I_0\left(ry\right) dy = \frac{R}{R^2-r^2}\left(1+ \frac{1}{\sqrt{R^2-r^2}}\right)e^{-\sqrt{R^2-r^2}}.
\end{split}
\end{align}Therefore, we have
\begin{align}\label{B3}
B&\leq \frac{\bar{s}^{3/2}\pi}{2}J\left(\frac{p'^0-p^0}{2}\right) \left[\frac{1}{\sqrt{R^2-r^2}}+\frac{R}{R^2-r^2}+\frac{R}{(R^2-r^2)^{3/2}}\right]e^{-\sqrt{R^2-r^2}}.
\end{align}
We can also find the useful estimates of $R^2-r^2$ in \cite{MR1211782} as
\begin{align}\label{Rr}
R^2-r^2 %=\big(\frac{(p^0+p'^0)^2}{4} - \frac{|p\times p'|^2}{\bar{g}^2} \big)  
= |p-p'|^2\frac{\bar{g}^2+4}{4\bar{g}^2} 
\geq \max\left\{\frac{\bar{g}^2}{4}+1,\frac{1}{4}|p-p'|^2\right\},
\end{align}
which lead to 
\begin{align*}
J\left(\frac{p'^0-p^0}{2}\right)e^{-\sqrt{R^2-r^2}} \leq e^{-\frac{1}{2}(p'^0-p^0)}e^{-\frac{1}{2}|p-p'|} \leq 1.
\end{align*}
In the last inequality, we used 
\begin{align*}
|p'^0-p^0| = \big| \sqrt{1+|p'|^2}-\sqrt{1+|p|^2}\big| = %\frac{||p'|^2-|p|^2|}{\sqrt{1+|p'|^2}+\sqrt{1+|p|^2}} =
\frac{|(|p'|-|p|)(|p'|+|p|)|}{\sqrt{1+|p'|^2}+\sqrt{1+|p|^2}}
%\leq ||p'|-|p|| 
\leq |p'-p|.
\end{align*}
The estimates \eqref{Rr} also implies that 
\begin{align*}
\sqrt{R^2-r^2} &\geq 1,\qquad \frac{1}{R^2-r^2} \leq \frac{4}{\bar{s}}.
\end{align*}
Thus the three terms inside the large bracket of \eqref{B3} are bounded as follows:
\begin{align*}
\frac{1}{\sqrt{R^2-r^2}}+\frac{R}{(R^2-r^2)^{3/2}} \leq \frac{2R}{R^2-r^2} \leq 4\frac{p^0+p'^0}{\bar{s}}.
\end{align*}
Combining the estimates above yields 
\begin{align}\label{B esti}
B&\leq C \bar{s}^{1/2}(p^0+p'^0).
\end{align}
Now substituting it in \eqref{gamma21}, we turn back to estimate of $\Gamma_{2,1}$ as
\begin{align*}
\big| \langle \Gamma_{2,1}(f,h) ,\eta \rangle_{L^2_p} \big|&\leq C \int_{\mathbb{R}^3}\frac{dp}{p^0}\int_{\mathbb{R}^3}\frac{dp'}{p'^0}\bar{s}^{\frac{1}{2}}(p^0+p'^0)J(p'^0/2)|f(p)||h(p')||\eta(p)|.
\end{align*}
We use $\bar{s}\leq 4 p^0p'^0$ to have 
\begin{align*}
\big| \langle \Gamma_{2,1}(f,h) ,\eta \rangle_{L^2_p} \big|&\leq C \int_{\mathbb{R}^3}dp\int_{\mathbb{R}^3}dp'\left( \sqrt{\frac{p^0}{p'^0}}+\sqrt{\frac{p'^0}{p^0}}\right)J(p'^0/2)|f(p)||h(p')||\eta(p)|.
\end{align*}
Using the boundedness $1\leq p^0$ and $\sqrt{p'^0}J(p'^0/4)\leq C$ yields
\begin{align*}
	\big| \langle \Gamma_{2,1}(f,h) ,\eta \rangle_{L^2_p} \big|&\leq C \int_{\mathbb{R}^3}dp\int_{\mathbb{R}^3}dp'(p^0)^{\frac{1}{2}}J(p'^0/4)|f(p)||h(p')||\eta(p)|.
\end{align*}
Finally, we apply the H\"{o}lder inequality and obtain
\begin{align*}
\big| \langle \Gamma_{2,1}(f,h) ,\eta \rangle_{L^2_p} \big|&\leq C \|h\|_{L^2_p} \int_{\mathbb{R}^3}dp~(p^0)^{\frac{1}{2}}|f(p)||\eta(p)| \cr
&\leq C \|f\|_{\nu}\|h\|_{L^2_p}\|\eta\|_{\nu}.
\end{align*}
\end{proof}

Regarding the linear terms $K_1f$ and $K_2f$ in \eqref{Kf}, we have the following estimates.
\begin{lemma}\label{K esti} For $i=1,2$, the compact operator $K_i$ satisfies the following estimate.
\begin{align*}
\langle	K_if , h \rangle_{L^2_p}&\leq C \|f \|_{\nu}\|h\|_{\nu}.
\end{align*}
\end{lemma}
\begin{proof}
Recall the definition of $K_1$ in \eqref{Kf}. Using the energy conservation law, we can replace  post-collisional variables by pre-collisional variables on the exponential term:
\begin{align*}
K_1f%&=\int_{\mathbb{R}^3}dq\int_{\mathbb{S}^2}dw~ v_{\o}\sigma(g,\theta) \sqrt{m+\tau m^2(p')}\sqrt{m+\tau m^2(q')}f(q)\cr
&\leq C\int_{\mathbb{R}^3}dq\int_{\mathbb{S}^2}dw~ v_{\o}\sigma(g,\theta) J(p^0/2)J(q^0/2)|f(q)|.
\end{align*}
%\begin{align*}
%\langle	K_1f , h \rangle_{L^2_p}&\leq  \int_{\mathbb{R}^3}dp\int_{\mathbb{R}^3}dq\int_{\mathbb{S}^2}dw v_{\o}\sigma(g,\theta) J(p^0/2)J(q^0/2)|f(q)||h(p)|. %\cr
%&\leq \left(\int_{\mathbb{R}^3}p^0J(p^0)dp\right)^{\frac{1}{2}}\left(\int_{\mathbb{R}^3}|h(p)|^2dp\right)^{\frac{1}{2}}\left(\int_{\mathbb{R}^3}q^0J(q^0)dq\right)^{\frac{1}{2}}\left(\int_{\mathbb{R}^3}|f(q)|^2dq\right)^{\frac{1}{2}}.
%\end{align*}
Then by $v_{\o}\leq 1$ and $g\leq (p^0)^{1/2}(q^0)^{1/2}$ and the H\"{o}lder inequality, we have   
\begin{align*}
\big| \langle K_1f , h \rangle_{L^2_p} \big|&\leq C\|f\|_{L^2_p}\|h\|_{L^2_p}.
\end{align*}
For $K_2$, by the simple boundedness $\sqrt{m+\tau m^2(p)}\leq CJ(p^0/2)$, we have  
\begin{align*}
K_2f&\leq C\int_{\mathbb{R}^3}dq\int_{\mathbb{S}^2}dw~ v_{\o}\sigma(g,\theta) J(q^0/2)J(q'^0/2) |f(p')| .
\end{align*}
Then we take the inner product with $h$ and write it as 
\begin{align*}
\langle	K_2f , h \rangle_{L^2_p}&\leq C
\int_{\mathbb{R}^3}dp\int_{\mathbb{R}^3}dq\int_{\mathbb{S}^2}dw~ v_{\o}\sigma(g,\theta) J(q^0/2)J(q'^0/2) |f(p')||h(p)| .
\end{align*}
If we substitute $h(q')=J(q'^0/2)$ in $\Gamma_6$ in \eqref{gamma6}, then we get 
\begin{align*}
	\big| \langle \Gamma_6(f,J) ,h \rangle_{L^2_p} \big|&\leq C \int_{\mathbb{R}^3}dp\int_{\mathbb{R}^3}dq\int_{\mathbb{S}^2}dw \ v_{\o}\sigma(g,\theta) J(q^0/2)J(q'^0/2)|f(p')||h(p)|.
\end{align*}
Thus, the estimates of $\Gamma_6(f,h)$ gives 
\begin{align*}
	\big| \langle K_2f , h \rangle_{L^2_p} \big|&\leq C\|f\|_{\nu}\|h\|_{\nu}.
\end{align*}
\end{proof}

\subsection{Estimates of the third-order nonlinear terms}

In this subsection, we provide the upper-bound estimates for the third-order nonlinear terms. 
\begin{lemma}\label{nonlin fff} We have
\begin{align*}
\big|\langle T(f,h,\eta) ,\xi \rangle_{L^2_p} \big| \leq  C\left( \|f\|_{L^2_p}\|h\|_{L^2_p}\|\eta\|_{\nu} +\|f\|_{L^2_p}\|h\|_{\nu}\|\eta\|_{L^2_p}  +\|f\|_{\nu }\|h\|_{L^2_p}\|\eta\|_{L^2_p} \right)\|\xi\|_{\nu}.
\end{align*}
\end{lemma}
\begin{proof}
Recall the definition of third-order nonlinear terms in \eqref{T}. Applying the change of variables $\phi \rightarrow \pi-\phi$ and $\theta\rightarrow \pi+\theta$, we have $T_1(f,h,\eta)=T_2(f,h,\eta)$. Thus we only consider the proof for $T_1$,$T_3$ and $T_4$. \newline
{\bf (1) Estimates of $T_1$: } Using the boundedness $\sqrt{m+\tau m^2(p)}\leq CJ(p^0/2)$, we have
\begin{align*}
|T_1(f,h,\eta)|%&= \int_{\mathbb{R}^3}dq\int_{\mathbb{S}^2}dw~v_{\o}\sigma(g,\theta) f(p)\sqrt{m+\tau m^2(q)}h(q)\sqrt{m+\tau m^2(p')}\eta(p') \cr
&\leq C \int_{\mathbb{R}^3}dq\int_{\mathbb{S}^2}dw~ v_{\o}\sigma(g,\theta)J(p'^0/2)J(q^0/2)|f(p)||h(q)||\eta(p')|.
\end{align*} 
We take the inner product with $\xi$ to have 
\begin{multline*}
\big| \langle T_1(f,h,\eta), \xi \rangle_{L^2_p} \big|  \cr
\leq C \int_{\mathbb{R}^3}dp \int_{\mathbb{R}^3}dq\int_{\mathbb{S}^2}dw~v_{\o}\sigma(g,\theta)J(p'^0/2)J(q^0/2) |f(p)||h(q)||\eta(p')||\xi(p)|.
\end{multline*}
Using the H\"{o}lder inequality, we disunite the integrand as follows: 
\begin{align*}
\big| \langle T_1(f,h,\eta), \xi \rangle_{L^2_p} \big| &\leq C \left(\int_{\mathbb{R}^3}dp \int_{\mathbb{R}^3}dq\int_{\mathbb{S}^2}dw~v_{\o}\sigma(g,\theta)J(p'^0/2)J(q^0/2) |f(p)|^2|h(q)|^2\right)^{\frac{1}{2}} \cr
&\times \left(\int_{\mathbb{R}^3}dp \int_{\mathbb{R}^3}dq\int_{\mathbb{S}^2}dw~v_{\o}\sigma(g,\theta)J(p'^0/2)J(q^0/2)|\eta(p')|^2|\xi(p)|^2\right)^{\frac{1}{2}} \cr
&= T_{11}\times T_{12}.
\end{align*}
We first consider $T_{11}$. Using $v_{\o} \leq C$ and $g\leq 2 \sqrt{p^0q^0}$, we have 
\begin{align*}
(T_{11})^2 &\leq \int_{\mathbb{R}^3}dp \int_{\mathbb{R}^3}dq\int_{\mathbb{S}^2}dw~(p^0)^{\frac{1}{2}}(q^0)^{\frac{1}{2}}J(p'^0/2)J(q^0/2) |h(q)|^2|f(p)|^2 \cr
&\leq C \int_{\mathbb{R}^3} dq|h(q)|^2 \int_{\mathbb{R}^3}dp(p^0)^{\frac{1}{2}}|f(p)|^2  \cr
&\leq C \|f\|_{\nu}^2\|h\|_{L^2_{p}}^2,
\end{align*}
where we used $(q^0)^{1/2}J(q^0/2)\leq C$. To estimate $T_{12}$, we rewrite $(T_{12})^2$ as in the form of \eqref{gamma21}:
\begin{multline*}
(T_{12})^2=\int_{\mathbb{R}^3}\frac{dp}{p^0} \int_{\mathbb{R}^3}\frac{dq}{q^0}\int_{\mathbb{R}^3}\frac{dp'}{p'^0}\int_{\mathbb{R}^3}\frac{dq'}{q'^0} s\sigma(g,\theta)J(p'^0/2)J(q^0/2)\\\times |\eta(p')|^2|\xi(p)|^2\delta^{(4)}(p^{\mu}+q^{\mu}-q'^{\mu}-q'^{\mu}).
\end{multline*}
Then we can find that the $dqdq'$ integral has a similar form as of \eqref{B}:
\begin{align*}
B'&= \int_{\mathbb{R}^3}\frac{dq}{q^0}\int_{\mathbb{R}^3}\frac{dq'}{q'^0} s\sigma(g,\theta)J(q^0/2) \delta^{(4)}(p^{\mu}+q^{\mu}-q'^{\mu}-q'^{\mu}).
\end{align*}
The only difference is that $J(q^0)$ is modified by $J(q^0/2)$. Thus the exponential growth $J((p'^0-p^0)/2)$ in \eqref{B3} is changed by $J((p'^0-p^0)/4)$. But since the $R$ and $r$ in \eqref{Rrdef} are also changed by $R/2$ and $r/2$, we can have
\begin{align*}
J\left(\frac{p'^0-p^0}{4}\right)e^{-\sqrt{R^2-r^2}/2} \leq e^{-\frac{1}{4}(p'^0-p^0)}e^{-\frac{1}{4}|p-p'|} \leq 1.
\end{align*}
Thus we can apply the estimate in \eqref{B esti}:
\begin{align*}
	B&\leq C \bar{s}^{1/2}(p^0+p'^0),
\end{align*}
which gives
\begin{align}\label{T_12}
(T_{12})^2	\leq C \int_{\mathbb{R}^3}\frac{dp}{p^0} \int_{\mathbb{R}^3}\frac{dp'}{p'^0}\bar{s}^{1/2}(p^0+p'^0)J(p'^0/2)|\eta(p')|^2|\xi(p)|^2.
\end{align}
Using $\bar{s} \leq C p^0p'^0$ and $(p'^0)^{1/2}J(p'^0/2)\leq C$, we have
\begin{align*}
(T_{12})^2	\leq C \int_{\mathbb{R}^3}dp(p^0)^{1/2}|\xi(p)|^2 \int_{\mathbb{R}^3}dp'|\eta(p')|^2 \leq C\|\eta\|_{L^2_p}^2\|\xi\|_{\nu}^2 .
\end{align*}
Combining the estimates of $T_{11}$ and $T_{12}$ yields
\begin{align*}
\big| \langle T_1(f,h,\eta), \xi \rangle_{L^2_p} \big| \leq C \|f\|_{\nu}\|h\|_{L^2_{p}}\|\eta\|_{L^2_p}\|\xi\|_{\nu}.
\end{align*}
{\bf (2) Estimates of $T_3$: } Using $\sqrt{m+\tau m^2(p)}\leq J(p^0/2)$, we estimate $T_3$ as follows: 
\begin{align*}
|T_3(f,h,\eta)|%&= -\int_{\mathbb{R}^3}dq\int_{\mathbb{S}^2}dw v_{\o}\sigma(g,\theta)\sqrt{m+\tau m^2(q')}f(q')\sqrt{m+\tau m^2(p')}h(p')\eta(p) \cr
&\leq \int_{\mathbb{R}^3}dq\int_{\mathbb{S}^2}dw~ v_{\o}\sigma(g,\theta)J(p'^0/2)J(q'^0/2)|f(p)||h(p')||\eta(q')|.
\end{align*}
%The exponential decays about two post collisional variables $p'^0$ and $q'^0$ make the estimate of $T_3$ easy. 
Via the H\"{o}lder inequality, we can split the variable as follows: 
\begin{align*}
\big| \langle T_3(f,h,\eta), \xi \rangle_{L^2_p} \big| &\leq C \left(\int_{\mathbb{R}^3}dp \int_{\mathbb{R}^3}dq\int_{\mathbb{S}^2}dw~v_{\o}\sigma(g,\theta)J(p'^0/2)J(q'^0/2) |f(p)|^2|h(p')|^2\right)^{\frac{1}{2}} \cr
&\times \left(\int_{\mathbb{R}^3}dp \int_{\mathbb{R}^3}dq\int_{\mathbb{S}^2}dw~v_{\o}\sigma(g,\theta)J(p'^0/2)J(q'^0/2)|\xi(p)|^2|\eta(q')|^2\right)^{\frac{1}{2}}\cr
&= T_{31} \times T_{32}.
\end{align*}Via the change of variables $p' \leftrightarrow q'$, we can see that $T_{31}$ and $T_{32}$ are indeed identical. Thus we only consider $T_{31}$ part here. Similarly to the estimates of $\Gamma_2$, we rewrite $(T_{31})^2$ as follows:
\begin{multline*}
\int_{\mathbb{R}^3}\frac{dp}{p^0} \int_{\mathbb{R}^3}\frac{dq}{q^0}\int_{\mathbb{R}^3}\frac{dp'}{p'^0}\int_{\mathbb{R}^3}\frac{dq'}{q'^0} s\sigma(g,\theta)J(p'^0/2)J(q'^0/2)|f(p)|^2|h(p')|^2\delta^{(4)}(p^{\mu}+q^{\mu}-q'^{\mu}-q'^{\mu}).
\end{multline*}
Then we note that the energy conservation law and the exponential decays with respect to the two post-collisional momenta can together imply the decays with respect to all pre-post collisional variables:
\begin{align*}
J(p'^0/2)J(q'^0/2)=J(p^0/4)J(q^0/4)J(p'^0/4)J(q'^0/4).
\end{align*}
Then, we use the boundedness of $B$ in \eqref{B esti} as 
\begin{align*}
B&= \int_{\mathbb{R}^3}\frac{dq}{q^0}\int_{\mathbb{R}^3}\frac{dq'}{q'^0} s\sigma(g,\theta)J(q^0/4) \delta^{(4)}(p^{\mu}+q^{\mu}-q'^{\mu}-q'^{\mu}) \leq C \bar{s}^{1/2}(p^0+p'^0),
\end{align*}
which yields
\begin{align*}
(T_{31})^2	\leq C \int_{\mathbb{R}^3}\frac{dp}{p^0} \int_{\mathbb{R}^3}\frac{dp'}{p'^0}\bar{s}^{1/2}(p^0+p'^0)J(p^0/4)J(p'^0/4)|f(p)|^2|h(p')|^2.
\end{align*}
Using $\bar{s} \leq C p^0p'^0$ and $(p^0)^{1/2}J(p^0/4)\leq C$, we have
\begin{align*}
(T_{31})^2	\leq C \int_{\mathbb{R}^3}dp|f(p)|^2 \int_{\mathbb{R}^3}dp'|h(p')|^2 \leq C\|f\|_{L^2_p}^2\|h\|_{L^2_p}^2 .
\end{align*}
Since $T_{32}$ has the same form with $T_{31}$, we have
\begin{align*}
\big| \langle T_3(f,h,\eta), \xi \rangle_{L^2_p} \big| \leq C \|f\|_{L^2_p} \|h\|_{L^2_p} \|\eta\|_{L^2_p} \| \xi \|_{L^2_p}.
\end{align*}

{\bf (3) Estimates of $T_4$: }
Recall the definition of $T_4$ in \eqref{T}. Using Lemma \ref{mJ}, we can have $CJ(p^0/2) \leq \sqrt{m+\tau m^2(p)} \leq CJ(p^0/2)$, which yields 
\begin{align*}
|T_4(f,h,\eta)|&\leq C\int_{\mathbb{R}^3}dq\int_{\mathbb{S}^2}dw~ v_{\o}\sigma(g,\theta)  \frac{J(q^0/2)J(p'^0/2)J(q'^0/2)}{J(p^0/2)} |f(q)||h(p')||\eta(q')|.
\end{align*}
Using the energy conservation law $J(p'^0/2)J(q'^0/2)=J(p^0/2)J(q^0/2)$, we have 
\begin{align*}
|T_4(f,h,\eta)|&\leq \int_{\mathbb{R}^3}dq\int_{\mathbb{S}^2}dw~ v_{\o}\sigma(g,\theta)  J(q^0) |f(q)||h(p')||\eta(q')|.
\end{align*}
Then we take the inner product with $\xi$ to have  
\begin{align*}
\big| \langle T_4(f,h,\eta), \xi \rangle_{L^2_p} \big| 
&\leq C \int_{\mathbb{R}^3}dp \int_{\mathbb{R}^3}dq\int_{\mathbb{S}^2}dw~v_{\o}\sigma(g,\theta)J(q^0) |f(q)||h(p')||\eta(q')||\xi(p)|.
\end{align*}
Via the H\"{o}lder inequality, we divide the variable into two pre-collisional variables and two post-collisional variables as follows: 
\begin{align*}
\big| \langle T_4(f,h,\eta), \xi \rangle_{L^2_p} \big| 
&\leq C \left(\int_{\mathbb{R}^3}dp \int_{\mathbb{R}^3}dq\int_{\mathbb{S}^2}dw~v_{\o}\sigma(g,\theta)J(q^0) |f(q)|^2|\xi(p)|^2\right)^{\frac{1}{2}} \cr
&\times \left(\int_{\mathbb{R}^3}dp \int_{\mathbb{R}^3}dq\int_{\mathbb{S}^2}dw~v_{\o}\sigma(g,\theta)J(q^0) |h(p')|^2|\eta(q')|^2\right)^{\frac{1}{2}} \cr
&= T_{41} \times T_{42}.
\end{align*}
For the estimates of $T_{41}$, we use $v_{\o}\sigma(g,\theta) \leq C (p^0)^{1/2}(q^0)^{1/2}$ to have
\begin{align*}
(T_{41})^2&\leq C \int_{\mathbb{R}^3}dp \int_{\mathbb{R}^3}dq\int_{\mathbb{S}^2}dw~(p^0)^{\frac{1}{2}}(q^0)^{\frac{1}{2}}J(q^0) |f(q)|^2|\xi(p)|^2
%&\leq C \int_{\mathbb{S}^2}\int_{\mathbb{R}^3} (p^0)^{\frac{1}{2}} |\xi(p)|^2dp\int_{\mathbb{R}^3}|\eta(q)|^2 dq dw  \cr
\leq C \|f\|_{L^2_p}^2\|\xi\|_{\nu}^2,
\end{align*}
where we used $(q^0)^{\frac{1}{2}}J(q^0) \leq C$. 
Similarly, we have 
\begin{align*}
(T_{42})^2 &\leq C \int_{\mathbb{R}^3}dp \int_{\mathbb{R}^3}dq\int_{\mathbb{S}^2}dw~ (p^0)^{\frac{1}{2}} |h(p')|^2|\eta(q')|^2.
\end{align*}
We use $(p^0)^{1/2}\leq (p'^0)^{1/2}+(q'^0)^{1/2}$ to have 
\begin{align*}
(T_{42})^2 &\leq C \int_{\mathbb{R}^3}dp \int_{\mathbb{R}^3}dq\int_{\mathbb{S}^2}dw~\left((p'^0)^{\frac{1}{2}}+(q'^0)^{\frac{1}{2}}\right) |h(p')|^2|\eta(q')|^2.
\end{align*}
Then applying the pre-post collisional change of variables $(p,q) \leftrightarrow (p',q')$ as in \eqref{pq,p'q'} yields 
\begin{align*}
T_{42}  &\leq C \left(\| h \|_{L^2_p} \| \eta \|_{\nu}  + \| h \|_{\nu} \| \eta \|_{L^2_p} \right).
\end{align*}
Combining the estimates of $T_{41}$ and $T_{42}$ yields 
\begin{align*}
\big| \langle T_4(f,h,\eta), \xi \rangle_{L^2_p} \big|&\leq C \left(\| h \|_{L^2_p} \| \eta \|_{\nu}  + \| h \|_{\nu} \| \eta \|_{L^2_p} \right)\|f\|_{L^2_p}\|\xi\|_{\nu}.
\end{align*}
\end{proof}

\section{Local existence}\label{sec:localintime}
In this section, we construct the local-in-time classical solution. Here we briefly introduce a main difference of the quantum Boltzmann theory from the Newtonian one. 
From the statistical description of quantum Boltzmann equation in \cite{MR0258399}, we denote that the term $(1-F(p))$ in the collision operator for fermions is the probability that a fermion is being placed in the $p$ momentum place after a collision. Thus the ratio $(1-F(p))$ has to be non-negative. Then, in each iteration scheme in the proof of the local existence, the collision operator depending on $F^{n+1}$ is well-defined only when $F^{n+1}$ is in the interval $[0,1]$ in the case of fermions. Thus we need to prove the boundedness of the solution $F^{n+1}$ in each iteration step as well. 
\begin{theorem}\label{Local}
Let $N\geq 3$. Suppose that the initial data $F_0$ satisfies 
\begin{align*}
\left\{\begin{array}{ll} 0 \leq F_0(x,p)=m(p)+\sqrt{m(p)-m^2(p)}f_0(x,p) \leq 1, \quad \mbox{for fermion} \\ 0 \leq F_0(x,p)=m(p)+\sqrt{m(p)+m^2(p)}f_0(x,p), \quad \hspace{6mm} \mbox{for boson} , \end{array}  \right.
\end{align*}
Then there exist $M_0>0$ and $T_*>0$ such that if $\mathcal{E}(f_0)\leq \frac{M_0}{2}$ then there exists a unique local-in-time solution f(x,p,t) of \eqref{pertf} satisfying
\begin{enumerate}
\item The higher order energy is uniformly bounded:
\[\sup_{0 \leq t\leq T_*}\mathcal{E}(f(t))\leq M_0.\]
\item The distribution function is bounded in $t\in[0,T_*]$:
\begin{align*}
\left\{\begin{array}{ll} 0\leq F(x,p,t)=m(p)+\sqrt{m(p)-m^2(p)}f(x,p,t) \leq 1, \quad \mbox{for fermions} \\ 0 \leq F(x,p,t)=m(p)+\sqrt{m(p)+m^2(p)}f(x,p,t), \quad \hspace{6mm} \mbox{for bosons} , \end{array}  \right.
\end{align*}
%\item The solution of perturbed relativistic Fermi-Dirac equation \eqref{pertf} satisfies the conservation laws \eqref{consf}.
\item The higher order energy norm $\mathcal{E}(f(t))$ is continuous in $t\in[0,T_*]$.
\end{enumerate}
\end{theorem}
\begin{proof}
We first take an iteration scheme as follows:
\begin{align}\label{iter}
	(\partial_t +\hat{p}\cdot \nabla_x)F^{n+1}=Q(F^n,F^n,F^{n+1},F^n),
\end{align}
with $F^0(x,p,t)=F_0(x,p)$. Note that $F^{n+1}$ in the operator $Q$ is placed in the $p$ variable location. We proceed the proof by applying the induction argument. Let us assume that 
\begin{align}\label{assume}
\left\{\begin{array}{ll}0 \leq F^n(x,p,t) \leq 1 \quad \mbox{for fermions,} \\ 0 \leq F^n(x,p,t) \qquad \hspace{3mm} \mbox{for bosons,} \end{array}  \ \text{and}\  \sup_{0 \leq t\leq T_*}\mathcal{E}(f^n(t))\leq M_0. \right .
\end{align}
Note that the collision operator $Q(F^n,F^n,F^{n+1},F^n)$ is well-defined when $F^{n+1}$ is in the interval $[0,1]$. We define
\begin{align*}
G(F)&=\int_{\mathbb{R}^3}dq\int_{\mathbb{S}^2}dw~v_{\o}\sigma(g,\theta)F(p')F(q')(1+\tau F(q)), \cr
R(F)&=\int_{\mathbb{R}^3}dq\int_{\mathbb{S}^2}dw~v_{\o}\sigma(g,\theta)(1+\tau F(p'))(1+\tau F(q'))F(q),
\end{align*}
to have
\begin{align*}
	Q(F^n,F^n,F^{n+1},F^n)=G(F^n)(1+\tau F^{n+1}(p))-R(F^n)F^{n+1}(p).
\end{align*}
Then we rewrite \eqref{iter} as follows:
\begin{align}\label{RG}
(\partial_t +\hat{p}\cdot \nabla_x-\tau G(F^n)+R(F^n))F^{n+1}=G(F^n).
\end{align}
In the case of fermions, we additionally consider the upper bound of $F^{n+1}$. We observe that the induction hypothesis $(\ref{assume})_1$ implies
\begin{align*}
	0\leq G(F^n), \quad 0 \leq R(F^n).
\end{align*}
Since $R(F^n)$ is positive, we have
\begin{align*}
(\partial_t +\hat{p}\cdot \nabla_x+G(F^n)+R(F^n))F^{n+1}\leq G(F^n)+R(F^n).
\end{align*}The associated ODE for the particle characteristic trajectory is given by $dX(s)/ds = \hat{p}(s)$ where $X(s)=X(s;t,x,p)$.
We integrate over the particle path to obtain
%\begin{multline*}
%F^{n+1}(x,p,t)\leq e^{-\int_0^t(G(F^n)+R(F^n))d\tau}F_0(x-pt,p) \cr
%+\int_0^t e^{-\int_s^t(G(F^n)+R(F^n))d\tau}(G(F^n)+R(F^n))(x+(s-t) p,p,s)ds.
%\end{multline*}
\begin{multline*}
	F^{n+1}(X(t),p,t)\leq e^{-\int_0^t(G(F^n)+R(F^n))d\tau}F_0(X(0),p) \cr
	+\int_0^t e^{-\int_s^t(G(F^n)+R(F^n))d\tau}(G(F^n)+R(F^n))(X(s),p,s)ds.
\end{multline*} We observe from 
\begin{align*}
	\frac{d}{ds} \left\{e^{-\int_s^t(G(F^n)+R(F^n))d\tau}\right\} =  e^{-\int_s^t(G(F^n)+R(F^n))d\tau}(G(F^n)+R(F^n))(X(s),p,s),
\end{align*}
\iffalse
\begin{multline*}
	F^{n+1}(x,p,t)\leq e^{-\int_0^t(G(F^n)+R(F^n))d\tau}F_0(x-\hat{p}t,p) \cr
	+\int_0^t e^{-\int_s^t(G(F^n)+R(F^n))(x-\hat{p}(t-\tau),p,\tau)d\tau}(G(F^n)+R(F^n))(x-\hat{p}(t-s),p,s)ds.
\end{multline*}
We observe from 
\begin{multline*}
\frac{d}{ds} \left\{e^{-\int_s^t(G(F^n)+R(F^n))(x-\hat{p}(t-\tau),p,\tau)d\tau}\right\} \cr
=  e^{-\int_s^t(G(F^n)+R(F^n))(x-\hat{p}(t-\tau),p,\tau)d\tau}(G(F^n)+R(F^n))(x-\hat{p}(t-s),p,s)
\end{multline*}
\fi
that
\begin{align*}
	F^{n+1}(x,p,t) &\leq e^{-\int_0^t(G(F^n)+R(F^n))d\tau}F_0(x-\hat{p}t,p)+1-e^{-\int_0^t(G(F^n)+R(F^n))d\tau} \cr
	&=e^{-\int_0^t(G(F^n)+R(F^n))d\tau}(F_0(x-\hat{p}t,p)-1)+1 \cr
	&\leq 1,
\end{align*}
where we used the boundedness of the initial data $F_0 \leq 1 $. 

We now consider the lower bound of $F^{n+1}$. By $G(F^n)\geq 0 $ on \eqref{RG}, we have 
\begin{align*}
	(\partial_t +\hat{p}\cdot \nabla_x-\tau G(F^n)+R(F^n))F^{n+1}\geq 0,
\end{align*}
which yields
\begin{align*}
	F^{n+1}(x,p,t) \geq  e^{-\int_0^t(-\tau G(F^n)+R(F^n))dt}F_0(x-\hat{p}t,p) \geq 0,
\end{align*}
where we used the boundedness of the initial data $F_0\geq 0 $. Now the collision operator of \eqref{iter} is well-defined.

 Substituting $F^{n+1}=m+\sqrt{m+\tau m^2}f^{n+1}$ on \eqref{iter} gives
\begin{align*}
	(\partial_t +\hat{p}\cdot \nabla_x+\nu)f^{n+1}=Kf^n+\Gamma(f^n,f^{n+1}),
\end{align*}
where $K=K_2-K_1$ and 
\begin{align*}
	\Gamma(f^n,f^{n+1})&= \sum_{i=1,2,4}\Gamma_i(f^{n+1},f^n)+\sum_{i=3,5,6}\Gamma_i(f^n,f^n) \cr
	&\quad +\sum_{i=1,2,3}T_i(f^{n+1},f^n,f^n)+T_4(f^n,f^n,f^n).
\end{align*}
We take $\partial^{\alpha}$ on each side to have
\begin{align*}
	\partial_t\partial^{\alpha}f^{n+1}+\hat{p}\cdot\nabla_x\partial^{\alpha}f^{n+1}+\partial^{\alpha}(\nu f^{n+1})&=\partial^{\alpha}Kf^n+\partial^{\alpha}\Gamma(f^n,f^{n+1}).
\end{align*}
We take the $L^2_{x,p}$ inner product with $\partial^{\alpha}f^{n+1}$, then the nonlinear estimates Lemma \ref{nonlin ff} and Lemma \ref{nonlin fff} with the induction hypothesis $(\ref{assume})_2$ yield
\begin{align*}
	(1-C\sqrt{M_0}-CM_0-CT_*^2M_0 -CM_0^2-CM_0^3)\sup_{0\leq t \leq T_*}\mathcal{E}_{n+1}(t) \leq \frac{M_0}{2}
\end{align*}
For sufficiently small $T_*$ and $M_0$, we have
\begin{align*}
	\sup_{0\leq t \leq T_*}\mathcal{E}_{n+1}(t) \leq M_0.
\end{align*}
Taking the limit as $n\rightarrow \infty$ gives a local-in-time classical solution. 
The remaining proof is standard as in \cite{MR1908664,MR2000470} and we omit it.
\end{proof}

\section{Global existence}
In this section, we extend the local solution constructed in Theorem \ref{Local} to a global solution. For this we first recover the full coercivity estimates of the linear operator $L$.
\subsection{Coercivity estimate}
Recall the definition of $Pf$ in Lemma \ref{null space}. Since $Pf$ is the orthonormal projection to the $L^2_p$ space with the following basis, 
\begin{align*}
\left\{\sqrt{m+\tau m^2},p_1\sqrt{m+\tau m^2},p_2\sqrt{m+\tau m^2},p_3\sqrt{m+\tau m^2},p^0\sqrt{m+\tau m^2}\right\},
\end{align*}
$Pf$ can be written as follows by the Gram-Schmidt process:  
\begin{align*}
Pf=\mathcal{A}\sqrt{m+\tau m^2}+\mathcal{B}\cdot p\sqrt{m+\tau m^2}+\mathcal{C}p^0\sqrt{m+\tau m^2},
\end{align*}
where $\mathcal{A}$, $\mathcal{B}$ and $\mathcal{C}$ are given by 
\begin{align*}
	\mathcal{A} &= \frac{1}{\lambda}\int_{\mathbb{R}^3}dp~f\sqrt{m+\tau m^2}-\frac{\lambda_0}{\lambda}\frac{1}{\lambda_{00}-\frac{\lambda_0^2}{\lambda}}\left(\int_{\mathbb{R}^3}dp~fp^0\sqrt{m+\tau m^2}-\frac{\lambda_0}{\lambda}\int_{\mathbb{R}^3}dp~f\sqrt{m+\tau m^2}\right), \cr
	\mathcal{B}_i&= \frac{1}{\lambda_i}\int_{\mathbb{R}^3}dp~fp^i\sqrt{m+\tau m^2},\cr
	\mathcal{C}&= \frac{1}{\lambda_{00}-\frac{\lambda_0^2}{\lambda}}\left(\int_{\mathbb{R}^3}dp~fp^0\sqrt{m+\tau m^2}-\frac{\lambda_0}{\lambda}\int_{\mathbb{R}^3}dp~f\sqrt{m+\tau m^2}\right),
\end{align*}
and
\begin{align*}
\lambda &= \int_{\mathbb{R}^3}dp~m(p)+\tau m^2(p), \qquad \hspace{6mm}
\lambda_i = \int_{\mathbb{R}^3}dp~(p^i)^2(m(p)+\tau m^2(p)), \cr
\lambda_0 &= \int_{\mathbb{R}^3}dp~p^0(m(p)+\tau m^2(p)), \qquad 
\lambda_{00} = \int_{\mathbb{R}^3}dp~(p^0)^2(m(p)+\tau m^2(p)),
\end{align*}
for $i=1,2,3$. Since $Pf$ has the exponential decay $\sqrt{m+\tau m^2}$, we can have
\begin{align}\label{PfABC}
\sum_{|\alpha|\leq N}\| \partial^{\alpha} Pf \|_{x,\nu} &\leq \sum_{|\alpha|\leq N} \left(\| \partial^{\alpha}\mathcal{A} \|_{L^2_x}+\| \partial^{\alpha}\mathcal{B} \|_{L^2_x}+\| \partial^{\alpha}\mathcal{C} \|_{L^2_x}\right).
\end{align}
Now we substitute $f= Pf+(I-P)f$ in the perturbation equation \eqref{pertf} to have
\begin{align*}
(\partial_t+\hat{p}\cdot \nabla_x)(Pf) = -(\partial_t+\hat{p}\cdot \nabla_x)((I-P)f)-L(I-P)f+\Gamma(f)+T(f),
\end{align*}
where we used $L(Pf)=0$ by Lemma \ref{null space}. We expand the left-hand side as follows: 
\begin{align*}
\left(\partial_t\mathcal{A} +\sum_{i=1}^3\partial_{x_i}\mathcal{A}\frac{p_i}{p^0} + \sum_{i=1}^3(\partial_t\mathcal{B}_i+\partial_{x_i}\mathcal{C})p_i+\sum_{1 \leq i,j \leq 3}\partial_{x_i}\mathcal{B}_j\frac{p_ip_j}{p^0} + \partial_t\mathcal{C}p^0 \right) \sqrt{m+\tau m^2},
\end{align*}
which is a linear combination of following $14$-basis:
\begin{align}\label{basis}
\left\{\sqrt{m+\tau m^2},\quad \frac{p_i}{p^0}\sqrt{m+\tau m^2},\quad p_i\sqrt{m+\tau m^2},\quad \frac{p_ip_j}{p^0}\sqrt{m+\tau m^2},\quad p^0\sqrt{m+\tau m^2} \right\},
\end{align}
for $1\leq i,j \leq 3$. To denote the right-hand side, we define  
\begin{align}\label{l,h}
l&=-(\partial_t+\hat{p}\cdot \nabla_x+L)((I-P)f), \qquad 
h=\Gamma(f)+T(f).
\end{align}
By expanding $l$ and $h$ with respect to the $14$-basis elements in \eqref{basis}, we obtain the following macro-micro system:
\begin{align}\label{system}
\begin{split}
\partial_t\mathcal{A}&= l_a+h_a,\cr
\partial_{x_i}\mathcal{A}&=l_i+h_i, \cr \partial_t\mathcal{B}_i+\partial_{x_i}\mathcal{C}&= l_{bci}+h_{bci},\cr
\partial_{x_i}\mathcal{B}_j+\partial_{x_j}\mathcal{B}_i &= l_{ij}+h_{ij},\cr
\partial_t\mathcal{C}&= l_c+h_c, 
\end{split}
\end{align}
for $i,j=1,\cdots,3$. On the right-hand side, $l_a, l_i, l_{bci}, l_{ij}, l_c$ and $h_a, h_i, h_{bci}, h_{ij}, h_c$ are the coefficients of expansion of $l$ and $h$ with respect to the basis in \eqref{basis}, respectively. We define the summation of the coefficients of $l$ and $h$ as follows: 
\begin{align*}
\tilde{l} &= l_a + l_c + \sum_{1\leq i \leq 3}\left(l_i + l_{bci}\right) + \sum_{1\leq i,j \leq 3}l_{ij}, \cr
\tilde{h} &= h_a + h_c + \sum_{1\leq i \leq 3}\left(h_i + h_{bci}\right) + \sum_{1\leq i,j \leq 3}h_{ij}.
\end{align*}
Note that the macro-micro system of \eqref{system} is identical to the macro-micro system of the relativistic Landau-Maxwell model (98)-(102) in \cite{MR2100057} when the electromagnetic field is zero $B=E=0$. However, we provide the full proof here for the readers' convenience.
We first observe following property of $\mathcal{A}$, $\mathcal{B}$ and $\mathcal{C}$.
The conservation laws of $f$ in \eqref{consf} can be written in the following form. 
\begin{lemma}\label{ABC0} We have 
\begin{align*}
\int_{\mathbb{T}^3}\mathcal{A}(x,t)dx= \int_{\mathbb{T}^3}\mathcal{B}(x,t)dx= \int_{\mathbb{T}^3}\mathcal{C}(x,t)dx= 0.
\end{align*}
\end{lemma}
\begin{proof}
The conservation laws in \eqref{NPE0} implies
\begin{align*}
\int_{\mathbb{T}^3\times \mathbb{R}^3} dxdp~(F-F_0)\left( \begin{array}{c} 1 \cr p^{\mu} \end{array}\right)  &= 0.
\end{align*}
Combining with the assumption \eqref{assumption}, we have 
\begin{align*}
\int_{\mathbb{T}^3\times \mathbb{R}^3} dxdp~f\sqrt{m+\tau m^2}\left( \begin{array}{c} 1 \cr p^{\mu} \end{array}\right) = 0.
\end{align*}
We derived desired results.
\end{proof}
We now establish the estimates of $\mathcal{A}, \mathcal{B}, \mathcal{C}$. The proof of the following lemma is motivated by the proof of the estimate of (108) in \cite{MR2100057}.
\begin{lemma}\label{ABC} Let $N \geq 3$. We have 
\begin{align*}
\sum_{|\alpha| \leq N}\left\{\|\partial^{\alpha}\mathcal{A}\|_{L^2_x}+\|\partial^{\alpha}\mathcal{B}\|_{L^2_x}+\|\partial^{\alpha}\mathcal{C}\|_{L^2_x}\right\} \leq C\sum_{|\alpha|\leq N-1}\left( \|\partial^{\alpha}\tilde{l}\|_{L^2_x}+\|\partial^{\alpha}\tilde{h}\|_{L^2_x}\right).
\end{align*}
\end{lemma}
\begin{proof}
We first prove the estimate of $\mathcal{A}$. Taking $\partial^{\alpha}$ on the first equation of \eqref{system} gives 
\begin{align*}
	\partial^{\alpha}\partial_t\mathcal{A}&= \partial^{\alpha}l_a+\partial^{\alpha}h_a.
\end{align*}
we take the innerproduct with $\partial^{\alpha}\partial_t\mathcal{A}$ and apply the H\"{o}lder inequality to have
\begin{align*}
\|\partial^{\alpha}\partial_t \mathcal{A}\|_{L_{x}^2}^2 \leq C\left(\|\partial^{\alpha}l_a\|_{L^2_x}\|\partial^{\alpha}\partial_t \mathcal{A}\|_{L_{x}^2}+\|\partial^{\alpha}h_a\|_{L^2_x}\|\partial^{\alpha}\partial_t \mathcal{A}\|_{L_{x}^2}\right).
\end{align*} 
Dividing each side by $\|\partial^{\alpha}\partial_t \mathcal{A}\|_{L_{x}^2}$, we have
\begin{align*}
\|\partial^{\alpha}\partial_t \mathcal{A}\|_{L_{x}^2} \leq C(\|\partial^{\alpha}l_a\|_{L^2_x}+\|\partial^{\alpha}h_a\|_{L^2_x}).
\end{align*}
Similarly we take $\partial^{\alpha}$ on the second equation of \eqref{system} to have
\begin{align*}
	\partial^{\alpha}\partial_{x_i}\mathcal{A}&=\partial^{\alpha}l_i+\partial^{\alpha}h_i.
\end{align*}
Multiplying $\partial^{\alpha}\partial_{x_i}\mathcal{A}$ on each side and applying the H\"{o}lder inequality, we have
\begin{align*}
\|\partial^{\alpha}\nabla_x\mathcal{A}\|_{L_{x}^2}\leq C  (\|\partial^{\alpha}l_i\|_{L^2_x}+\|\partial^{\alpha}h_i\|_{L^2_x}).
\end{align*}
We apply the Poincar\'{e} inequality on $\mathcal{A}$ to have
\begin{align*}
	\|\mathcal{A}\|_{L_{x}^2}- \left\|\frac{1}{|\mathbb{T}^3|}\int_{\mathbb{T}^3}\mathcal{A}(x)dx\right\|_{L_{x}^2} \leq C  \|\nabla_x \mathcal{A}\|_{L^2_x}.
\end{align*}
Then the conservation law for $\mathcal{A}$ in Lemma \ref{ABC0} implies 
\begin{align*}
	\|\mathcal{A}\|_{L_{x}^2}\leq C  \|\nabla_x \mathcal{A}\|_{L^2_x}.
\end{align*}
This completes the estimate of $\mathcal{A}$. 
Since $\mathcal{B}$ and $\mathcal{C}$ are connected by the third equation of \eqref{system}, we consider the fourth and the fifth equations of \eqref{system} first. For the notational brevity, we use $\partial_j$ to denote $\partial_{x_j}$. Using the fourth equation of \eqref{system}, we calculate
\begin{align*}
	\triangle \mathcal{B}_i = \sum_{1\leq j\leq 3}\partial_{jj}\mathcal{B}_i &= \sum_{j \neq i} \partial_{jj} \mathcal{B}_i +\partial_{ii} \mathcal{B}_i \cr
	&= \sum_{j \neq i} \left(\partial_{j}l_{ij} + \partial_{j}h_{ij}-\partial_{ji} \mathcal{B}_j\right) + \frac{1}{2}(\partial_{i}l_{ii} + \partial_{i}h_{ii}).
\end{align*}
We substitute the fourth equation for $i=j$ case to have 
\begin{align*}
	\triangle \mathcal{B}_i &= \sum_{j \neq i} \left(\partial_{j}l_{ij} + \partial_{j}h_{ij}-\frac{1}{2}(\partial_{i}l_{jj} +\partial_{i}h_{jj})\right) + \frac{1}{2}(\partial_{i}l_{ii} + \partial_{i}h_{ii}).
\end{align*}
Taking $\partial^{\alpha}$ and multiplying $\partial^{\alpha}\mathcal{B}_i$ on each side yield
\begin{align}\label{Bx}
	\|\partial^{\alpha}\nabla_x \mathcal{B}\|_{L^2_x} \leq C \sum_{1 \leq i,j \leq 3} \left(\|\partial^{\alpha}l_{ij}\|_{L^2_x}+\|\partial^{\alpha}h_{ij}\|_{L^2_x}\right).
\end{align}
We take $\partial^{\alpha}$ on the fifth equation of \eqref{system} to have
\begin{align}\label{tC}
	\|\partial^{\alpha}\partial_t \mathcal{C}\|_{L_{x}^2} \leq C\|\partial^{\alpha}l_c\|_{L^2_x}+\|\partial^{\alpha}h_c\|_{L^2_x}.
\end{align}
Now we use the third equation of \eqref{system}. We consider the estimate of $\mathcal{B}$ when there are more than one temporal derivatives. For this, we take the temporal derivative $\partial_t^{n}$ for $1\leq n\leq N-1 $ on the third equation of \eqref{system} to have 
\begin{align*}
	\partial_t^{n}\partial_t\mathcal{B}_i&= \partial_t^{n}l_{bci}+\partial_t^{n}h_{bci}-\partial_t^{n}\partial_{x_i}\mathcal{C}.
\end{align*}
Then the estimate \eqref{tC} gives 
\begin{align}\label{Bt}
	\|\partial_t^{n+1}\mathcal{B}_i\|_{L^2_x}& \leq \|\partial_t^{n}l_{bci}\|_{L^2_x}+\|\partial_t^{n}h_{bci}\|_{L^2_x} + C\left(\|\partial_{x_i}\partial_t^{n-1}l_c\|_{L^2_x}+\|\partial_{x_i}\partial_t^{n-1}h_c\|_{L^2_x}\right).
\end{align}
For the estimate of $\mathcal{B}$ when there are less than two temporal derivatives, we apply the Poincar\'{e} inequality on $\mathcal{B}$ and $\partial_t\mathcal{B}$ to have
\begin{align*}
	\|\partial_t^n\mathcal{B}_i\|_{L^2_x}-\left\|\frac{1}{|\mathbb{T}^3|}\partial_t^n\int_{\mathbb{T}^3}  \mathcal{B}_i dx\right\| &\leq \|\nabla_x \partial_t^n \mathcal{B}_i\|_{L^2_x}, 
\end{align*}
for $n=0,1$. Then the conservation law in Lemma \ref{ABC0} yields 
\begin{align*}
	\|\partial_t^n\mathcal{B}_i\|_{L^2_x}&\leq \|\nabla_x \partial_t^n \mathcal{B}_i\|_{L^2_x}.
\end{align*}
Combining with \eqref{Bx} and \eqref{Bt}, we conclude that
\begin{align}\label{BT}
\sum_{|\alpha| \leq N}\|\partial^{\alpha}\mathcal{B}\|_{L^2_x} \leq C\sum_{|\alpha|\leq N-1}\left( \|\partial^{\alpha}\tilde{l}\|_{L^2_x}+\|\partial^{\alpha}\tilde{h}\|_{L^2_x}\right).
\end{align}
For the estimate of $\mathcal{C}$, we take $\partial^{\alpha}$ on third equation of \eqref{system} to have
\begin{align*}
	\partial^{\alpha}\partial_{x_i}\mathcal{C}&= \partial^{\alpha}l_{bci}+\partial^{\alpha}h_{bci}-\partial^{\alpha}\partial_t\mathcal{B}_i.
\end{align*}
Applying \eqref{BT} yields 
\begin{align}\label{xC}
\|\partial^{\alpha}\partial_{x_i}\mathcal{C}\|_{L^2_x}&\leq C\sum_{|\beta|= |\alpha|}\left( \|\partial^{\beta}\tilde{l}\|_{L^2_x}+\|\partial^{\beta}\tilde{h}\|_{L^2_x}\right).
\end{align}
Then the Poincar\'{e} inequality with Lemma \ref{ABC0} gives
\begin{align*}
\|\mathcal{C}\|_{L_{x}^2}\leq C  \|\nabla_x \mathcal{C}\|_{L^2_x}.
\end{align*}
Combining with the estimates \eqref{tC} and \eqref{xC}, we derive the desired result.
\end{proof}

\begin{lemma}\label{lh esti} Let $N\geq 3$. Suppose that  
\begin{align*}
\sum_{|\alpha|\leq N}\|\partial^{\alpha}f\|_{L^2_{x,p}}^2 \leq M_0 .
\end{align*}
Then we have
\begin{align*}
&(1) \ \sum_{|\alpha|\leq N-1} \|\partial^{\alpha}\tilde{l}\|_{L^2_x} \leq C \sum_{|\alpha|\leq N}\|(I-P)\partial^{\alpha}f\|_{x,\nu}, \cr 
&(2) \ \sum_{|\alpha|\leq N} \|\partial^{\alpha}\tilde{h}\|_{L^2_x} \leq C(\sqrt{M_0}+M_0) \sum_{|\alpha|\leq N}\|\partial^{\alpha}f\|_{x,\nu}.
\end{align*}
\end{lemma}
\begin{proof}
For the Newtonian Boltzmann equation and the relativistic Landau-Maxwell system, analogous proofs can be found in \cite{MR2000470} and \cite{MR2100057}, respectively. The difference now is that the estimate of the nonlinear term $h$ includes the third-order nonlinear terms. \newline
(1) We denote the $14$-basis of \eqref{basis} as $\{e_i\}_{1\leq i \leq 14}$, and let $\{e_i^*\}_{1\leq i \leq 14}$ be the corresponding orthonormal basis. Then the orthonormal basis can be written by a linear combination of the original basis $\{e_i\}_{1\leq i \leq 14}$ as follows: 
\begin{align*}
e_i^* = \sum_{j=1}^{14} C_{ij}e_j,
\end{align*}
for $j=1,\cdots,14$. We consider the orthonormal expansion of $l$ as follows:
\begin{align*}
l= \sum_{i=1}^{14}\langle l,e_i^*\rangle_{L^2_p}e_i^*  = \sum_{i=1}^{14}\left\langle l, \sum_{j=1}^{14} C_{ij}e_j\right\rangle_{L^2_p} \sum_{k=1}^{14} C_{ik}e_k.
\end{align*}
Then the coefficient of $e_k$ can be read as follows:
\begin{align*}
\sum_{1\leq i,j\leq 14} C_{ij}C_{ik}\langle l, e_j\rangle_{L^2_p},
\end{align*}
which correspond to $l_a, l_i, l_{bci}, l_{ij}$, and $l_c$. 
By the definition of $l$ in \eqref{l,h} and the linear operator $L$ in Proposition \ref{linearization}, we can write $l$ as
\begin{align*}
l&=-(\partial_t+\hat{p}\cdot \nabla_x+\nu+K_1-K_2)((I-P)f).
\end{align*}
For $|\alpha|=N-1$, we have
%\begin{align*}
%\bigg\|\int_{\mathbb{R}^3}\partial^{\alpha} l \cdot e_i(p) dp  \bigg\|_{L^2_x}  = \bigg\|\int_{\mathbb{R}^3}\partial^{\alpha} (\partial_t+\hat{p}\cdot \nabla_x+L)((I-P)f) \cdot e_i(p) dp  \bigg\|_{L^2_x}
%\end{align*}
\begin{multline*}
\bigg\|\int_{\mathbb{R}^3} dp~\partial^{\alpha} l \cdot e_i(p)  \bigg\|_{L^2_x} \cr
\leq  \bigg\|\int_{\mathbb{R}^3} dp~\bigg(\sum_{|\beta|= N}(1+\hat{p})(I-P)\partial^{\beta}f + (\nu+K_1-K_2)(I-P)\partial^{\alpha}f\bigg) \cdot e_i(p)   \bigg\|_{L^2_x}. 
\end{multline*}
We use the H\"{o}lder inequality on the $(1+\hat{p})(I-P)\partial^{\beta}f$ term and the $\nu(I-P)\partial^{\alpha}f$ term. Then we have
\begin{align*}
\int_{\mathbb{R}^3} dp~ \bigg(\sum_{|\beta|= N}(1+\hat{p})(I-P)\partial^{\beta}f + \nu(I-P)\partial^{\alpha}f\bigg) \cdot e_i(p)  \leq \sum_{|\alpha|\leq N}\|(I-P)\partial^{\alpha}f \|_{L^2_{x,p}} ,
\end{align*}
where we used the exponential decay $\sqrt{m+\tau m^2}$ of the basis $e_i(p)$ in \eqref{basis} that implies
\begin{align*}
\int_{\mathbb{R}^3}\left((1+\hat{p})^2+\nu(p)\right)e_i^2(p) dp \leq C.
\end{align*}
From the estimate of the compact operator $K_j$ from Lemma \ref{K esti}, we also have
\begin{align*}
	\langle	K_jf , e_i \rangle_{L^2_p}&\leq C \|f \|_{\nu},
\end{align*}
for $j=1,2$ and $i=1,\cdots,14$. The estimates above together yield
\begin{align*}
\bigg\|\int_{\mathbb{R}^3}\partial^{\alpha} l \cdot e_i(p) dp  \bigg\|_{L^2_x}  & \leq C \sum_{|\alpha| = N}\|(I-P)\partial^{\alpha}f \|_{L^2_{x,p}}+\|(I-P)\partial^{\alpha}f \|_{x,\nu} \cr
& \leq C \sum_{|\alpha| \leq N}\|(I-P)\partial^{\alpha}f \|_{x,\nu}.
\end{align*} 
(2) In the same manner, the coefficient of $e_k$ of expansion of $l$ can be written as
\begin{align*}
\sum_{1\leq i,j\leq 14} C_{ij}C_{ik}\langle h, e_i\rangle_{L^2_p}.
\end{align*}
For the second-order nonlinear terms, by Lemma \ref{nonlin ff}, we have
\begin{align*}
\big| \langle \partial^{\alpha} \Gamma(f,f) ,e_i \rangle_{L^2_{x,p}} \big| %&\leq C \sum_{|\alpha_1|+|\alpha_2|\leq|\alpha|} \big| \langle  \Gamma(\partial^{\alpha_1}f,\partial^{\alpha_2}f) ,e_i \rangle_{L^2_{x,p}} \big| 
\leq  C \sum_{|\alpha_1|+|\alpha_2|\leq|\alpha|}\|\partial^{\alpha_1}f\|_{L^2_{x,p}}\|\partial^{\alpha_2}f\|_{x,\nu}
\leq C \sqrt{M_0} \sum_{|\alpha_1|\leq|\alpha|} \|\partial^{\alpha_1}f\|_{x,\nu}.
\end{align*}
For the third-order nonlinear term, Lemma \ref{nonlin fff} yields
\begin{align*}
\big| \langle \partial^{\alpha} T(f,f,f) ,e_i \rangle \big| %&\leq C \sum_{|\alpha_1|+|\alpha_2|+|\alpha_3|\leq|\alpha|} \big| \langle  T(\partial^{\alpha_1}f,\partial^{\alpha_2}f,\partial^{\alpha_3}f) ,e_i \rangle \big| \cr
&\leq  C \sum_{|\alpha_1|+|\alpha_2|+|\alpha_3|\leq|\alpha|}\|\partial^{\alpha_1}f\|_{L^2_p}\|\partial^{\alpha_2}f\|_{L^2_p}\|\partial^{\alpha_3}f\|_{\nu}.
\end{align*}
Then the Sobolev embedding $H^2(\mathbb{T}^3)\subset\subset L^{\infty}(\mathbb{T}^3)$ implies 
\begin{align*}
\bigg\|\int_{\mathbb{R}^3}\partial^{\alpha} T(f,f,f) \cdot e_i(p) dp  \bigg\|_{L^2_x}  %&\leq  C\sum_{|\alpha_1|\leq|\alpha|} \|\partial^{\alpha_1}f\|_{L^2_{x,p}} \sum_{|\alpha_2|\leq|\alpha|} \|\partial^{\alpha_2}f\|_{L^2_{x,p}}\sum_{|\alpha_3|\leq|\alpha|} \|\partial^{\alpha_3}f\|_{x,\nu} \cr
&\leq CM_0\sum_{|\alpha_1|\leq|\alpha|}\|\partial^{\alpha_1}f\|_{x,\nu}.
\end{align*}
So we obtain the desired results.
\end{proof}
We now have all the estimates to recover the full coercivity. We combine \eqref{PfABC} with Lemma \ref{ABC} and Lemma \ref{lh esti} to have
\begin{align*}
\sum_{|\alpha|\leq N}\| \partial^{\alpha} Pf \|_{x,\nu} %&\leq \sum_{|\alpha|\leq N} \left(\| \partial^{\alpha}\mathcal{A} \|_{L^2_x}+\| \partial^{\alpha}\mathcal{B} \|_{L^2_x}+\| \partial^{\alpha}\mathcal{C} \|_{L^2_x}\right) \cr
&\leq \sum_{|\alpha|\leq N-1} \left( \| \partial^{\alpha}l \|_{L^2_x}+\| \partial^{\alpha}h \|_{L^2_x} \right)\cr
&\leq  C\sum_{|\alpha|\leq N}\left(\|(I-P)\partial^{\alpha}f\|_{x,\nu}+C\sqrt{M_0}\|\partial^{\alpha}f\|_{x,\nu}\right).
\end{align*}
Thus, for a sufficiently small $M_0$, there exists $\delta>0$ such that  
\begin{align}\label{full coer}
\sum_{|\alpha|\leq N}\langle L\partial^{\alpha}f, \partial^{\alpha}f\rangle_{L^2_{x,p}} &\geq \delta \sum_{|\alpha|\leq N}\|\partial^{\alpha}f\|_{x,\nu}^2,
\end{align}by Lemma \ref{coercivity}.

\subsection{Global existence} We now have all the ingredients for the proof of Theorem \ref{Main Theorem}. Extending the local solution constructed in Theorem \ref{Local} to a global solution is standard as in \cite{MR2000470,MR2095473}. We only sketch the proof here. Recall that substituting $F=m+\sqrt{m+\tau m^2}f$ on \eqref{RQBE} yields the following linearized equation: 
\begin{align*}
\partial_tf+\hat{p}\cdot\nabla_xf +Lf&= \Gamma(f)+T(f), \cr
f(x,p,0) &= f_0(x,p). 
\end{align*}
We take $\partial^{\alpha}$ on each side to have
\begin{align*}
\partial_t\partial^{\alpha}f+\hat{p}\cdot\nabla_x\partial^{\alpha}f+L\partial^{\alpha}f&= \partial^{\alpha}\Gamma(f)+\partial^{\alpha}T(f).
\end{align*}
Taking the inner product with $\partial^{\alpha}f $ yields
\begin{align*}
\frac{1}{2}\frac{d}{dt}\| \partial^{\alpha}f\|_{L^2_{x,p}}^2+ \langle L\partial^{\alpha}f, \partial^{\alpha}f\rangle_{L^2_{x,p}}&= \langle \partial^{\alpha}f, \partial^{\alpha}\Gamma(f) \rangle_{L^2_{x,p}}+\langle \partial^{\alpha}f, \partial^{\alpha}T(f) \rangle_{L^2_{x,p}}.
\end{align*}
Applying the full coercivity estimate \eqref{full coer}, we have
\begin{align*}
\frac{1}{2}\frac{d}{dt}\| \partial^{\alpha}f\|_{L^2_{x,p}}^2+ \delta \| \partial^{\alpha}f\|_{x,\nu}^2&\leq  \langle \partial^{\alpha}f, \partial^{\alpha}\Gamma(f) \rangle_{L^2_{x,p}}+\langle \partial^{\alpha}f, \partial^{\alpha}T(f) \rangle_{L^2_{x,p}}.
\end{align*}
The estimate of the second-order nonlinear terms in Lemma \ref{nonlin ff} gives 
\begin{align*}
\langle \partial^{\alpha}f, \partial^{\alpha}\Gamma(f) \rangle_{L^2_{x,p}} \leq C \sum_{|\alpha_1|+|\alpha_2| \leq N} \int_{\mathbb{T}^3}dx ~\left(\|\partial^{\alpha_1}f\|_{L^2_p}\|\partial^{\alpha_2}f\|_{\nu}+\|\partial^{\alpha_1}f\|_{\nu}\|\partial^{\alpha_2}f\|_{L^2_p}\right)\|\partial^{\alpha}f\|_{\nu}.
\end{align*}
Without loss of generality, we assume that $\alpha_1$ is less than or equal to $\alpha_2$. Since $N\geq 3$, we have $|\alpha_1|+2 \leq N $. Thus the Sobolev embedding $H^2(\mathbb{T}^3)\subset\subset L^{\infty}(\mathbb{T}^3)$ implies 
\begin{align*}
\sum_{|\alpha| \leq N}\langle \partial^{\alpha}f, \partial^{\alpha}\Gamma(f) \rangle_{L^2_{x,p}} &\leq C \sum_{|\alpha| \leq N}\|\partial^{\alpha}f\|_{L^2_{x,p}}\sum_{|\alpha| \leq N}\|\partial^{\alpha}f\|_{x,\nu}\sum_{|\alpha| \leq N}\|\partial^{\alpha}f\|_{x,\nu}\cr
&\leq C\sqrt{\mathcal{E}(t)}\sum_{|\alpha| \leq N} \|\partial^{\alpha}f \|_{x,\nu}^2.
\end{align*}
For the estimate of the third-order nonlinear terms, we apply Lemma \ref{nonlin fff} to have
\begin{multline*}
\langle \partial^{\alpha}f, \partial^{\alpha}T(f) \rangle_{L^2_{x,p}} \leq  C \sum_{|\alpha_1|+|\alpha_2|+|\alpha_3| \leq N} \int_{\mathbb{T}^3}dx ~ \big( \|\partial^{\alpha_1}f\|_{L^2_p}\|\partial^{\alpha_2}f\|_{L^2_p}\|\partial^{\alpha_3}f\|_{\nu} \cr
+\|\partial^{\alpha_1}f\|_{L^2_p}\|\partial^{\alpha_2}f\|_{\nu}\|\partial^{\alpha_3}f\|_{L^2_p}  +\|\partial^{\alpha_1}f\|_{\nu }\|\partial^{\alpha_2}f\|_{L^2_p}\|\partial^{\alpha_3}f\|_{L^2_p} \big)\|\partial^{\alpha}f\|_{\nu}.
\end{multline*}
Similarly we assume that $\alpha_1$ and $\alpha_2$ are less than or equal to $\alpha_3$. 
Since $N\geq 3$, we have $|\alpha_1|+2\leq N$ and $|\alpha_2|+2 \leq N $. By the Sobolev embedding $H^2(\mathbb{T}^3)\subset\subset L^{\infty}(\mathbb{T}^3)$, we have
\begin{align*}
\sum_{|\alpha| \leq N}\langle \partial^{\alpha}f, \partial^{\alpha}T(f) \rangle_{L^2_{x,p}} &\leq C \sum_{|\alpha| \leq N}\|\partial^{\alpha}f\|_{L^2_{x,p}}\sum_{|\alpha| \leq N}\|\partial^{\alpha}f\|_{L^2_{x,p}}\sum_{|\alpha| \leq N}\|\partial^{\alpha}f\|_{x,\nu}\sum_{|\alpha| \leq N}\|\partial^{\alpha}f\|_{x,\nu}\cr
&\leq C\mathcal{E}(t)\sum_{|\alpha| \leq N} \|\partial^{\alpha}f \|_{x,\nu}^2.
\end{align*}
Combining these estimates, we conclude that 
\begin{align*}
\frac{1}{2}\frac{d}{dt}\| \partial^{\alpha}f\|_{L^2_{x,p}}^2 + \delta \sum_{|\alpha| \leq N} \|\partial^{\alpha}f \|_{x,\nu}^2 &\leq  C \left(\sqrt{\mathcal{E}(t)}+\mathcal{E}(t)\right) \sum_{|\alpha| \leq N} \|\partial^{\alpha}f \|_{x,\nu}^2.
\end{align*}
The remaining proof can be established by the standard continuity argument as in \cite{MR2000470,MR2095473}. This completes the proof.

\noindent {\bf Acknowledgement:}
 J. W. Jang is supported by CRC 1060 \textit{The mathematics of emergent effects} at the
University of Bonn funded through the German Science Foundation (DFG).

 S.-B. Yun is supported by Samsung Science and Technology Foundation under Project Number SSTF-BA1801-02.


\begin{thebibliography}{10}
\bibitem{akama1970relativistic} Akama, H.: Relativistic Boltzmann equation for plasmas. Journal of the Physical Society of Japan, {\bf28} (1970), no. 2, 478-488.

\bibitem{MR1857879} Alexandre, R., Villani, C.: On the Boltzmann equation for long-range interactions. Comm. Pure Appl. Math. {\bf55} (2002), no. 1, 30-70.


\bibitem{MR2102321} Andr\'{e}asson, H., Calogero, S., Illner, R.: On blowup for gain-term-only classical and relativistic Boltzmann equations. 
Math. Methods Appl. Sci. {\bf27} (2004), no. 18, 2231-2240.

\bibitem{MR2997586} Arkeryd, L., Esposito, R., Marra, R. and Nouri, A.: Exponential stability of the solutions to the Boltzmann equation for the Benard problem. 
Kinet. Relat. Models {\bf5} (2012), no. 4, 673-695.

\bibitem{MR4096124} Bae, G.-C., Yun, S.-B.: Quantum BGK model near a global Fermi-Dirac distribution. SIAM J. Math. Anal. {\bf52} (2020), no. 3, 2313-2352.

\bibitem{MR2301288} Benedetto, D., Castella, F., Esposito, R. and Pulvirenti, M.:
A short review on the derivation of the nonlinear quantum Boltzmann equations. 
Commun. Math. Sci. {\bf5} (2007), 55-71.


\bibitem{MR2357423} Benedetto, D., Castella, F., Esposito, R. and Pulvirenti, M.:
From the N-body Schr\"{o}dinger equation to the quantum Boltzmann equation: a term-by-term convergence result in the weak coupling regime. 
Comm. Math. Phys. {\bf277} (2008), no. 1, 1-44. 


\bibitem{MR2534787} Bobylev, A. V., Cercignani, C. and Gamba, I. M.:
On the self-similar asymptotics for generalized nonlinear kinetic Maxwell models.
Comm. Math. Phys. {\bf291} (2009), no. 3, 599-644.

\bibitem{MR3040372} Boblylev, A. V., Pulvirenti, M. and Saffirio, C.:
From particle systems to the Landau equation: a consistency result. 
Comm. Math. Phys. {\bf319} (2013), no. 3, 683-702.


\bibitem{MR3493188} Briant, M., Einav, A.:
On the Cauchy problem for the homogeneous Boltzmann-Nordheim equation for bosons: local existence, uniqueness and creation of moments. 
J. Statist. Phys. {\bf163} (2016), no. 5, 1108-1156.


\bibitem{buss2012transport} Buss, O., Gaitanos, T., Gallmeister, K., Van Hees, H., Kaskulov, M., Lalakulich, O., Larionov, A. B., Leitner, T., Weil, J. and Mosel, U.: Transport-theoretical description of nuclear reactions. Physics Reports, {\bf512} (2012), no. 1-2, 1-124.


\bibitem{MR2227952} Cercignani, C.:
Slow rarefied flows.
Theory and application to micro-electro-mechanical systems. Progress in Mathematical Physics, 41. Birkh\"{a}user Verlag, Basel, 2006.

\bibitem{MR1313028} Cercignani, C.:
The Boltzmann equation and its applications.
Applied Mathematical Sciences, 67. Springer-Verlag, New York, 1988. 

\bibitem{MR1307620} Cercignani, C., Illner, R. and Pulvirenti, M.:
The mathematical theory of dilute gases. 
Applied Mathematical Sciences, 106. Springer-Verlag, New York, 1994.


\bibitem{MR1898707} Cercignani, C., Kremer, G. M.:
The relativistic Boltzmann equation: theory and applications. 
Progress in Mathematical Physics, 22. Birkh\"{a}user Verlag, Basel, 2002.

\bibitem{2006.02540} Chapman, J., Jang, J. W., and Strain, R. M.: On the Determinant Problem for the Relativistic Boltzmann Equation. arXiv preprint arXiv:2006.02540. (2020).

\bibitem{MR0258399} Chapman, S., Cowling, T. G.:
The mathematical theory of non-uniform gases. An account of the kinetic theory of viscosity, thermal conduction and diffusion in gases.
Third edition, prepared in co-operation with D. Burnett Cambridge University Press, London 1970.

\bibitem{MR635279} De Groot, S. R., Van Leeuwen, W. A., Van Weert, C. G.: Relativistic kinetic theory.
Principles and Applications. North-Holland Publishing Co., Amsterdam-New York, 1980. 

\bibitem{MR2525118} Desvillettes, L., Mouhot, C.:
Stability and uniqueness for the spatially homogeneous Boltzmann equation with long-range interactions.
Arch. Ration. Mech. Anal. {\bf193} (2009), no. 2, 227-253.

\bibitem{MR1014927} DiPerna, R. J., Lions, P.-L.:
On the Cauchy problem for Boltzmann equations: global existence and weak stability.
Ann. of Math. {\bf130} (1989), no. 2, 321-366.

\bibitem{duan2017relativistic} Duan, R., Yu, H.:
The relativistic Boltzmann equation for soft potentials. 
Adv. Math. {\bf312} (2017), 315-373.

\bibitem{MR1151987} Dudy\'{n}ski, M., Ekiel-Je\.{z}ewska, M. L.:
Global existence proof for relativistic Boltzmann equation.
J. Statist. Phys. {\bf66} (1992), no. 3-4, 991-1001.


\bibitem{MR933458} Dudy\'{n}ski, M., Ekiel-Je\.{z}ewska, M. L.:
On the linearized relativistic Boltzmann equation. I. Existence of solutions.
Comm. Math. Phys. {\bf115} (1988), no. 4, 607-629.


\bibitem{dudynski2007relativistic} Dudy\'{n}ski, M., Ekiel-Je\.{z}ewska, M. L.: The relativistic Boltzmann equation-mathematical and physical aspects. J. Tech. Phys, {\bf48} (2007), no. 1, 39-47.


\bibitem{MR2145021} Escobedo, M., Mischler, S. and Valle, M. A.:
Entropy maximisation problem for quantum relativistic particles.
Bull. Soc. Math. France {\bf133} (2005), no. 1, 87-120.


\bibitem{MR1958975} Escobedo, M., Mischler, S. and Valle, M. A.:
Homogeneous Boltzmann equation in quantum relativistic kinetic theory. 
Electronic Journal of Differential Equations. Monograph, 4. Southwest Texas State University, San Marcos, TX, 2003.

\bibitem{MR3215584} Escobedo, M., Vel\'{a}zquez, J. J. L.:
On the blow up and condensation of supercritical solutions of the Nordheim equation for bosons. 
Comm. Math. Phys. {\bf330} (2014), no. 1, 331-365.


\bibitem{MR3157048} Gallagher, I., Saint-Raymond, L. and Texier, B.:
From Newton to Boltzmann: hard spheres and short-range potentials.
Zurich Lectures in Advanced Mathematics. European Mathematical Society (EMS), Z\"{u}rich, 2013.


\bibitem{MR1379589} Glassey, R. T.:
The Cauchy problem in kinetic theory. Society for Industrial and Applied Mathematics (SIAM), Philadelphia, PA, 1996.


\bibitem{MR1211782} Glassey, R. T., Strauss, W. A.:
Asymptotic stability of the relativistic Maxwellian. 
Publ. Res. Inst. Math. Sci. {\bf29} (1993), no. 2, 301-347.

\bibitem{MR1321370} Glassey, R. T., Strauss, W. A.:
Asymptotic stability of the relativistic Maxwellian via fourteen moments. 
Transport Theory Statist. Phys. {\bf24} (1995), no. 4-5, 657-678.

\bibitem{MR1105532} Glassey, R. T., Strauss, W. A.:
On the derivatives of the collision map of relativistic particles.
Transport Theory Statist. Phys. {\bf20} (1991), no. 1, 55-68.

\bibitem{MR33674} Grad, H.:
On the kinetic theory of rarefied gases.
Comm. Pure Appl. Math. {\bf2} (1949), 331-407.

\bibitem{MR0135535} Grad, H.:
Principles of the kinetic theory of gases. Handbuch der Physik (herausgegeben von S. Fl\"{u}gge), Bd. 12, Thermodynamik der Gase pp. 205-294 Springer-Verlag, Berlin-G\"{o}ttingen-Heidelberg, 1958.

\bibitem{MR3307944} Gradshteyn, I. S., Ryzhik, I. M.:
Table of integrals, series, and products.
Translated from the Russian. Translation edited and with a preface by Daniel Zwillinger and Victor Moll. Eighth edition. Revised from the seventh edition. Elsevier/Academic Press, Amsterdam, 2015.

\bibitem{MR2784329} Gressman, P., Strain, R. M.:
Global classical solutions of the Boltzmann equation without angular cut-off.
J. Amer. Math. Soc. {\bf24} (2011), no. 3, 771-847.

\bibitem{MR2095473} Guo, Y.:
The Boltzmann equation in the whole space. 
Indiana Univ. Math. J. {\bf53} (2004), no. 4, 1081-1094.

\bibitem{MR2000470} Guo, Y.:
The Vlasov-Maxwell-Boltzmann system near Maxwellians. 
Invent. Math. {\bf153} (2003), no. 3, 593-630.

\bibitem{MR1908664} Guo, Y.:
The Vlasov-Poisson-Boltzmann system near Maxwellians.
Comm. Pure Appl. Math. {\bf55} (2002), no. 9, 1104-1135.

\bibitem{MR2891870} Guo, Y., Strain, R. M.:
Momentum regularity and stability of the relativistic Vlasov-Maxwell-Boltzmann system. 
Comm. Math. Phys. {\bf310} (2012), no. 3, 649-673.

\bibitem{harris2004introduction} Harris, S.: An introduction to the theory of the Boltzmann equation. Courier Corporation. 2004.

\bibitem{MR2982812} Ha, S.-Y., Jeong, E. and Strain, R. M.:
Uniform $L^1$-stability of the relativistic Boltzmann equation near vacuum. 
Commun. Pure Appl. Anal. {\bf12} (2013), no. 2, 1141-1161.

\bibitem{huang1987statistical} Huang, K.: Statistical Mechanics, John Wiley and Sons. New York. 1963.

\bibitem{Jang2016} Jang, J. W.: Global classical solutions to the relativistic Boltzmann equation without angular cut-off, Ph.D. thesis, University of Pennsylvania, 2016.

\bibitem{jang2019propagation} Jang, J. W., Strain, R. M. and Yun, S.-B.: Propagation of uniform upper bounds for the spatially homogeneous relativistic Boltzmann equation, arXiv preprint arXiv:1907.05784. (2019).

\bibitem{MR3880739} Jang, J. W., Yun, S.-B.:
Gain of regularity for the relativistic collision operator. 
Appl. Math. Lett. {\bf90} (2019), 162-169.

\bibitem{Jang-Yun-Lp} Jang, J. W., Yun, S.-B.: Propagation of $L^p$ estimates for the Spatially Homogeneous Relativistic Boltzmann Equation. arXiv preprint arXiv:2001.11672. (2020).

\bibitem{kikuchi1930kinetische} Kikuchi, S., Nordheim, L.: \"{U}ber die kinetische Fundamentalgleichung in der Quantenstatistik. Zeitschrift f\"{u}r Physik A Hadrons and nuclei, {\bf60} (1930), no. 9-10, 652-662. 

\bibitem{kim2016introduction} Kim, M., Lee, C., Kim, Y. and Jeon, S.: Introduction to the DaeJeon Boltzmann-Uehling-Uhlenbeck (DJBUU)
Project. New Physics, {\bf66} (2016), no. 12, 1563-1570.

\bibitem{MR0479206} Lanford, O. E.:
Time evolution of large classical systems. Dynamical systems, theory and applications (Rencontres, Battelle Res. Inst., Seattle, Wash., 1974), pp. 1-111. Lecture Notes in Phys., Vol. 38, Springer, Berlin, 1975.

\bibitem{MR3389279} Lapitski, D.:
Development of the Quantum Lattice Boltzmann method for simulation of quantum electrodynamics with applications to graphene.
Thesis (D.Phil.)–University of Oxford (United Kingdom). 2014.

\bibitem{MR1773932} Lemou, M.: Linearized quantum and relativistic Fokker-Planck-Landau equations. Math. Methods Appl. Sci. {\bf23} (2000), no. 12, 1093-1119.

\bibitem{li2008recent} Li, B. A., Chen, L. W. and Ko, C. M.: Recent progress and new challenges in isospin physics with heavy-ion reactions. Physics Reports, {\bf464} (2008), no. 4-6, 113-281.

\bibitem{li1997equation} Li, B. A., Ko, C. M., and Ren, Z.: Equation of state of asymmetric nuclear matter and collisions of neutron-rich nuclei. Physical Review Letters, {\bf78} (1997), no. 9, 1644.


\bibitem{MR3906275} Li, W., Lu, X.: Global existence of solutions of the Boltzmann equation for Bose-Einstein particles with anisotropic initial data. 
J. Funct. Anal. {\bf276} (2019), no. 1, 231-283.

\bibitem{MR4796} Lichnerowicz, A., Marrot, R.: Propri\'{e}t\'{e}s statistiques des ensembles de particules en relativit\'{e} restreinte.
C. R. Acad. Sci. Paris {\bf210} (1940), 759-761.

\bibitem{MR2902121} Liu, S., Ma, X., Yu, H.: Optimal time decay of the quantum Landau equation in the whole space. J. Differential Equations. {\bf252} (2012), no. 10, 5414-5452.

\bibitem{MR1751703} Lu, X.: A modified Boltzmann equation for Bose-Einstein particles: isotropic solutions and long-time behavior. 
J. Statist. Phys. {\bf98} (2000), no. 5-6, 1335-1394.

\bibitem{MR3451497} Lu, X.:
Long time convergence of the Bose-Einstein condensation. 
J. Statist. Phys. {\bf162} (2016), no. 3, 652-670.


\bibitem{MR2096049} Lu, X.:
On isotropic distributional solutions to the Boltzmann equation for Bose-Einstein particles. 
J. Statist. Phys. {\bf116} (2004), no. 5-6, 1597-1649.

\bibitem{MR1861208} Lu, X.:
On spatially homogeneous solutions of a modified Boltzmann equation for Fermi-Dirac particles. 
J. Statist. Phys. {\bf105} (2001), no. 1-2, 353-388.

\bibitem{MR2264618} Lu, X.:
On the Boltzmann equation for Fermi-Dirac particles with very soft potentials: averaging compactness of weak solutions. 
J. Statist. Phys. {\bf124} (2006), no. 2-4, 517-547.

\bibitem{MR2433484} Lu, X.:
On the Boltzmann equation for Fermi-Dirac particles with very soft potentials: global existence of weak solutions. 
J. Differential Equations. {\bf245} (2008), no. 7, 1705-1761.

\bibitem{MR3038680} Lu, X.:
The Boltzmann equation for Bose-Einstein particles: condensation in finite time. 
J. Statist. Phys. {\bf150} (2013), no. 6, 1138-1176.


\bibitem{MR3217534} Lu, X.:
The Boltzmann equation for Bose-Einstein particles: regularity and condensation. 
J. Statist. Phys. 156 (2014), no. 3, 493-545.



\bibitem{MR2157856} Lu, X.:
The Boltzmann equation for Bose-Einstein particles: velocity concentration and convergence to equilibrium. 
J. Statist. Phys. {\bf119} (2005), no. 5-6, 1027-1067.

\bibitem{MR2029003} Lu, X., Wennberg, B.:
On stability and strong convergence for the spatially homogeneous Boltzmann equation for Fermi-Dirac particles. 
Arch. Ration. Mech. Anal. {\bf168} (2003), no. 1, 1-34.

\bibitem{MR2811470} Lu, X., Zhang, X.:
On the Boltzmann equation for 2D Bose-Einstein particles. 
J. Statist. Phys. {\bf143} (2011), no. 5, 990-1019.

\bibitem{MR2683475} Saint-Raymond, L.:
Hydrodynamic limits of the Boltzmann equation.
Lecture Notes in Mathematics, 1971. Springer-Verlag, Berlin, 2009.

\bibitem{spohn2012large} Spohn, H.: Large scale dynamics of interacting particles. Springer Science and Business Media. 2012.

\bibitem{MR2728733} Strain, R. M.:
Asymptotic stability of the relativistic Boltzmann equation for the soft potentials. 
Comm. Math. Phys. {\bf300} (2010), no. 2, 529-597.

\bibitem{MR2765751} Strain, R. M.:
Coordinates in the relativistic Boltzmann theory. 
Kinet. Relat. Models {\bf4} (2011), no. 1, 345-359.

\bibitem{MR2679588} Strain, R. M.:
Global Newtonian limit for the relativistic Boltzmann equation near vacuum. 
SIAM J. Math. Anal. {\bf42} (2010), no. 4, 1568-1601.


\bibitem{MR2366140} Strain, R. M., Guo, Y.:
Exponential decay for soft potentials near Maxwellian. 
Arch. Ration. Mech. Anal. {\bf187} (2008), no. 2, 287-339.


\bibitem{MR2100057} Strain, R. M., Guo, Y.:
Stability of the relativistic Maxwellian in a collisional plasma. 
Comm. Math. Phys. {\bf251} (2004), no. 2, 263-320.



\bibitem{MR3166961} Strain, R. M., Yun, S.-B.:
Spatially homogeneous Boltzmann equation for relativistic particles. 
SIAM J. Math. Anal. {\bf46} (2014), no. 1, 917-938.

\bibitem{MR2911100} Strain, R. M., Zhu, K.:
Large-time decay of the soft potential relativistic Boltzmann equation in $\mathbb{R}^3_x$. 
Kinet. Relat. Models. {\bf5} (2012), no. 2, 383-415.

\bibitem{succi2002lattice} Succi, S.: Lattice Boltzmann equation for relativistic quantum mechanics. Philosophical Transactions of the Royal Society of London. Series A: Mathematical, Physical and Engineering Sciences, {\bf360} (2002), no. 1792, 429-436. 

\bibitem{uehling1933transport} Uehling, E. A., Uhlenbeck, G. E.: Transport phenomena in einstein-bose and fermi-dirac gases. i. Physical Review, {\bf43} (1933), no. 7, 552.

\bibitem{van1982generalized} Van Weert, C. G.: Generalized hydrodynamics from relativistic kinetic theory. Physica A: Statistical Mechanics and its Applications, {\bf111} (1982), no. 3, 537-552.

\bibitem{MR1942465} Villani, C.:
A review of mathematical topics in collisional kinetic theory. Handbook of mathematical fluid dynamics, Vol. I, 71-305, North-Holland, Amsterdam, 2002. 


\bibitem{wang2018global} Wang, Y.:
Global well-posedness of the relativistic Boltzmann equation. 
SIAM J. Math. Anal. {\bf50} (2018), no. 5, 5637-5694.
	
\end{thebibliography}
\end{document}